\documentclass[11pt]{article}
\usepackage{amssymb}
\usepackage{amsmath}
\usepackage{mathrsfs}
\usepackage{graphics}
\usepackage{graphicx}
\usepackage{xcolor}
\usepackage{subfigure}
\usepackage[T1]{fontenc}
\usepackage{latexsym,amssymb,amsmath,amsfonts,amsthm}\usepackage{txfonts}
\newcommand{\be}{\begin{eqnarray}}
\newcommand{\ben}{\begin{eqnarray*}}
\newcommand{\en}{\end{eqnarray}}
\newcommand{\enn}{\end{eqnarray*}}

\newtheorem{theorem}{Theorem}[section]
\newtheorem{lemma}{Lemma}[section]
\newtheorem{prp}[theorem]{Proposition}
\newtheorem{thm}[theorem]{Theorem}

\newtheorem{dfn}{Definition}[section]

\newtheorem{remark}{Remark}

\definecolor{rr}{rgb}{0,0,0}
\begin{document}
\renewcommand{\theequation}{\arabic{section}.\arabic{equation}}
\begin{titlepage}
\title{\bf Large deviations for stochastic porous media equations}
\author{ Rangrang Zhang\\
{\small  School of  Mathematics and Statistics,
Beijing Institute of Technology, Beijing, 100081, China.}\\
({\sf rrzhang@amss.ac.cn})}
\date{}
\end{titlepage}
\maketitle

\noindent\textbf{Abstract}:
 In this paper, we establish the Freidlin-Wentzell type large deviation principles for porous medium-type equations perturbed by small multiplicative noise. The porous medium operator $\Delta (|u|^{m-1}u)$ is allowed. Our proof is based on weak convergence approach.



\section{Introduction}
In recent years, there has been a lot of interest in porous media equations both in deterministic and random case (see, e.g., \cite{BDaR,BDa,DaR,DGG,PR} and the references therein). In this paper, we are interested in the asymptotic behaviour of porous media equations with small multiplicative noise.
More precisely, fix any
$T>0$ and let $(\Omega,\mathcal{F},\mathbb{P},\{\mathcal{F}_t\}_{t\in
[0,T]},(\{\beta_k(t)\}_{t\in[0,T]})_{k\in\mathbb{N}})$ be a stochastic basis. Without loss of generality, here the filtration $\{\mathcal{F}_t\}_{t\in [0,T]}$ is assumed to be complete and $\{\beta_k(t)\}_{t\in[0,T]},k\in\mathbb{N}$, are one-dimensional real-valued i.i.d  $\{\mathcal{F}_t\}_{t\in [0,T]}-$Wiener processes. We use $\mathbb{E}$ to denote the expectation with respect to $\mathbb{P}$.
Fix any $N\in\mathbb{N}$, let $\mathbb{T}^N\subset\mathbb{R}^N$ denote the $N-$dimensional torus (suppose the periodic length is $1$).
We are concerned with the following porous media equations with stochastic forcing
\begin{eqnarray}\label{P-18}
\left\{
  \begin{array}{ll}
  du(t,x)=\Delta (|u(t,x)|^{m-1}u(t,x))dt+\Phi(u(t,x)) dW(t) \quad {\rm{in}}\ \mathbb{T}^N\times (0,T],\\
u(\cdot,0)=u_0(\cdot)\in L^{m+1}(\mathbb{T}^N) \quad {\rm{on}}\ \mathbb{T}^N,
  \end{array}
\right.
\end{eqnarray}
for $m\in (1, \infty)$. Here $u:(\omega,x,t)\in\Omega\times\mathbb{T}^N\times[0,T]\mapsto u(\omega,x,t):=u(x,t)\in\mathbb{R}$ is a random field,
that is, the equation is periodic in the space variable $x\in \mathbb{T}^N$,
 the coefficient $\Phi:\mathbb{R}\to\mathbb{R}$ is measurable and fulfills certain conditions specified later,
and $W$ is a cylindrical Wiener process defined on a given (separable) Hilbert space $U$ with
the form $W(t)=\sum_{k\geq 1}\beta_k(t) e_k,t\in[0,T]$, where $(e_k)_{k\geq 1}$ is a complete orthonormal base in the Hilbert space $U$. Clearly, equations (\ref{P-18}) can be viewed as a special case of a class of SPDE of the type
\begin{eqnarray}\label{P-19}
\left\{
  \begin{array}{ll}
  du(t,x)=\Delta A(u(t,x))dt+\Phi(u(t,x)) dW(t) \quad {\rm{in}}\ \mathbb{T}^N\times (0,T],\\
u(\cdot,0)=u_0(\cdot)\in L^{m+1}(\mathbb{T}^N) \quad {\rm{on}}\ \mathbb{T}^N.
  \end{array}
\right.
\end{eqnarray}

Having a stochastic forcing term in (\ref{P-19}) is very natural and important for various modeling problems arising in a wide variety of fields, e.g., physics, engineering, biology and so on. Up to now, the Cauchy problem for the stochastic equations (\ref{P-18}) has been studied by a lot of papers and different approaches, for example based on monotonicity in $H^{-1}$, based on entropy solutions and based on kinetic solutions have been developed. Specifically, in \cite{BDaR,P75,PR,RRW,RWW}, a monotone operator approach is employed in the space $H^{-1}$. When applied to the Nemytskii type diffusion coefficients, the condition could be verified if $\Phi$ are affine linear functions of $u$, otherwise, the map $u\rightarrow \Phi(u)$ are not known to be Lipschitz continuous in $H^{-1}$, even if $\Phi$ is smooth. In order to relax the assumptions on $\Phi$, alternative approaches based on $L^1-$techniques have been proposed.
In the deterministic setting, this has been realized via the theory of accretive operators going back to Crandall-Ligget \cite{CL}, entropy solutions studied by Otto \cite{O96}, Kruzkov \cite{K70} and kinetic solution by Lions et al. \cite{L-P-T} and Chen, Perthame \cite{CP}. In the stochastic setting, an entropy solution was first introduced by Kim in \cite{K} when studying the conservation laws driven by additive noise wherein the author proposed a method of compensated compactness to prove the existence of a stochastic weak entropy solution via vanishing viscosity approximation. Moreover, a Kruzkov-type method was used there to prove the uniqueness. Later, Vallet and Wittbold \cite{V-W} extended the results of Kim to the multi-dimensional Dirichlet problem with additive noise.  Concerning the case of the
equation with multiplicative noise, for Cauchy problem over the whole spatial space, Feng and Nualart \cite{F-N} introduced a notion of strong entropy solutions in order to prove the uniqueness of the entropy solution.  On the other hand, using a kinetic formulation, Debussche and Vovelle \cite{D-V-1} solved the Cauchy problem for stochastic conservation laws in any dimension by making use of a notion of kinetic solutions. In view of the equivalence between kinetic formulation and entropy solution, they obtained the existence and uniqueness of the entropy solutions.
The literature concerning the entropy and kinetic solutions to stochastic degenerate parabolic equations (\ref{P-19}) is quite extensive, let us mention some works. For instance, Bauzet et al. \cite{BVW} studied the degenerate parabolic-hyperbolic Cauchy problem under the assumptions that $A$ is globally Lipschitz and when $\Phi$ is Lipschitz, a behavior $A(u)=|u|^{m-1}u$ near the origin is allowed only for $m>2$. Moreover, by using a kinetic formulation, Gess and Hofmanov\'{a} \cite{GH} showed the global well-posedness of stochastic porous media equations, where the boundedness of $A'$ is released, $\Phi$ is assumed to be Lipschitz and $\sqrt{A'(u)}$ is $\gamma-$H\"{o}lder continuous with $\gamma>\frac{1}{2}$ which forces $m>2$. Recently, based on a notion of entropy solutions, Dareiotis et al. \cite{DGG} established the well-posedness of (\ref{P-18}) in the full range $m\in(1,\infty)$ under mild assumptions on the Nemytskii type diffusion coefficient $\Phi$, where the authors proved an $L^1-$contraction estimates as well as a generalized $L^1-$stability estimates.
There are also a lot of interest on the stochastic fast diffusion equation, that is, $m\in (0,1]$ (see \cite{BR,RRW}). To learn invariant measures for the stochastic porous media equations, we can refer readers to \cite{BDa,DaR}.

\

 From statistical mechanics point of view, asymptotic analysis for vanishing the noise force is
important and interesting for studying stochastic porous media, in which establishing large deviation principles is a core step for finer analysis as well as gaining deeper insight for the described physical evolutions. There are several works on large deviation principles (LDP) for the stochastic porous media equations, we mention some of them. {\color{rr}R\"{o}ckner} et al. \cite{RWW} established the LDP for a class of generalized stochastic porous media equations for both small noise and short time in the space $C([0,T];H^{-1})$ by utilizing the monotonicity of the porous medium operator in $H^{-1}$. Later, Liu \cite{L10}  established LDP for the distributions
of stochastic evolution equations with general monotone drift and small
multiplicative noise. As application, the author proved the LDP holds for stochastic porous media equations in the space $C([0,T];H^{-1})$.
The purpose of this paper is to prove that the kinetic solution to the stochastic porous medium-type equations (\ref{P-19}) satisfies Freidlin-Wentzell type LDP in the space $L^1([0,T];L^1(\mathbb{T}^N))$, which is a delicate result compared with \cite{RWW} and \cite{L10}. On the other hand, Dong et al. \cite{DZZ} established the LDP for quasilinear parabolic SPDE, where the authors handled the hard term $div(B(u)\nabla u)$ with $B$ being uniformly positive definite, bounded and Lipschitz. For our model, it holds that $\Delta A(u)=div(a^2(u)\nabla u)$ with $a^2(r)=A'(r)$, so it has similar structure as the term $div(B(u)\nabla u)$, but $a^2(r)$ is neither bounded nor Lipschitz. Thus, our case is much more complex and difficult than \cite{DZZ}.

To study the Freidlin-Wentzell's LDP for SPDE, an important tool is the weak convergence approach, which is developed by Dupuis and Ellis in \cite{DE}. The key idea of this approach is to prove certain variational representation formula about the Laplace transform of bounded continuous functionals, which then leads to the verification of the equivalence between the LDP and the Laplace principle. In particular, for Brownian functionals, an elegant variational representation formula has been established by Bou\'{e}, Dupuis in \cite{MP} and by Budhiraja, Dupuis in \cite{BD}. Recently, a sufficient condition to verify the large deviation criteria of Budhiraja et al. \cite{BDM} for functionals of Brownian motions is proposed by Matoussi et al. in \cite{MSZ}, which turns out to be more suitable for SPDEs arising from fluid mechanics. Thus, in the present paper, we adopt this new sufficient condition.

To our knowledge, the present paper is the first work towards establishing the LDP directly for the kinetic solution to the stochastic porous medium-type equations (\ref{P-19}). The starting point for our research was the paper of Dareiotis et al. \cite{DGG}, where the global well-posedness of entropy solution to (\ref{P-19}) was established. {\color{rr}According to the equivalence between entropy solution and kinetic solution (see Proposition \ref{prp-6} in the below)}, we firstly deduce the existence and uniqueness of kinetic solution to (\ref{P-19}). Due to the fact that the kinetic solutions are living in a rather irregular space comparing to various type solutions for parabolic SPDEs, it is indeed a challenge to establish LDP for the stochastic porous media equations with general noise force.
In order to prove the LDP holds for the kinetic solution in the space $L^1([0,T];L^1(\mathbb{T}^N))$, our proof strategy mainly consists of the following procedures. As an important part of the proof, we need to obtain the global well-posedness of the associated skeleton equations. For showing the uniqueness, we establish a general result concerning the stability of the strong solution map on the coefficients by utilizing the doubling of variables method. For showing the existence result, we adopt the similar approach as \cite{DGG}.
To complete the proof of the large deviation principle, we also need to study the weak convergence of the small noise perturbations of the problem (\ref{P-19}) in the random directions of the Cameron-Martin space of the driving Brownian motions. To verify the convergence of the randomly perturbed equation to the corresponding unperturbed equation in $L^1([0,T];L^1(\mathbb{T}^N))$, an auxiliary approximating process is introduced and the doubling of variables method is employed.

\

The rest of the paper is organised as follows. The mathematical formulation of stochastic porous media equations is presented in Section 2. In Section 3, we introduce the weak convergence method and state our main result. Section 4 is devoted to the study of the associated skeleton equations. The large deviation principle is proved in Section 5.

\section{Preliminaries}
Let us first introduce the notations which will be used later on.
 $C_b$ represents the space of bounded, continuous functions and $C^1_b$ stands for the space of bounded, continuously differentiable functions having bounded first order derivative. Let $\|\cdot\|_{L^p(\mathbb{T}^N)}$ denote the norm of Lebesgue space $L^p(\mathbb{T}^N)$ for $p\in (0,\infty]$. In particular, set $H=L^2(\mathbb{T}^N)$ with the corresponding norm $\|\cdot\|_H$. For all $a\geq0$, let
$H^a(\mathbb{T}^N)=W^{a,2}(\mathbb{T}^N)$ be the usual Sobolev space of order $a$ with the norm
\[
\|u\|^2_{H^a(\mathbb{T}^N)}=\sum_{|\alpha|=|(\alpha_1,...,\alpha_N)|=\alpha_1+\cdots+\alpha_N\leq a}\int_{\mathbb{T}^N}|D^{\alpha}u(x)|^2dx.
\]
 $H^{-a}(\mathbb{T}^N)$  stands for the topological dual of $H^a(\mathbb{T}^N)$, whose norm is denoted by $\|\cdot\|_{H^{-a}(\mathbb{T}^N)}$. Moreover, we use the brackets $\langle\cdot,\cdot\rangle$ to denote the duality between $C^{\infty}_c(\mathbb{T}^N\times \mathbb{R})$ and the space of distributions over $\mathbb{T}^N\times \mathbb{R}$.
Similarly, for $1\leq p\leq \infty$ and $q:=\frac{p}{p-1}$, the conjugate exponent of $p$, we denote
\[
\langle F, G \rangle:=\int_{\mathbb{T}^N}\int_{\mathbb{R}}F(x,\xi)G(x,\xi)dxd\xi, \quad F\in L^p(\mathbb{T}^N\times \mathbb{R}), G\in L^q(\mathbb{T}^N\times \mathbb{R}),
\]
and also for a measure $m$ on the Borel measurable space $\mathbb{T}^N\times[0,T]\times \mathbb{R}$
\[
m(\phi):=\langle m, \phi \rangle:=\int_{\mathbb{T}^N\times[0,T]\times \mathbb{R}}\phi(x,t,\xi)dm(x,t,\xi), \quad  \phi\in C_b(\mathbb{T}^N\times[0,T]\times \mathbb{R}).
\]
\subsection{Hypotheses}
 Set
\begin{eqnarray}\label{e-1}
a(r)=\sqrt{A'(r)},\quad \Psi(r)=\int^r_0a(s)ds.
\end{eqnarray}
Following \cite{DGG}, we impose conditions on the nonlinearity $A$ via assumptions on $\Psi$, with some constants $m>1$, $K\geq 1$, which are fixed throughout the whole paper.
Precisely, we assume
\begin{description}
  \item[\textbf{Hypothesis H}] The initial value $u_0$ satisfies $\|u_0\|^{m+1}_{L^{m+1}(\mathbb{T}^N)}<\infty$. The function $A: \mathbb{R}\rightarrow \mathbb{R}$ is differentiable, strictly increasing and odd. The function $a$ is differentiable away from the origin and satisfies the bounds
 \begin{eqnarray}
 |a(0)|\leq K, \quad |a'(r)|\leq K|r|^{\frac{m-3}{2}},\quad  \rm{if}\  r>0,
  \end{eqnarray}
as well as
\begin{eqnarray}
 K a(r)\geq I_{|r|\geq 1}, \quad K|\Psi(r)-\Psi(s)|\geq \left\{
                                                          \begin{array}{ll}
                                                            |r-s|, & \ \rm{if}\ \ |r|\vee |s|\geq 1, \\
                                                            |r-s|^{\frac{m+1}{2}}, & \ \ \rm{if}\ |r|\vee |s|< 1.
                                                          \end{array}
                                                        \right.
  \end{eqnarray}
      For each $u\in \mathbb{R}$, the map $\Phi(u): U\rightarrow H$ is defined by $\Phi(u) e_k=g^k(\cdot, u)$, where each $g^k(\cdot,u)$ is a regular function on $\mathbb{T}^N$. Denote by $g=(g^1,g^2,\cdot\cdot\cdot)$.
      More precisely, we assume that $g: \mathbb{T}^N\times \mathbb{R}\rightarrow l_2$ satisfies the bounds,
\begin{eqnarray}\label{equ-28}
G(x,u)=|g(x,u)|_{l^2}:=\Big(\sum_{k\geq 1}|g^k(x,u)|^2\Big)^{\frac{1}{2}}&\leq& K(1+|u|),\\
\label{equ-29}
|g(x,u)-g(y,v)|_{l^2}:=\Big(\sum_{k\geq 1}|g^k(x,u)-g^k(y,v)|^2\Big)^{\frac{1}{2}}&\leq& K(|x-y|+|u-v|),
\end{eqnarray}
for $x, y\in \mathbb{T}^N, u,v\in \mathbb{R}$.
For $g=(g^1,g^2,\cdot\cdot\cdot), \tilde{g}=(\tilde{g}^1,\tilde{g}^2,\cdot\cdot\cdot)$ as above, we set
\begin{eqnarray}\label{e-14}
d(g, \tilde{g}):=\sup_{u\in \mathbb{R},x\in \mathbb{T}^N}\frac{\sum_{k\geq 1}|g^k(x,u)-\tilde{g}^k(x,u)|^2}{(1+|u|)^{m+1}}
\end{eqnarray}
\end{description}
\begin{remark}
In order to obtain the large deviations, our assumptions are stronger than those used by \cite{DGG} to prove the existence and uniqueness of (\ref{P-19}), that is,
the condition (\ref{equ-29}) on $g$ is a special case of Assumption 2.2 in \cite{DGG} with parameters $\kappa=\frac{1}{2}$ and $\bar{\kappa}=1$.
\end{remark}
Based on the above notations, equation (\ref{P-19}) can be rewritten as
\begin{eqnarray}\label{P-19-1}
\left\{
  \begin{array}{ll}
  du(t,x)=\Delta A(u(t,x))dt+\sum_{k\geq 1}g^k(x,u(t,x)) d\beta_k(t) \quad {\rm{in}}\ \mathbb{T}^N\times (0,T],\\
u(\cdot,0)=u_0(\cdot)\in L^{m+1}(\mathbb{T}^N) \quad {\rm{on}}\ \mathbb{T}^N.
  \end{array}
\right.
\end{eqnarray}
We denote by $\mathcal{E}(A,g,u_0)$ the Cauchy problem (\ref{P-19-1}).

\subsection{Entropy solution and global well-posedness}
Firstly, we recall the following entropy solution of (\ref{P-19-1}) introduced by \cite{DGG}.
Set
\begin{eqnarray}
\Psi_f(r):=\int^r_0f(s)a(s)ds,\quad \forall \ f\in C(\mathbb{R}).
\end{eqnarray}
Clearly, by (\ref{e-1}), it gives that $\Psi=\Psi_1$.
\begin{dfn}\label{dfn-4}(Entropy solution)
An entropy solution of $\mathcal{E}(A,g,u_0)$ is a predictable stochastic process $u: \Omega\times [0,T]\rightarrow L^{m+1}(\mathbb{T}^N)$ such that
\begin{description}
  \item[(i)] $u\in L^{m+1}(\Omega\times [0,T];L^{m+1}(\mathbb{T}^N))$
  \item[(ii)] For all $f\in C_b(\mathbb{R})$, we have $\Psi_f(u)\in L^2(\Omega\times [0,T];H^1(\mathbb{T}^N))$ and
\[
\partial_i \Psi_f(u)=f(u)\partial_i \Psi(u).
\]

  \item[(iii)] For all convex function $\eta\in C^2(\mathbb{R})$ with $\eta''$ compactly supported and all non-negative $\phi\in C^1_c( \mathbb{T}^N\times[0,T))$, we have
\begin{eqnarray}\notag
-\int^T_0\int_{\mathbb{T}^N}\eta(u)\partial_t \phi dxdt&\leq & \int_{\mathbb{T}^N}\eta(u_0)\phi(0)dx+\int^T_0\int_{\mathbb{T}^N}q_{\eta}(u)\Delta \phi dxdt\\ \notag
&& +\int^T_0\int_{\mathbb{T}^N}\Big(\frac{1}{2}\phi\eta''(u)G^2(x,u)-\phi\eta''(u)|\nabla \Psi(u)|^2\Big)dxdt\\
\label{e-2}
&& + \sum_{k\geq1}\int^T_0\int_{\mathbb{T}^N}\phi\eta'(u)g^k(x,u)dxd\beta_k(t), \quad a.s.,
\end{eqnarray}
where $q_{\eta}$ is any function satisfying $q'_{\eta}(\xi)=\eta'(\xi)a^2(\xi)$.
\end{description}

\end{dfn}
\begin{remark}
$(\eta, q_{\eta})$ is called entropy-entropy flux pair. If $\eta(r)=\pm r$, it follows from (\ref{e-2}) that any entropy solution satisfies (\ref{P-19-1}) is a weak solution of (\ref{P-19-1})  (weak in both space and time).
\end{remark}

Referring to Theorem 2.1 in \cite{DGG}, the following global well-posedness is proved.
\begin{thm}\label{th-1}
Let $(A, g, u_0)$ satisfy Hypotheses H. Then there exists a unique entropy solution to (\ref{P-19-1}) with initial condition $u_0$. Moreover, if $\tilde{u}$ is the unique entropy solution to (\ref{P-19-1}) with initial value $\tilde{u}_0$, then
\begin{eqnarray}\label{e-3}
\underset{0\leq t\leq T}{{\rm{ess\sup}}}\ \mathbb{E}\|u(t)-\tilde{u}(t)\|_{L^1(\mathbb{T}^N)}\leq \mathbb{E}\|u_0-\tilde{u}_0\|_{L^1(\mathbb{T}^N)}.
\end{eqnarray}

\end{thm}

The authors of the work \cite{DGG} also show the stability of the solution map with respect to the coefficients in the following sense.
\begin{thm}
Let $(A_n)_{n\in \mathbb{N}}, (g_n)_{n\in \mathbb{N}}$ satisfy Hypothesis H uniformly in $n$ and the initial values satisfy $\sup_{n}\|u_{0,n}\|_{L^{m+1}(\mathbb{T}^N)}< \infty$. Assume furthermore that $A_n\rightarrow A$ uniformly on compact sets of $\mathbb{R}$, $u_{0,n}\rightarrow u_0$ in $L^{m+1}(\mathbb{T}^N)$ and $d(g_n, g)\rightarrow 0$, as $n\rightarrow \infty$. Let $u_n, u$ be the entropy solutions of $\mathcal{E}(A_n, g_n, u_{0,n}), \mathcal{E}(A, g, u_{0})$, respectively. Then $u_n\rightarrow u$ in $L^1(\Omega\times [0,T]\times \mathbb{T}^N)$, as $n\rightarrow \infty$.
\end{thm}

\subsection{Kinetic solution and generalized kinetic solution}
In this subsection, we pay attention to the definition of kinetic solution. We introduce the kinetic solution to equations (\ref{P-19-1}) as follows. Keeping in mind that we are working on the stochastic basis $(\Omega,\mathcal{F},\mathbb{P},\{\mathcal{F}_t\}_{t\in [0,T]},(\beta_k(t))_{k\in\mathbb{N}})$.
\begin{dfn}(Kinetic measure)\label{dfn-3}
 A map $m$ from $\Omega$ to the set of non-negative, finite measures over $\mathbb{T}^N\times [0,T]\times\mathbb{R}$ is said to be a kinetic measure, if
\begin{description}
  \item[1.] $ m $ is measurable, that is, for each $\phi\in C_b(\mathbb{T}^N\times [0,T]\times \mathbb{R}), \langle m, \phi \rangle: \Omega\rightarrow \mathbb{R}$ is measurable,
  \item[2.] $m$ vanishes for large $\xi$, i.e.,
\begin{eqnarray*}\label{equ-37}
\lim_{R\rightarrow +\infty}\mathbb{E}[m(\mathbb{T}^N\times [0,T]\times B^c_R)]=0,
\end{eqnarray*}
where $B^c_R:=\{\xi\in \mathbb{R}, |\xi|\geq R\}$
  \item[3.] for every $\phi\in C_b(\mathbb{T}^N\times \mathbb{R})$, the process
\[
(\omega,t)\in\Omega\times[0,T]\mapsto \int_{\mathbb{T}^N\times [0,t]\times \mathbb{R}}\phi(x,\xi)dm(x,s,\xi)\in\mathbb{R}
\]
is predictable.
\end{description}
\end{dfn}
Let $\mathcal{M}^+_0(\mathbb{T}^N\times [0,T]\times \mathbb{R})$ be the space of all bounded, nonnegative random measures $m$ satisfying (\ref{equ-37}).

\begin{dfn}(Kinetic solution)\label{dfn-1}
Let $u_0\in L^{m+1}(\mathbb{T}^N)$. A measurable function $u: \mathbb{T}^N\times [0,T]\times\Omega\rightarrow \mathbb{R}$ is called a kinetic solution to (\ref{P-19-1}) with initial datum $u_0$, if
\begin{description}
  \item[1.] $(u(t))_{t\in[0,T]}$ is predictable,
  \item[2.] there exists $C_m>0$ such that
\[
\mathbb{E}\left(\int^T_0\|u(t)\|^{m+1}_{L^{m+1}(\mathbb{T}^N)}dt\right)\leq C_m,
\]
\item[3.] there exists a kinetic measure $m$ such that $f:= I_{u>\xi}$ satisfies the following
\begin{eqnarray*}\notag
&&\int^T_0\langle f(t), \partial_t \varphi(t)\rangle dt+\langle f_0, \varphi(0)\rangle +\int^T_0\langle f(t), a^2(\xi) \Delta \varphi (t)\rangle dt\\ \notag
&=& \int^T_0\int_{\mathbb{T}^N}\partial_{\xi}\varphi (x,t,u(x,t))|\nabla \Psi(u)|^2dxdt\\ \notag
&& -\sum_{k\geq 1}\int^T_0\int_{\mathbb{T}^N} g^k(x, u(x,t))\varphi (x,t,u(x,t))dxd\beta_k(t) \\
\label{P-21}
&& -\frac{1}{2}\sum_{k\geq 1}\int^T_0\int_{\mathbb{T}^N}\partial_{\xi}\varphi (x,t,u(x,t))G^2(x,u(t,x))dxdt+ m(\partial_{\xi} \varphi), \ a.s. ,
\end{eqnarray*}
for all $\varphi\in C^1_c(\mathbb{T}^N\times [0,T)\times \mathbb{R})$, where $u(t)=u(\cdot,t,\cdot)$, $G^2=\sum^{\infty}_{k=1}|g^k|^2$ and $a(\cdot)$ is defined by (\ref{e-1}).
\end{description}
\end{dfn}

In order to prove the existence of a kinetic solution, the generalized kinetic solution was introduced in \cite{D-V-1}.
\begin{dfn}(Young measure)
 Let $(X,\lambda)$ be a finite measure space. Let $\mathcal{P}_1(\mathbb{R})$ denote the set of all (Borel) probability measures on $\mathbb{R}$. A map $\nu:X\to\mathcal{P}_1(\mathbb{R})$ is
said to be a Young measure on $X$, if for each $\phi\in C_b(\mathbb{R})$, the map $z\in X\mapsto \nu_z(\phi)\in\mathbb{R}$ is measurable. Next, we say that a Young measure $\nu$ vanishes at infinity if, for each  $p\geq 1$, the following holds
\begin{eqnarray}\label{equ-26}
\int_X\int_{\mathbb{R}}|\xi|^pd\nu_z(\xi)d\lambda(z)<+\infty.
\end{eqnarray}

\end{dfn}
\begin{dfn}(Kinetic function)
Let $(X,\lambda)$ be a finite measure space. A measurable function $f: X\times \mathbb{R}\rightarrow [0,1]$
is called a kinetic function, if there exists a Young measure $\nu$ on $X$ that vanishes at infinity such that $\forall\xi\in \mathbb{R}$
\[
f(z,\xi)=\nu_z(\xi,+\infty)
\]
holds for $\lambda-a.e.$ $z\in X,$.
We say that $f$ is an equilibrium if there exists a measurable function $u: X\rightarrow \mathbb{R}$ such that $f(z,\xi)=I_{u(z)>\xi}$ a.e., or equivalently, $\nu_z=\delta_{u(z)}$ for $\lambda- a.e. \ z\in X$.
\end{dfn}

 Let $f: X\times \mathbb{R}\rightarrow [0,1]$ be a kinetic function, we use $\bar{f}$ to denote its conjugate function $\bar{f}:=1-f$. We also denote by $\Lambda_f$ the function defined by $\Lambda_f(z,\xi)=f(z,\xi)-I_{0>\xi}$. This correction to $f$ is integral on $\mathbb{R}$. In fact, the function $\Lambda_f$ is decreasing faster than any power of $|\xi|$ at infinity. Indeed, we have $\Lambda_f(z,\xi)=-\nu_{z}(-\infty,\xi)$ when $\xi<0$ and $\Lambda_f(z,\xi)=\nu_{z}(\xi,+\infty)$ when $\xi>0$. Therefore, by (\ref{equ-26}), it yields
\begin{eqnarray}\label{qq-6}
|\xi|^p\int_X|\Lambda_f(z,\xi)|d\lambda(z)\leq \int_X\int_{\mathbb{R}}|\xi|^pd\nu_z(\xi)d\lambda(z)<\infty
\end{eqnarray}
for all $\xi\in \mathbb{R}$ and $1\leq p<\infty$.

\begin{dfn}(Generalized kinetic solution)\label{dfn-2}
Let $f_0:\Omega\times\mathbb{T}^N\times \mathbb{R}\rightarrow [0,1]$ be a kinetic function with
$(X,\lambda)=(\Omega\times\mathbb{T}^N,\mathbb{P}\otimes dx)$. A measurable function $f:\Omega\times\mathbb{T}^N\times[0,T]\times\mathbb{R}\rightarrow[0,1]$ is said to be a generalized kinetic solution to (\ref{P-19-1}) with initial datum $f_0$, if
\begin{description}
  \item[1.] $(f(t))_{t\in[0,T]}$ is predictable,
 \item[2.] $f$ is a kinetic function with $(X,\lambda)=(\Omega\times\mathbb{T}^N\times[0,T],\mathbb{P}\otimes dx\otimes dt)$ and there exists a constant $C_m>0$ such that $\nu:=-\partial_{\xi} f$ fulfills the following
\begin{eqnarray*}
\mathbb{E}\left(esssup_{t\in [0,T]}\int_{\mathbb{T}^N}\int_{\mathbb{R}}|\xi|^{m+1}d\nu_{x,t}(\xi)dxdt\right)\leq C_m,
\end{eqnarray*}
\item[3.] there exists a kinetic measure $m$ such that for $\varphi\in C^1_c(\mathbb{T}^N\times [0,T)\times \mathbb{R})$,
\begin{eqnarray}\label{P-22}\notag
&&\int^T_0\langle f(t), \partial_t \varphi(t)\rangle dt+\langle f_0, \varphi(0)\rangle +\int^T_0\langle f(t), a^2(\xi)\Delta \varphi (t)\rangle dt\\ \notag
&=& \int^T_0\int_{\mathbb{T}^N}\int_{\mathbb{R}}\partial_{\xi}\varphi (x,t,\xi)|\nabla \Psi(\xi)|^2d\nu_{x,t}(\xi)dxdt\\ \notag
&& -\sum_{k\geq 1}\int^T_0\int_{\mathbb{T}^N}\int_{\mathbb{R}} g^k(x,\xi)\varphi (x,t,\xi)d\nu_{x,t}(\xi)dxd\beta_k(t) \\
&& -\frac{1}{2}\int^T_0\int_{\mathbb{T}^N}\int_{\mathbb{R}}\partial_{\xi}\varphi (x,t,\xi)G^2(x,\xi)d\nu_{x,t}(\xi)dxdt+ m(\partial_{\xi} \varphi),\  a.s..
\end{eqnarray}
\end{description}
\end{dfn}
Referring to \cite{D-V-1}, almost surely, any generalized solution admits possibly different left and right weak limits at any point $t\in[0,T]$. This property is important for establishing a comparison principle which allows to prove uniqueness. Also, it allows us to see that the weak form (\ref{P-22}) of the equation satisfied by a generalized kinetic solution can be strengthened. We write below a formulation which is weak only respect to $x$ and $\xi$.
The following result is proved in \cite{D-V-1}.
\begin{prp}(Left and right weak limits)\label{prp-3} Let $f_0$ be a kinetic initial datum and $f$ be a generalized kinetic solution to (\ref{P-19-1}) with initial {\color{rr}value} $f_0$. Then $f$ admits, almost surely, left and right limits respectively at every point $t\in [0,T]$. More precisely, for any  $t\in [0,T]$, there exist kinetic functions $f^{t\pm}$ on $\Omega\times \mathbb{T}^N\times \mathbb{R}$ such that $\mathbb{P}-$a.s.
\[
\langle f(t-r),\varphi\rangle\rightarrow \langle f^{t-},\varphi\rangle
\]
and
\[
\langle f(t+r),\varphi\rangle\rightarrow \langle f^{t+},\varphi\rangle
\]
as $r\rightarrow 0$ for all $\varphi\in C^1_c(\mathbb{T}^N\times \mathbb{R})$. Moreover, almost surely,
\[
\langle f^{t+}-f^{t-}, \varphi\rangle=-\int_{\mathbb{T}^N\times[0,T]\times \mathbb{R}}\partial_{\xi}\varphi(x,\xi)I_{\{t\}}(s)dm(x,s,\xi).
\]
In particular, almost surely, the set of $t\in [0,T]$ fulfilling that $f^{t+}\neq f^{t-}$ is countable.
\end{prp}
For a generalized kinetic solution $f$, define $f^{\pm}$ by $f^{\pm}(t)=f^{t \pm}$, $t\in [0,T]$. Since we are dealing with the filtration associated to Brownian motion, both $f^{\pm}$ are  clearly predictable as well. Also $f=f^+=f^-$ almost everywhere in time and we can take any of them in an integral with respect to the Lebesgue measure or in a stochastic integral. However, if the integral is with respect to a measure--typically a kinetic measure in this article, the integral is not well defined for $f$ and may differ if one chooses either $f^+$ or $f^-$.

As discussed above, with the aid of Proposition \ref{prp-3}, the weak form (\ref{P-22}) satisfied by a generalized kinetic solution can be strengthened to weak only respect to $x$ and $\xi$. Concretely,
 \begin{lemma}\label{lem-1}
 The generalized kinetic solution $f$ satisfying (\ref{P-22})  can be strengthened to for any $t\in [0,T]$ and $\varphi\in C^1_c(\mathbb{T}^N\times [0,T)\times \mathbb{R})$,
 \begin{eqnarray}\notag
&&-\langle f^+(t),\varphi\rangle+\langle f_{0}, \varphi\rangle+\int^t_0\langle f(s), a(\xi)\cdot \nabla \varphi\rangle ds\\
\notag
&=&-\sum_{k\geq 1}\int^t_0\int_{\mathbb{T}^N}\int_{\mathbb{R}}g_k(x,\xi)\varphi(x,\xi)d\nu_{x,s}(\xi)dxd\beta_k(s)\\
\label{qq-17}
&& -\frac{1}{2}\int^t_0\int_{\mathbb{T}^N}\int_{\mathbb{R}}\partial_{\xi}\varphi(x,\xi)G^2(x,\xi)d\nu_{x,s}(\xi)dxds+ \langle m,\partial_{\xi} \varphi\rangle([0,t]), \quad a.s.,
\end{eqnarray}
 where $\nu:=-\partial_{\xi} f$ and $\langle m,\partial_{\xi} \varphi\rangle([0,t])=\int_{\mathbb{T}^N\times[0,t]\times \mathbb{R}}\partial_{\xi}\varphi(x,\xi)dm(x,s,\xi)$.

 \end{lemma}
 \begin{proof}
 For all $t\in [0,T]$, consider a function $\alpha$ defined by
\begin{eqnarray}\label{qq-33}
\alpha(s)=\left\{
            \begin{array}{ll}
              1, & s\leq t,\\
              1-\frac{s-t}{r}, & t\leq s\leq t+r, \\
              0, & t+r\leq s,
            \end{array}
          \right.
\end{eqnarray}
 Then $\alpha\in C([0,T];\mathbb{R})=:C([0,T])$. Clearly, $C([0,T])$ with the metric of uniform convergence is a complete separable space and we know from Weierstrass approximation theorem that polynomials in one variable with rational coefficients denoted by $\mathbb{Q}[s]$ is a countable dense subset of $C([0,T])$. Therefore, we can take a sequence $\{\alpha_n\}_{n\geq 1}\subset \mathbb{Q}[s]$ such that
 \begin{eqnarray} \label{qq-26}
 \lim_{n\rightarrow \infty}\sup_{s\in [0,T]}|\alpha_n(s)-\alpha(s)|=0.
 \end{eqnarray}
 In view of $\alpha_n\in C^1_c([0,T])$, then  (\ref{P-22})  holds for a test function of the form $(x,s,\xi)\rightarrow \varphi (x,\xi)\alpha_n(s)$, where $\varphi\in C^1_c(\mathbb{T}^N\times \mathbb{R})$. Since $\{\alpha_n\}_{n\geq 1}$ is countable, we can find a common $\mathbb{P}$ full measure set $\tilde{\Omega}\subset \Omega$ such that  (\ref{P-22})  holds for all $\alpha_n$.

 Now, we focus on taking limitation $n\rightarrow \infty$ of (\ref{P-22}) for any fixed $\omega$ in $\tilde{\Omega}$. In the following, we fix $\omega\in\tilde{\Omega}$.
For all $\varphi\in C^1_c(\mathbb{T}^N\times \mathbb{R})$, the map
\begin{eqnarray*}
 J_{\varphi}: t&\longmapsto&\int^t_0\langle f(s),a(\xi)\cdot \nabla \varphi \rangle ds+\sum_{k\geq 1}\int^t_0\int_{\mathbb{T}^N}\int_{\mathbb{R}}g_k(x,\xi)\varphi(x,\xi)d\nu_{x,s}(\xi)dxd\beta_k(s)\\
 &&+\frac{1}{2}\int^t_0\int_{\mathbb{T}^N}\int_{\mathbb{R}}\partial_{\xi}\varphi(x,\xi)G^2(x,\xi)d\nu_{x,s}(\xi)dxds
\end{eqnarray*}
 is continuous. By Fubini theorem, the weak formulation (\ref{P-22}) for $\alpha_n$ is equivalent to
 \begin{eqnarray*}
 \int^T_0g(t)\alpha'_n(t)dt+\langle f_0, \varphi\rangle \alpha_n(0)=\langle m, \partial_{\xi}\varphi \rangle (\alpha_n),
\end{eqnarray*}
 where $g(t):=\langle f(t), \varphi\rangle-J_{\varphi}(t)$. Taking into account the fact that both $\langle f_0, \varphi\rangle \alpha_n(0)$ and $\langle m, \partial_{\xi}\varphi \rangle (\alpha_n)$ are finite, we obtain that $\partial_t g(t)$ is a Radon measure on $(0,T)$, i.e., the function $g(t)\in BV(0,T)$. Hence, by (\ref{qq-26}), it yields $\int^T_0\partial_tg(t)\alpha_n(t)dt\rightarrow\int^T_0\partial_tg(t)\alpha(t)dt$, as $n\rightarrow \infty$, which implies
\begin{eqnarray*}
 \int^T_0g(t)\alpha'_n(t)dt&=&g(T)\alpha_n(T)-g(0)\alpha_n(0)-\int^T_0\partial_tg(t)\alpha_n(t)dt\\
 &\rightarrow& g(T)\alpha(T)-g(0)\alpha(0)-\int^T_0\partial_tg(t)\alpha(t)dt=\int^T_0g(t)\alpha'(t)dt.
\end{eqnarray*}
 Since $\langle f_0, \varphi\rangle<\infty $, by (\ref{qq-26}), we have $\langle f_0, \varphi\rangle \alpha_n(0)\rightarrow \langle f_0, \varphi\rangle \alpha(0)$. Moreover, by (\ref{equ-37}), it yields $\langle m, \partial_{\xi}\varphi \rangle<\infty$, which implies
 $\langle m, \partial_{\xi}\varphi \rangle (\alpha_n)\rightarrow\langle m, \partial_{\xi}\varphi \rangle (\alpha)$. Thus, we have
 \begin{eqnarray*}
 \int^T_0g(t)\alpha'(t)dt+\langle f_0, \varphi\rangle \alpha(0)=\langle m, \partial_{\xi}\varphi \rangle (\alpha),
\end{eqnarray*}
 which means that (\ref{P-22})  holds for test functions of the form $(x,s,\xi)\rightarrow \varphi (x,\xi)\alpha(s)$, where $\varphi\in C^1_c(\mathbb{T}^N\times \mathbb{R})$ and $\alpha$ is defined by (\ref{qq-33}). By simple calculation, letting $r\rightarrow 0$, we derive that (\ref{qq-17}) holds.

\end{proof}

{\color{rr}As stated in the introduction, the starting point of this paper is the equivalence between entropy solution and kinetic solution. Now, we give a brief proof.
\begin{prp}\label{prp-6}
Let $u_0\in L^{m+1}(\mathbb{T}^N)$. For a kinetic solution to (\ref{P-19-1}) in the sense of Definition \ref{dfn-1} is equivalent to be an entropy solution $u$ to (\ref{P-19-1}) in the sense of Definition \ref{dfn-4}.
\end{prp}
\begin{proof} Let us begin with the proof of a kinetic solution is an entropy solution. To achieve it, we
choose test functions $\varphi(x,t,\xi)=\phi(x,t)\eta'(\xi)$, where the non-negative function $\phi(x,t)\in C^{1}_c(\mathbb{T}^N\times [0,T))$ and the convex function $\eta\in C^2(\mathbb{R})$ with $\eta''>0$ compactly supported. Assume $u(x,t)$ is a kinetic solution to (\ref{P-19-1}), then the corresponding kinetic functions can be written as $f=I_{u(x,t)>\xi}, f_0=I_{u_0>\xi}$. From (\ref{P-21}), we deduce that
\begin{eqnarray}\notag
&&\int^T_0\int_{\mathbb{T}^N}\int_{\mathbb{R}}I_{u(x,t)>\xi}\eta'(\xi) \partial_t \phi(x,t)d\xi dx dt+\int_{\mathbb{T}^N}\int_{\mathbb{R}}I_{u_0>\xi} \eta'(\xi) \phi(0,x)d\xi dx\\ \notag
&=& -\int^T_0\int_{\mathbb{T}^N}\int_{\mathbb{R}}I_{u(x,t)>\xi} a^2(\xi)\eta'(\xi) \Delta \phi (t)d\xi dx dt\\ \notag
&& +\int^T_0\int_{\mathbb{T}^N}\phi (x,t,u(x,t))\eta''(u(x,t))|\nabla \Psi(u)|^2dxdt\\ \notag
&& -\sum_{k\geq 1}\int^T_0\int_{\mathbb{T}^N} g^k(x, u(x,t))\phi (x,t,u(x,t))\eta'(u(x,t))dxd\beta_k(t) \\
\label{P-21-1}
&& -\frac{1}{2}\sum_{k\geq 1}\int^T_0\int_{\mathbb{T}^N}\phi (x,t,u(x,t))\eta''(u(x,t))G^2(x,u(t,x))dxdt+ m(\phi(x,t)\eta''(u(x,t))).
\end{eqnarray}

In view of $\phi\in C^1_c(\mathbb{T}^N\times[0,T))$, we deduce that
\begin{eqnarray}\notag
\int^T_0\int_{\mathbb{T}^N}\int_{\mathbb{R}}I_{u(x,t)>\xi}\eta'(\xi) \partial_t \phi(x,t)d\xi dx dt
&=&\int^T_0\int_{\mathbb{T}^N}(\eta(u(x,t))-\eta(-\infty)) \partial_t \phi(x,t) dx dt\\
\label{eq-1}
&=&\int^T_0\int_{\mathbb{T}^N}\eta(u(x,t)) \partial_t \phi(x,t) dx dt.
\end{eqnarray}
Similarly, it yields
\begin{eqnarray}\label{eq-2}
\int_{\mathbb{T}^N}\int_{\mathbb{R}}I_{u_0>\xi} \eta'(\xi) \phi(0,x)d\xi dx
=\int_{\mathbb{T}^N} \eta(u_0) \phi(0,x) dx.
\end{eqnarray}
Taking into account that $q'_{\eta}(\xi)=a^2(\xi)\eta'(\xi)$ and $\phi\in C^1_c(\mathbb{T}^N\times[0,T))$, we arrive at
\begin{eqnarray}\notag
&&\int^T_0\int_{\mathbb{T}^N}\int_{\mathbb{R}}I_{u(x,t)>\xi} a^2(\xi)\eta'(\xi) \Delta \phi (x,t)d\xi dx dt\\
\notag
&=& \int^T_0\int_{\mathbb{T}^N}\int_{\mathbb{R}}I_{u(x,t)>\xi} q'_{\eta}(\xi) \Delta \phi (x,t)d\xi dx dt\\ \notag
&=& \int^T_0\int_{\mathbb{T}^N}(q_{\eta}(u(x,t))-q_{\eta}(-\infty))\Delta \phi (x,t) dx dt\\
\label{eq-3}
&=& \int^T_0\int_{\mathbb{T}^N}q_{\eta}(u(x,t))\Delta \phi (x,t) dx dt.
\end{eqnarray}
Based on (\ref{eq-1})-(\ref{eq-3}) and by $m(\phi(x,t)\eta''(u(x,t)))\geq 0$, it follows that $u$ satisfies (\ref{e-2}).

Conversely, we suppose $u(x,t)$ is an entropy solution to
(\ref{P-19-1}) satisfying (\ref{e-2}). Then for any non-negative $\phi(x,t)\in C^1_c(\mathbb{T}^N\times[0,T))$ and any convex function $\eta\in C^2(\mathbb{R})$ with $\eta''>0$ compactly supported, we define a measure $m$ as follows:
\begin{eqnarray*}
m(\phi\otimes \eta'')&:=& \int^T_0\int_{\mathbb{T}^N}\eta(u)\partial_t \phi dxdt+\int_{\mathbb{T}^N}\eta(u_0)\phi(0)dx+\int^T_0\int_{\mathbb{T}^N}q_{\eta}(u)\Delta \phi dxdt\\
&&+\int^T_0\int_{\mathbb{T}^N}\Big(\frac{1}{2}\phi\eta''(u)G^2(x,u)-\phi\eta''(u)|\nabla \Psi(u)|^2\Big)dxdt\\
&& +\sum_{k\geq1}\int^T_0\int_{\mathbb{T}^N}\phi\eta'(u)g^k(x,u)dxd\beta_k(t)\geq 0.
\end{eqnarray*}
Taking $\eta(\xi)=\int^{\xi}_{-\infty}\varsigma$, by
utilizing (\ref{eq-1})-(\ref{eq-3}), we conclude that (\ref{P-21}) holds for $\varphi(x,t,\xi)=\phi (x,t)\varsigma(\xi)$. Since the test functions $\varphi(x,t,\xi)=\phi (x,t)\varsigma(\xi)$ form a dense subset of $C^1_c(\mathbb{T}^N\times[0,T)\times \mathbb{R})$, we get (\ref{P-21}) holds for any $\varphi\in C^1_c(\mathbb{T}^N\times[0,T)\times \mathbb{R})$.

\end{proof}
}

On the basis of Proposition \ref{prp-6}, we deduce from Theorem \ref{th-1} that
\begin{thm}\label{thm-4}
(Existence, Uniqueness) Let $u_0\in L^{m+1}(\mathbb{T}^N)$. Assume Hypothesis H holds. Then there is a unique kinetic solution $u$ to equation (\ref{P-19-1}) with initial datum $u_0$.
\end{thm}

\section{Freidlin-Wentzell large deviations and statement of the main result}
We start with a brief account of notions of large deviations.
Let $\{X^\varepsilon\}_{\varepsilon>0}$ be a family of random variables defined on a given probability space $(\Omega, \mathcal{F}, \mathbb{P})$ taking values in some Polish space $\mathcal{E}$.
\begin{dfn}
(Rate function) A function $I: \mathcal{E}\rightarrow [0,\infty]$ is called a rate function if $I$ is lower semicontinuous. A rate function $I$ is called a good rate function if the level set $\{x\in \mathcal{E}: I(x)\leq M\}$ is compact for each $M<\infty$.
\end{dfn}
\begin{dfn}
(Large deviation principle) The sequence $\{X^\varepsilon\}$ is said to satisfy the large deviation principle with rate function $I$ if for each Borel subset $A$ of $\mathcal{E}$
      \[
      -\inf_{x\in A^o}I(x)\leq \lim \inf_{\varepsilon\rightarrow 0}\varepsilon \log \mathbb{P}(X^\varepsilon\in A)\leq \lim \sup_{\varepsilon\rightarrow 0}\varepsilon \log \mathbb{P}(X^\varepsilon\in A)\leq -\inf_{x\in \bar{A}}I(x),
      \]
      where $A^o$ and $\bar{A}$ denote the interior and closure of $A$ in $\mathcal{E}$, respectively.
\end{dfn}

Suppose $W(t)$ is a cylindrical Wiener process on a Hilbert space $U$ defined on a filtered probability space $(\Omega, \mathcal{F},\{\mathcal{F}_t\}_{t\in [0,T]}, \mathbb{P} )$ ( that is, the paths of $W$ take values in $C([0,T];\mathcal{U})$, where $\mathcal{U}$ is another Hilbert space such that the embedding $U\subset \mathcal{U}$ is Hilbert-Schmidt).
Now we define
\begin{eqnarray*}
&\mathcal{A}:=\{\phi: \phi\ is\ a\ U\text{-}valued\ \{\mathcal{F}_t\}\text{-}predictable\ process\ such\ that \ \int^T_0 |\phi(s)|^2_Uds<\infty\ \mathbb{P}\text{-}a.s.\};\\
&S_M:=\{ h\in L^2([0,T];U): \int^T_0 |h(s)|^2_Uds\leq M\};\\
&\mathcal{A}_M:=\{\phi\in \mathcal{A}: \phi(\omega)\in S_M,\ \mathbb{P}\text{-}a.s.\}.
\end{eqnarray*}
Here and in the sequel of this paper, we will always refer to the weak topology on the set $S_M$.

Suppose for each $\varepsilon>0, \mathcal{G}^{\varepsilon}: C([0,T];\mathcal{U})\rightarrow \mathcal{E}$ is a measurable map and let $X^{\varepsilon}:=\mathcal{G}^{\varepsilon}(W)$. Now, we list below sufficient conditions for the large deviation principle of the sequence $X^{\varepsilon}$ as $\varepsilon\rightarrow 0$.
\begin{description}
  \item[\textbf{Condition A} ] There exists a measurable map $\mathcal{G}^0: C([0,T];\mathcal{U})\rightarrow \mathcal{E}$ such that the following conditions hold
\end{description}
\begin{description}
  \item[(a)] For every $M<\infty$, let $\{h^{\varepsilon}: \varepsilon>0\}$ $\subset \mathcal{A}_M$. If $h_{\varepsilon}$ converges to $h$ as $S_M$-valued random elements in distribution, then $\mathcal{G}^{\varepsilon}(W(\cdot)+\frac{1}{\sqrt{\varepsilon}}\int^{\cdot}_{0}h^\varepsilon(s)ds)$ converges in distribution to $\mathcal{G}^0(\int^{\cdot}_{0}h(s)ds)$.
  \item[(b)] For every $M<\infty$, the set $K_M=\{\mathcal{G}^0(\int^{\cdot}_{0}h(s)ds): h\in S_M\}$ is a compact subset of $\mathcal{E}$.
\end{description}

The following result is due to Budhiraja et al. in \cite{BD}.
\begin{thm}\label{thm-7}
If $\{\mathcal{G}^{\varepsilon}\}$ satisfies {condition A}, then $X^{\varepsilon}$ satisfies the large deviation principle on $\mathcal{E}$ with the
following good rate function $I$ defined by
\begin{eqnarray}\label{equ-27-1}
I(f)=\inf_{\{h\in L^2([0,T];U): f= \mathcal{G}^0(\int^{\cdot}_{0}h(s)ds)\}}\Big\{\frac{1}{2}\int^T_0|h(s)|^2_{U}ds\Big\},\ \ \forall f\in\mathcal{E}.
\end{eqnarray}
By convention, $I(f)=\infty$, if  $\Big\{h\in L^2([0,T];U): f= \mathcal{G}^0(\int^{\cdot}_{0}h(s)ds)\Big\}=\emptyset.$
\end{thm}

Recently, a new sufficient condition (Condition B below) to verify the assumptions in {condition A} (hence the large deviation principle) is proposed by Matoussi, Sabagh and Zhang in \cite{MSZ}. It turns out this new sufficient condition is suitable for establishing the large deviation principle for the stochastic porous media equation.
\begin{description}
  \item[\textbf{Condition B} ] There exists a measurable map $\mathcal{G}^0: C([0,T];\mathcal{U})\rightarrow \mathcal{E}$ such that the following two items hold
\end{description}
\begin{description}
  \item[(i)] For every $M<+\infty$, and for any family $\{h^{\varepsilon}; \varepsilon>0\}$ $\subset \mathcal{A}_M$ and any $\delta>0$,
      \[
      \lim_{\varepsilon\rightarrow 0}\mathbb{P}\Big(\rho(Y^\varepsilon, Z^\varepsilon)>\delta\Big)=0,
      \]
     where $Y^\varepsilon:=\mathcal{G}^{\varepsilon}\left(W(\cdot)+\frac{1}{\sqrt{\varepsilon}}\int^{\cdot}_{0}h^\varepsilon(s)ds\right)$, $Z^\varepsilon:=\mathcal{G}^0\left(\int^{\cdot}_{0}h^\varepsilon(s)ds\right)$,
     and $\rho(\cdot,\cdot)$ stands for the metric in the space $\mathcal{E}$.
  \item[(ii)] For every $M<+\infty$ and any family $\{h^\varepsilon; \varepsilon>0\}\subset S_M$ that converges to some element $h$ as $\varepsilon\rightarrow 0$,
      $\mathcal{G}^0\left(\int^{\cdot}_{0}h^\varepsilon(s)ds\right)$ converges to $\mathcal{G}^0\left(\int^{\cdot}_{0}h(s)ds\right)$ in the space $\mathcal{E}$.
\end{description}

\subsection{Statement of the main result}
In this paper, we are concerned with the following SPDE driven by small multiplicative noise
\begin{eqnarray}\label{P-1}
\left\{
  \begin{array}{ll}
  du^\varepsilon(x,t)=\Delta (A(u^\varepsilon))dt+\sqrt{\varepsilon}\sum_{k\geq 1}g^k(x,u^\varepsilon(t,x)) d\beta_k(t),\\
u^\varepsilon(0)=u_0,
  \end{array}
\right.
\end{eqnarray}
for $\varepsilon>0$, where $u_0\in L^{m+1}(\mathbb{T}^N)$. Under Hypothesis H, by Theorem \ref{thm-4}, there exists a unique kinetic solution $u^\varepsilon\in L^1([0,T]; L^1(\mathbb{T}^N))$  a.s..
Therefore, there exists a Borel-measurable function
\[
\mathcal{G}^{\varepsilon}: C([0,T];\mathcal{U})\rightarrow L^1([0,T];L^1(\mathbb{T}^N))
\]
such that $u^{\varepsilon}(\cdot)=\mathcal{G}^{\varepsilon}(W(\cdot))$.

Let $h\in L^2([0,T];U)$ with $h(t)=\sum_{k\geq 1}h_k(t) e_k$, we consider the following skeleton equation
\begin{eqnarray}\label{P-2}
\left\{
  \begin{array}{ll}
    du^h=\Delta (A(u^h))dt+\sum_{k\geq 1}g^k(x,u^h(t,x))h_k(t)dt,\\
    u^h(0)=u_0.
  \end{array}
\right.
\end{eqnarray}

The solution $u^h$, whose existence and uniqueness will be proved in next section, defines a measurable mapping $\mathcal{G}^0: C([0,T];\mathcal{U})\rightarrow L^1([0,T];L^1(\mathbb{T}^N))$ so that  $\mathcal{G}^0\Big(\int^{\cdot}_0 h(s)ds\Big):=u^h(\cdot)$.

\smallskip
We are ready to proceed with the statement of our main result.
\begin{thm}\label{thm-3}
Assume Hypothesis H holds. Then $u^{\varepsilon}$ satisfies the large deviation principle on $L^1([0,T];L^1(\mathbb{T}^N))$ with the good rate function $I$ given by (\ref{equ-27-1}).
\end{thm}

\section{Skeleton equation}

\subsection{Existence and uniqueness of solutions to the skeleton equation}\label{s-1}
In this subsection, we fix $h\in S_M$, and assume $h(t)=\sum\limits_{k\geq 1}h_k(t)e_k$, where $\{e_k\}_{k\geq 1}$ is an orthonormal basis of $U$. Now, we introduce definitions of solution to the skeleton equation (\ref{P-2}).
\begin{dfn}\label{dfn-4-1}(Entropy solution)
An entropy solution of (\ref{P-2}) is a measurable function $u^h: [0,T]\rightarrow L^{m+1}(\mathbb{T}^N)$ such that
\begin{description}
  \item[(i)] $u^h\in L^{m+1}(\Omega\times [0,T];L^{m+1}(\mathbb{T}^N))$
  \item[(ii)] For all $f\in C_b(\mathbb{R})$, we have $\Psi_f(u^h)\in L^2([0,T];H^1(\mathbb{T}^N))$ and
\[
\partial_i \Psi_f(u^h)=f(u^h)\partial_i \Psi(u^h).
\]

  \item[(iii)] For all convex function $\eta\in C^2(\mathbb{R})$ with $\eta''$ compactly supported and all non-negative $\phi\in C^1_c( \mathbb{T}^N\times[0,T))$, we have
\begin{eqnarray}\notag
-\int^T_0\int_{\mathbb{T}^N}\eta(u^h)\partial_t \phi dxdt&\leq & \int_{\mathbb{T}^N}\eta(u_0)\phi(0)dx+\int^T_0\int_{\mathbb{T}^N}q_{\eta}(u)\Delta \phi dxdt\\
\notag
&& -\int^T_0\int_{\mathbb{T}^N}\phi\eta''(u^h)|\nabla \Psi(u^h)|^2dxdt\\
\label{e-2-1}
&& +\sum_{k\geq 1}\int^T_0\int_{\mathbb{T}^N} g^k(x,u^h(x,t))\phi\eta'(u^h)h_k(t)dxdt, \quad a.s.,
\end{eqnarray}

where $q_{\eta}$ is any function satisfying $q'_{\eta}(\xi)=\eta'(\xi)a^2(\xi)$.
\end{description}

\end{dfn}
\begin{dfn}(Kinetic solution)
Let $u_0\in L^{m+1}(\mathbb{T}^N)$. A measurable function $u^h: \mathbb{T}^N\times [0,T]\rightarrow \mathbb{R}$ is said to be a kinetic solution to (\ref{P-2}), if there exists $C_m>0$ such that
\begin{eqnarray}\label{e-19}
\int^T_0\|u^h(t)\|^{m+1}_{L^{m+1}(\mathbb{T}^N)}dt\leq C_m,
\end{eqnarray}
and if there exists a measure $m_h\in \mathcal{M}^+_0(\mathbb{T}^N\times [0,T]\times \mathbb{R})$ such that $f_h:= I_{u^h>\xi}$ satisfies that for all $\varphi\in C^1_c(\mathbb{T}^N\times [0,T]\times \mathbb{R})$,
\begin{eqnarray}\notag
&&\int^T_0\langle f_h(t), \partial_t \varphi(t)\rangle dt+\langle f_0, \varphi(0)\rangle +\int^T_0\langle f_h(t), a^2(\xi)\Delta \varphi (t)\rangle dt\\
\notag
&=&\int^T_0\int_{\mathbb{T}^N}\partial_{\xi}\varphi (x,t,u^h(x,t))|\nabla \Psi(u^h)|^2dxdt\\
\label{P-3}
&& -\sum_{k\geq 1}\int^T_0\int_{\mathbb{T}^N} g^k(x,u^h(x,t))\varphi (x,t, u^h(x,t))h_k(t)dxdt + m_h(\partial_{\xi} \varphi),
\end{eqnarray}
where $f_0(x,\xi)=I_{u_0(x)>\xi}$.
\end{dfn}

\begin{dfn}(Generalized kinetic solution)
Let $f_0:\mathbb{T}^N\times \mathbb{R}\rightarrow [0,1]$ be a kinetic function. A measurable function $f_h:\mathbb{T}^N\times [0,T]\times \mathbb{R}\rightarrow [0,1]$  is said to be a generalized kinetic solution to (\ref{P-2}) with the initial datum $f_0$, if $(f_h(t))=(f_h(t,\cdot,\cdot))$ is a kinetic function such that for $\nu^h:=-\partial_{\xi} f_h$ satisfies
\begin{eqnarray}
\int^T_0\int_{\mathbb{T}^N}\int_{\mathbb{R}}|\xi|^{m+1}d\nu^h_{x,t}(\xi)dxdt\leq C_m,
\end{eqnarray}
where $C_m$ is a positive constant and there exists a measure $m_h\in \mathcal{M}^+_0(\mathbb{T}^N\times [0,T]\times \mathbb{R})$ such that for all $\varphi\in C^1_c(\mathbb{T}^N\times [0,T]\times \mathbb{R})$,
\begin{eqnarray}\label{P-4}\notag
&&\int^T_0\langle f_h(t), \partial_t \varphi(t)\rangle dt+\langle f_0, \varphi(0)\rangle +\int^T_0\langle f_h(t), a^2(\xi)\Delta \varphi (t)\rangle dt\\
&=& \int^T_0\int_{\mathbb{T}^N}\int_{\mathbb{R}}\partial_{\xi}\varphi (x,t,\xi)|\nabla \Psi(\xi)|^2d\nu_{x,t}(\xi)dxdt\\ \notag
&& -\sum_{k\geq 1}\int^T_0\int_{\mathbb{T}^N}\int_{\mathbb{R}} g^k(x,\xi)\varphi (x,t,\xi)h_k(t)d\nu^h_{x,t}(\xi)dxdt + m_h(\partial_{\xi} \varphi).
\end{eqnarray}
\end{dfn}

\

In the following, we firstly prove the uniqueness of the skeleton equations (\ref{P-2}). Then, based on the uniqueness, we show the existence.

Similarly to Lemma \ref{lem-1}, we reformulate (\ref{P-4}) to a strong version, it yields that for all $t\in [0,T]$,
\begin{eqnarray}\notag
&&-\langle f^+_h(t),\varphi\rangle+\langle f_0, \varphi\rangle+\int^t_0\langle f_h(s), a^2(\xi)\Delta \varphi\rangle ds\\
\notag
&=&
\int^t_0\int_{\mathbb{T}^N}\int_{\mathbb{R}}\partial_{\xi}\varphi (x,\xi)|\nabla \Psi(\xi)|^2d\nu_{x,s}(\xi)dxds\\
\label{P-5}
&& -\sum_{k\geq 1}\int^t_0\int_{\mathbb{T}^N}\int_{\mathbb{R}} g^k(x,\xi)\varphi (x,\xi)h_k(s)d\nu^h_{x,s}(\xi)dxds + \langle m_h,\partial_{\xi} \varphi\rangle([0,t]),
\end{eqnarray}
where $\langle m_h,\partial_{\xi} \varphi\rangle([0,t])=\int_{\mathbb{T}^N\times[0,t]\times \mathbb{R}}\partial_{\xi}\varphi(x,\xi)dm_h(x,s,\xi)$.

\
Consider the solution $\tilde{u}^h$ being the solution of
\begin{eqnarray}\label{e-6}
\left\{
  \begin{array}{ll}
    d\tilde{u}^h(t,x)=\Delta (\tilde{A}(\tilde{u}^h))dt+\sum_{k\geq 1}\tilde{g}^k(x,\tilde{u}^h(t,x)) h_k(t)dt,\\
    \tilde{u}^h(0)=\tilde{u}_0.
  \end{array}
\right.
\end{eqnarray}
Define
\[
\tilde{a}(r)=\sqrt{\tilde{A}'(r)}, \quad \tilde{\Psi}(r)=\int^r_0 \tilde{a}(s)ds.
\]
In the following, with the help of (\ref{P-5}), we prove a comparison theorem for two generalized kinetic solutions of  (\ref{P-2}) and (\ref{e-6}), respectively.

\begin{prp}\label{prp-1}
Assume $(A,g,u_0)$ and $(\tilde{A},\tilde{g},\tilde{u}_0)$ satisfy Hypothesis H, and let $f_i, i=1,2$ be two generalized solutions to (\ref{P-2}) and (\ref{e-6}) {\color{rr}with $f_{1,0}(x,\xi)=I_{u_0(x)>\xi}$ and $f_{2,0}(x,\xi)=I_{\tilde{u}_0(x)>\xi}$}, then, for $0\leq t\leq T$, and nonnegative test functions $\rho\in C^{\infty}(\mathbb{T}^N), \psi\in C^{\infty}_c(\mathbb{R})$, we have
\begin{eqnarray}\notag
&&\int_{(\mathbb{T}^N)^2}\int_{\mathbb{R}^2}\rho(x-y)\psi(\xi-\zeta)\Big(f^{\pm}_1(x,t,\xi)\bar{f}^{\pm}_2(y,t,\zeta)+\bar{f}^{\pm}_1(x,t,\xi){f}^{\pm}_2(y,t,\zeta)\Big)d\xi d\zeta dxdy\\ \notag
&\leq& \int_{(\mathbb{T}^N)^2}\int_{\mathbb{R}^2}\rho(x-y)\psi(\xi-\zeta)\Big(f_{1,0}(x,\xi)\bar{f}_{2,0}(y,\zeta)+\bar{f}_{1,0}(x,\xi){f}_{2,0}(y,\zeta)\Big)d\xi d\zeta dxdy\\
\label{P-7}
&&\ +2K_1+2K_2+2K_3,
\end{eqnarray}
where
\begin{eqnarray*}
K_1&=&-\int^t_0\int_{(\mathbb{T}^N)^2}\int_{\mathbb{R}^2}\partial^2_{x_iy_i}\rho(x-y)(l(\xi,\zeta)+\tilde{l}(\xi,\zeta))d\nu^1_{x,s}\otimes d \nu^2_{y,s}(\xi,\zeta) dxdyds,\\
K_2&=&-2\int^t_0\int_{(\mathbb{T}^N)^2}\int_{\mathbb{R}^2}\nabla_x \Psi(\xi)\cdot\nabla_y \tilde{\Psi}(\zeta) \alpha d\nu^1_{x,s}\otimes d \nu^2_{y,s}(\xi,\zeta)dxdyds,
\end{eqnarray*}
and
\begin{eqnarray*}
K_3=\sum_{k\geq 1}\int^t_0\int_{(\mathbb{T}^N)^2}\rho(x-y)\int_{\mathbb{R}^2}{\color{rr}\chi_1(\xi,\zeta)}(g^{k}(x,\xi)-\tilde{g}^{k}(y,\zeta))h_k(s)d \nu^1_{x,s}\otimes d\nu^2_{y,s}(\xi,\zeta) dxdyds
\end{eqnarray*}
with $l(\xi,\zeta)=\int^{\xi}_{\zeta}\int^{\xi'}_{\zeta}\psi(\xi'-\zeta')a^2(\xi')d\zeta'd\xi'$, $\tilde{l}(\xi,\zeta)=\int^{\xi}_{\zeta}\int^{\xi'}_{\zeta}\psi(\xi'-\zeta')\tilde{a}^2(\xi')d\zeta'd\xi'$ and {\color{rr}$\chi_1(\xi,\zeta)=\int^{\xi}_{-\infty}\psi(\xi'-\zeta)d\xi'=\int^{\xi-\zeta}_{-\infty}\psi(y)dy$}.
\end{prp}
\begin{proof}
Denote by $f_1(x,t,\xi)$ and $f_2(y,t,\zeta)$ be two generalized solutions to (\ref{P-2}) and (\ref{e-6}) with the corresponding kinetic measures $m_1$ and $m_2$.
Let $\varphi_1\in C^{\infty}_c(\mathbb{T}^N_x\times \mathbb{R}_{\xi})$ and
$\varphi_2\in C^{\infty}_c(\mathbb{T}^N_y\times \mathbb{R}_{\zeta})$. By (\ref{P-5}), we have
\begin{eqnarray}\label{P-8}\notag
\langle f^{+}_1(t),\varphi_1\rangle&=&\langle f_{1,0}, \varphi_1\rangle+\int^t_0\langle f_1(s), a^2(\xi) \Delta_x \varphi_1(s)\rangle ds\\ \notag
&& -\int^t_0\int_{\mathbb{T}^N}\int_{\mathbb{R}}\partial_{\xi}\varphi_1 (x,\xi)|\nabla_x \Psi(\xi)|^2d\nu_{x,s}(\xi)dxds\\
&& +\sum_{k\geq 1}\int^t_0\int_{\mathbb{T}^N}\int_{\mathbb{R}} g^{k}(x,\xi)\varphi_1 (x,\xi)h_k(s)d\nu^1_{x,s}(\xi)dxds -\langle m_1,\partial_{\xi} \varphi_1\rangle([0,t]),
\end{eqnarray}
where $f_{1,0}=I_{u_0>\xi}$ and $\nu^1_{x,s}(\xi)=-\partial_{\xi}f^{+}_1(s,x,\xi)=\partial_{\xi}\bar{f}^{+}_1(s,x,\xi)$.
Similarly,
\begin{eqnarray}\label{P-9}\notag
\langle \bar{f}^{+}_2(t),\varphi_2\rangle&=&\langle \bar{f}_{2,0}, \varphi_2\rangle+\int^t_0\langle \bar{f}_2(s), \tilde{a}^2(\zeta)\Delta_y \varphi_2(s)\rangle ds\\ \notag
&& +\int^t_0\int_{\mathbb{T}^N}\int_{\mathbb{R}}\partial_{\zeta}\varphi_2 (y,\zeta)|\nabla_y \tilde{\Psi}(\zeta)|^2d\nu_{y,s}(\zeta)dyds\\ \notag
&& -\sum_{k\geq 1}\int^t_0\int_{\mathbb{T}^N}\int_{\mathbb{R}} \tilde{g}^{k}(y,\zeta)\varphi_2 (y,\zeta)h_k(s)d\nu^2_{y,s}(\zeta)dyds +\langle m_2,\partial_{\zeta} \varphi_2\rangle([0,t]).
\end{eqnarray}
where $f_{2,0}=I_{\tilde{u}_0>\zeta}$ and  $\nu^2_{y,s}(\zeta)=\partial_{\zeta}\bar{f}^{+}_2(s,y,\zeta)=-\partial_{\zeta}f^{+}_2(s,y,\zeta)$.

Denote the duality distribution over $\mathbb{T}^N_x\times \mathbb{R}_\xi\times \mathbb{T}^N_y\times \mathbb{R}_\zeta$ by $\langle\langle\cdot,\cdot\rangle\rangle$. Setting $\alpha(x,\xi,y,\zeta)=\varphi_1(x,\xi)\varphi_2(y,\zeta)$ and using the integration by parts formula, we have
\begin{eqnarray}\notag
\langle\langle f^+_1(t)\bar{f}^+_2(t), \alpha \rangle\rangle
&=& \langle\langle f_{1,0}\bar{f}_{2,0}, \alpha \rangle\rangle+\int^t_0\int_{(\mathbb{T}^N)^2}\int_{\mathbb{R}^2}f_1\bar{f}_2(a^2(\xi)\Delta_x+\tilde{a}^2(\zeta)\Delta_y) \alpha d\xi d\zeta dxdyds\\ \notag
&& +\int^t_0\int_{(\mathbb{T}^N)^2}\int_{\mathbb{R}^2}f^+_1|\nabla_y \tilde{\Psi}(\zeta)|^2\partial_{\zeta}\alpha d\xi d\nu^2_{y,s}(\zeta)dxdyds\\ \notag
&& -\int^t_0\int_{(\mathbb{T}^N)^2}\int_{\mathbb{R}^2}\bar{f}^+_2|\nabla_x \Psi(\xi)|^2\partial_{\xi}\alpha  d\nu^1_{x,s}(\xi)d\zeta dxdyds\\ \notag
 &&-\sum_{k\geq 1}\int^t_0\int_{(\mathbb{T}^N)^2}\int_{\mathbb{R}^2}f^+_1(s,x,\xi)\alpha \tilde{g}^{k}(y,\zeta)h_k(s)d\xi d\nu^2_{y,s}(\zeta)dxdyds\\ \notag
&&+\sum_{k\geq 1}\int^t_0\int_{(\mathbb{T}^N)^2}\int_{\mathbb{R}^2}\bar{f}^+_2(s,y,\zeta)\alpha g^k(x,\xi)h_k(s)d\zeta d\nu^1_{x,s}(\xi)dxdyds\\ \notag
 &&+\int^t_0\int_{(\mathbb{T}^N)^2}\int_{\mathbb{R}^2}f^+_1(s,x,\xi)\partial_{\zeta} \alpha dm_2(y,\zeta,s)d\xi dx\\ \notag
 &&-\int^t_0\int_{(\mathbb{T}^N)^2}\int_{\mathbb{R}^2}\bar{f}^+_2(s,y,\zeta)\partial_{\xi} \alpha dm_1(x,\xi,s)d\zeta dy\\
\label{P-10}
&=:& \langle\langle f_{1,0}\bar{f}_{2,0}, \alpha \rangle\rangle+I_1+I_2+I_3+I_4+I_5+I_6+I_7.
\end{eqnarray}
Similarly, we  have
\begin{eqnarray}\notag
\langle\langle \bar{f}^+_1(t)f^+_2(t), \alpha \rangle\rangle
&=& \langle\langle \bar{f}_{1,0}f_{2,0}, \alpha \rangle\rangle+\int^t_0\int_{(\mathbb{T}^N)^2}\int_{\mathbb{R}^2}\bar{f}_1{f}_2(a^2(\xi)\Delta_x+\tilde{a}^2(\zeta)\Delta_y) \alpha d\xi d\zeta dxdyds\\ \notag
&& -\int^t_0\int_{(\mathbb{T}^N)^2}\int_{\mathbb{R}^2}\bar{f}^+_1|\nabla_y \tilde{\Psi}(\zeta)|^2\partial_{\zeta}\alpha d\xi d\nu^2_{y,s}(\zeta)dxdyds\\ \notag
&& +\int^t_0\int_{(\mathbb{T}^N)^2}\int_{\mathbb{R}^2}f^+_2|\nabla_x \Psi(\xi)|^2\partial_{\xi}\alpha  d\nu^1_{x,s}(\xi)d\zeta dxdyds\\ \notag
&& \ +\sum_{k\geq 1}\int^t_0\int_{(\mathbb{T}^N)^2}\int_{\mathbb{R}^2}\bar{f}^+_1(s,x,\xi)\alpha \tilde{g}^{k}(y,\zeta)h_k(s)d\xi d\nu^2_{y,s}(\zeta)dxdyds\\ \notag
&&\ -\sum_{k\geq 1}\int^t_0\int_{(\mathbb{T}^N)^2}\int_{\mathbb{R}^2}{f}^+_2(s,y,\zeta)\alpha g^k(x,\xi)h_k(s) d\nu^1_{x,s}(\xi)d\zeta dxdyds\\ \notag
&&\ -\int^t_0\int_{(\mathbb{T}^N)^2}\int_{\mathbb{R}^2}\bar{f}^+_1(s,x,\xi)\partial_{\zeta} \alpha dm_2(y,\zeta,s)d\xi dx\\ \notag
 &&\ +\int^t_0\int_{(\mathbb{T}^N)^2}\int_{\mathbb{R}^2}{f}^+_2(s,y,\zeta)\partial_{\xi} \alpha dm_1(x,\xi,s)d\zeta dy\\
\label{P-10-1}
&=:& \langle\langle \bar{f}_{1,0}{f}_{2,0}, \alpha \rangle\rangle+\bar{I}_1+\bar{I}_2+\bar{I}_3+\bar{I}_4+\bar{I}_5+\bar{I}_6+\bar{I}_7.
\end{eqnarray}

By a density argument, (\ref{P-10}) and (\ref{P-10-1}) remain true for any test function $\alpha\in C^{\infty}_c(\mathbb{T}^N_x\times \mathbb{R}_\xi\times \mathbb{T}^N_y\times \mathbb{R}_\zeta)$. The assumption that $\alpha$ is compactly supported can be relaxed thanks to (\ref{equ-37}) on $m_i$ and (\ref{equ-26}) on $\nu_i$, $i=1,2.$
 Using a truncation argument of $\alpha$, it is easy to see that (\ref{P-10}) and (\ref{P-10-1}) remain true if $\alpha \in C^{\infty}_b(\mathbb{T}^N_x\times \mathbb{R}_\xi\times \mathbb{T}^N_y\times \mathbb{R}_\zeta)$ is compactly supported in a neighbourhood of the diagonal
\[
\Big\{(x,\xi,x,\xi); x\in \mathbb{T}^N, \xi\in \mathbb{R}\Big\}.
\]
Taking $\alpha=\rho(x-y)\psi(\xi-\zeta)$, then we have the following remarkable identities
\begin{eqnarray}\label{P-11}
(\nabla_x+\nabla_y)\alpha=0, \quad (\partial_{\xi}+\partial_{\zeta})\alpha=0.
\end{eqnarray}
With the aid of (\ref{P-11}), we get
\begin{eqnarray*}
I_6&=&-\int^t_0\int_{(\mathbb{T}^N)^2}\int_{\mathbb{R}^2}f^+_1(s,x,\xi)\partial_{\xi} \alpha dm_2(y,\zeta,s)d\xi dx\\
&=&-\int^t_0\int_{(\mathbb{T}^N)^2}\int_{\mathbb{R}^2}\alpha dm_2(y,\zeta,s)d\nu^1_{x,s}(\xi)dx\leq 0,
\end{eqnarray*}
and
\begin{eqnarray*}
I_7&=&\int^t_0\int_{(\mathbb{T}^N)^2}\int_{\mathbb{R}^2}\bar{f}^+_2(s,y,\zeta)\partial_{\zeta} \alpha dm_1(x,\xi,s)d\zeta dy\\
&=&-\int^t_0\int_{(\mathbb{T}^N)^2}\int_{\mathbb{R}^2}\alpha dm_1(x,\xi,s)d\nu^2_{y,s}(\zeta) dy\leq 0.
\end{eqnarray*}
Similarly, we have $\bar{I}_6+\bar{I}_7\leq 0$.
Define
\[
l(\xi,\zeta)=\int^{\xi}_{\zeta}\int^{\xi'}_{\zeta}\psi(\xi'-\zeta')a^2(\xi')d\zeta'd\xi'.
\]
Clearly, $\partial_{\zeta}\partial_{\xi}l(\xi,\zeta)=-\psi(\xi-\zeta)a^2(\xi)$.
Then we deduce that
\begin{eqnarray}\notag
&&\int^t_0\int_{(\mathbb{T}^N)^2}\int_{\mathbb{R}^2}f_1\bar{f}_2a^2(\xi)\Delta_x\alpha d\xi d\zeta dxdyds\\ \notag
&=&- \int^t_0\int_{(\mathbb{T}^N)^2}\int_{\mathbb{R}^2}f_1\bar{f}_2\Delta_x\rho(x-y)\partial_{\zeta}\partial_{\xi}l(\xi,\zeta)d\xi d\zeta dxdyds\\ \notag
&=& \int^t_0\int_{(\mathbb{T}^N)^2}\int_{\mathbb{R}^2}f_1\partial_{\zeta}\bar{f}_2\Delta_x\rho(x-y)\partial_{\xi}l(\xi,\zeta)d\xi d\zeta dxdyds\\ \notag
&=& \int^t_0\int_{(\mathbb{T}^N)^2}\int_{\mathbb{R}^2}f_1\Delta_x\rho(x-y)\partial_{\xi}l(\xi,\zeta)d\xi d \nu^2_{y,s}(\zeta) dxdyds\\ \notag
&=& \int^t_0\int_{(\mathbb{T}^N)^2}\int_{\mathbb{R}^2}\Delta_x\rho(x-y)l(\xi,\zeta)d\nu^1_{x,s}\otimes d \nu^2_{y,s}(\xi,\zeta) dxdyds\\
\label{e-4}
&=& -\int^t_0\int_{(\mathbb{T}^N)^2}\int_{\mathbb{R}^2}\partial^2_{x_iy_i}\rho(x-y)l(\xi,\zeta)d\nu^1_{x,s}\otimes d \nu^2_{y,s}(\xi,\zeta) dxdyds.
\end{eqnarray}

Define
\[
\tilde{l}(\xi,\zeta)=\int^{\xi}_{\zeta}\int^{\xi'}_{\zeta}\psi(\xi'-\zeta')\tilde{a}^2(\zeta')d\zeta'd\xi'.
\]
Clearly, $\partial_{\zeta}\partial_{\xi}\tilde{l}(\xi,\zeta)=-\psi(\xi-\zeta)\tilde{a}^2(\zeta)$.
Similar to the above, we deduce that
\begin{eqnarray}\notag
&&\int^t_0\int_{(\mathbb{T}^N)^2}\int_{\mathbb{R}^2}f_1\bar{f}_2\tilde{a}^2(\zeta)\Delta_y \alpha d\xi d\zeta dxdyds\\
\label{e-5}
&=& -\int^t_0\int_{(\mathbb{T}^N)^2}\int_{\mathbb{R}^2}\partial^2_{x_iy_i}\rho(x-y)\tilde{l}(\xi,\zeta)d\nu^1_{x,s}\otimes d \nu^2_{y,s}(\xi,\zeta) dxdyds.
\end{eqnarray}
Combining (\ref{e-4}) and (\ref{e-5}), we get
\begin{eqnarray*}
I_1&=&-\int^t_0\int_{(\mathbb{T}^N)^2}\int_{\mathbb{R}^2}\partial^2_{x_iy_i}\rho(x-y)(l(\xi,\zeta)+\tilde{l}(\xi,\zeta))d\nu^1_{x,s}\otimes d \nu^2_{y,s}(\xi,\zeta) dxdyds\\
&=:&K_1.
\end{eqnarray*}
By the same method, we get $\bar{I}_1=K_1.$
Utilizing (\ref{P-11}), it follows that
\begin{eqnarray*}
I_2&=&-\int^t_0\int_{(\mathbb{T}^N)^2}\int_{\mathbb{R}^2}f^+_1|\nabla_y \tilde{\Psi}(\zeta)|^2\partial_{\xi}\alpha d\xi d\nu^2_{y,s}(\zeta)dxdyds\\
&=& \int^t_0\int_{(\mathbb{T}^N)^2}\int_{\mathbb{R}^2}\partial_{\xi}f^+_1|\nabla_y \tilde{\Psi}(\zeta)|^2\alpha d\xi d\nu^2_{y,s}(\zeta)dxdyds\\
&=& -\int^t_0\int_{(\mathbb{T}^N)^2}\int_{\mathbb{R}^2}|\nabla_y \tilde{\Psi}(\zeta)|^2\alpha d\nu^1_{x,s}\otimes d \nu^2_{y,s}(\xi,\zeta)dxdyds,
\end{eqnarray*}
and
\begin{eqnarray*}
I_3=-\int^t_0\int_{(\mathbb{T}^N)^2}\int_{\mathbb{R}^2}|\nabla_x \Psi(\xi)|^2 \alpha d\nu^1_{x,s}\otimes d \nu^2_{y,s}(\xi,\zeta)dxdyds.
\end{eqnarray*}
Combining the above two estimates, it follows that
\begin{eqnarray*}
I_2+I_3&=&-\int^t_0\int_{(\mathbb{T}^N)^2}\int_{\mathbb{R}^2}(|\nabla_x \Psi(\xi)|^2+|\nabla_y \tilde{\Psi}(\zeta)|^2) \alpha d\nu^1_{x,s}\otimes d \nu^2_{y,s}(\xi,\zeta)dxdyds\\
&\leq& -2\int^t_0\int_{(\mathbb{T}^N)^2}\int_{\mathbb{R}^2}\nabla_x \Psi(\xi)\cdot\nabla_y \tilde{\Psi}(\zeta) \alpha d\nu^1_{x,s}\otimes d \nu^2_{y,s}(\xi,\zeta)dxdyds\\
&=:& K_2.
\end{eqnarray*}
Using the same method as above, we achieve $\bar{I}_2+\bar{I}_3\leq K_2$.
Set
\[
\chi_1(\xi,\zeta)=\int^{\xi}_{-\infty}\psi(\xi'-\zeta)d\xi'
\]
for some $\xi, \zeta\in\mathbb{R}$.  Then
\begin{eqnarray}\notag
I_4
&=&-\sum_{k\geq 1}\int^t_0\int_{(\mathbb{T}^N)^2}\int_{\mathbb{R}^2}f^+_1(s,x,\xi)\rho(x-y)\partial_{\xi} \chi_1(\xi,\zeta) \tilde{g}^{k}(y,\zeta)h_k(s)d\xi d\nu^2_{y,s}(\zeta)dxdyds\\ \notag
&=&-\sum_{k\geq 1}\int^t_0\int_{(\mathbb{T}^N)^2}\int_{\mathbb{R}}\rho(x-y) \tilde{g}^{k}(y,\zeta)h_k(s)\Big(\int_{\mathbb{R}}f^+_1(s,x,\xi)\partial_{\xi} \chi_1(\xi,\zeta)d\xi\Big) d\nu^2_{y,s}(\zeta)dxdyds\\
\label{P-12}
&=&-\sum_{k\geq 1}\int^t_0\int_{(\mathbb{T}^N)^2}\int_{\mathbb{R}^2}\rho(x-y)\chi_1(\xi,\zeta) \tilde{g}^{k}(y,\zeta)h_k(s)d \nu^1_{x,s}\otimes d\nu^2_{y,s}(\xi,\zeta)dxdyds.
\end{eqnarray}
The third equality is obtained by
\begin{eqnarray*}\notag
\int_{\mathbb{R}}f^+_1(s,x,\xi)\partial_{\xi} \chi_1(\xi,\zeta)d\xi
&=&-\int_{\mathbb{R}}\partial_{\xi}f^+_1(s,x,\xi) \chi_1(\xi,\zeta)d\xi\\
&=&\int_{\mathbb{R}} \chi_1(\xi,\zeta)d\nu^1_{x,s}(\xi).
\end{eqnarray*}
Similarly, for $\xi, \zeta\in\mathbb{R}$, let
\[
\chi_2(\zeta,\xi)=\int^{\infty}_{\zeta}\psi(\xi-\zeta')d\zeta',
\]
then
\begin{eqnarray}\notag
I_5
&=&-\sum_{k\geq 1}\int^t_0\int_{(\mathbb{T}^N)^2}\int_{\mathbb{R}^2}\bar{f}^+_2(s,y,\zeta)\rho(x-y)\partial_{\zeta}\chi_2(\zeta, \xi) g^k(x,\xi)h_k(s)d\nu^1_{x,s}(\xi)d\zeta dxdyds\\ \notag
&=&-\sum_{k\geq 1}\int^t_0\int_{(\mathbb{T}^N)^2}\int_{\mathbb{R}}\rho(x-y) g^k(x,\xi)h_k(s)\Big(\int_{\mathbb{R}} \bar{f}^+_2(s,y,\zeta)\partial_{\zeta}\chi_2(\zeta, \xi)d\zeta\Big) d\nu^1_{x,s}(\xi) dxdyds\\
\label{P-13}
&=& \sum_{k\geq 1}\int^t_0\int_{(\mathbb{T}^N)^2}\int_{\mathbb{R}^2}\chi_2(\zeta, \xi)\rho(x-y) g^k(x,\xi)h_k(s) d \nu^1_{x,s}\otimes d\nu^2_{y,s}(\xi,\zeta) dxdyds.
\end{eqnarray}

Note that $\chi_1(\xi,\zeta)=\chi_2(\zeta, \xi)=\int^{\xi-\zeta}_{-\infty}\psi(y)dy$. We deduce from (\ref{P-12}) and (\ref{P-13}) that
\begin{eqnarray*}
I_4+I_5&=&\sum_{k\geq 1}\int^t_0\int_{(\mathbb{T}^N)^2}\rho(x-y)\int_{\mathbb{R}^2}\chi_1(\xi,\zeta)( g^k(x,\xi)-\tilde{g}^{k}(y,\zeta))h_k(s)d \nu^1_{x,s}\otimes d\nu^2_{y,s}(\xi,\zeta) dxdyds\\
&=:& K_3.
\end{eqnarray*}
Similarly, we have $\bar{I}_4+\bar{I}_5=K_3$.
Based on the above, the equation (\ref{P-7}) is established for $f^+_i$. To obtain the result for  $f^-_i$, we take $t_n\uparrow t$, write (\ref{P-7}) for $f^+_i(t_n)$ and let $n\rightarrow \infty$.

\end{proof}

\begin{thm}\label{thm-2}
Assume $(A,g,u_0)$ and $(\tilde{A},\tilde{g},\tilde{u}_0)$ satisfy Hypothesis H and let $u:=u^h, \tilde{u}:=\tilde{u}^h$ be kinetic solutions to (\ref{P-2}) and (\ref{e-6}), respectively. We claim that
\begin{description}
\item[(1)] if $A=\tilde{A}$ and $g=\tilde{g}$, then for a.e. $t\in [0,T]$,
\begin{eqnarray}\label{e-20}
\|u(t)-\tilde{u}(t)\|_{L^1(\mathbb{T}^N)}
\leq e^{K(T+M)}\|u_0-\tilde{u}_0\|_{L^1(\mathbb{T}^N)} .
\end{eqnarray}
  \item[(2)] furthermore, for all $\gamma, \delta\in (0,1)$, $\lambda\in [0,1]$ and $\alpha\in (0,1\wedge\frac{m}{2})$, we have for a.e. $t\in [0,T]$, it holds that
\begin{eqnarray*}
\|u(t)-\tilde{u}(t)\|_{L^1(\mathbb{T}^N)}
&\leq& \mathcal{E}_t(\gamma,\delta)+e^{K(T+M)}\Big[\|u_0-\tilde{u}_0\|_{L^1(\mathbb{T}^N)} +\mathcal{E}_0(\gamma,\delta)\\ \notag
&& +N_0\gamma^{-2}(\lambda^2+\delta^{2\alpha})\Big(1+\|u\|^m_{L^m([0,T]\times \mathbb{T}^N)}+\|\tilde{u}\|^m_{L^m([0,T]\times \mathbb{T}^N)}\Big)\\
\notag
&& +N_0\gamma^{-2}\|I_{|u|\geq R_{\lambda}}(1+|u|)\|^m_{L^m([0,T]\times \mathbb{T}^N)}+N_0\gamma^{-2}\|I_{|\tilde{u}|\geq R_{\lambda}}(1+|\tilde{u}|)\|^m_{L^m([0,T]\times \mathbb{T}^N)} \\
&& +2 K(\gamma+2\delta)(T+M)+2d^{\frac{1}{2}}(g,\tilde{g})C(m)(T+M)\Big],
\end{eqnarray*}
where $\mathcal{E}_0(\gamma,\delta), \mathcal{E}_t(\gamma,\delta)\rightarrow 0$ as $\gamma, \delta\rightarrow 0$, the constant $N_0$ is independent of $\gamma,\delta,\lambda$, and $R_{\lambda}$ is defined by
\[
R_{\lambda}=\sup\{R\in [0,\infty]: |a(r)-\tilde{a}(r)|\leq \lambda, \forall |r|< R\}.
\]

\end{description}
\end{thm}
\begin{proof}
Let $\rho_{\gamma}, \psi_{\delta}$ be approximations to the identity on $\mathbb{T}^N$ and $\mathbb{R}$, respectively. That is, let $\rho\in C^{\infty}(\mathbb{T}^N)$, $\psi\in C^{\infty}_c(\mathbb{R})$ be symmetric nonnegative functions such as $\int_{\mathbb{T}^N}\rho =1$, $\int_{\mathbb{R}}\psi =1$ and supp$\psi \subset (-1,1)$. We define
\[
\rho_{\gamma}(x)=\frac{1}{\gamma^N}\rho\Big(\frac{x}{\gamma}\Big), \quad \psi_{\delta}(\xi)=\frac{1}{\delta}\psi\Big(\frac{\xi}{\delta}\Big).
\]

Letting $\rho:=\rho_{\gamma}(x-y)$ and $\psi:=\psi_{\delta}(\xi-\zeta)$ in Proposition \ref{prp-1}, we get from (\ref{P-7}) that
\begin{eqnarray}\notag
 &&\int_{(\mathbb{T}^N)^2}\int_{\mathbb{R}^2}\rho_\gamma (x-y)\psi_{\delta}(\xi-\zeta)(f^{\pm}_1(x,t,\xi)\bar{f}^{\pm}_2(y,t,\zeta)+\bar{f}^{\pm}_1(x,t,\xi)f^{\pm}_2(y,t,\zeta))d\xi d\zeta dxdy\\
\notag
&\leq& \int_{(\mathbb{T}^N)^2}\int_{\mathbb{R}^2}\rho_\gamma (x-y)\psi_{\delta}(\xi-\zeta)(f_{1,0}(x,\xi)\bar{f}_{2,0}(y,\zeta)+\bar{f}_{1,0}(x,\xi)f_{2,0}(y,\zeta))d\xi d\zeta dxdy\\
\label{equ-39}
&&\ +2R_1+2R_2+2R_3,
\end{eqnarray}
where $R_1$, $R_2$, $R_3$ in (\ref{equ-39}) are the corresponding $K_1$, $K_2, K_3$ in the statement of Proposition \ref{prp-1} with $\rho$, $\psi$ replaced by $\rho_{\gamma}$, $\psi_{\delta}$, respectively. For simplicity, we still denote by $l(\xi,\zeta), \tilde{l}(\xi,\zeta), \chi_1(\xi,\zeta)$ with $\rho$, $\psi$ replaced by $\rho_{\gamma}$, $\psi_{\delta}$, respectively.

For any $t\in [0,1]$, define the error term
\begin{eqnarray}\notag
&&\mathcal{E}_t(\gamma,\delta)\\ \notag
&:=&\int_{(\mathbb{T}^N)^2}\int_{\mathbb{R}^2}(f^{\pm}_1(x,t,\xi)\bar{f}^{\pm}_2(y,t,\zeta)+\bar{f}^{\pm}_1(x,t,\xi){f}^{\pm}_2(y,t,\zeta))\rho_{\gamma}(x-y)\psi_{\delta}(\xi-\zeta)dxdyd\xi d\zeta\\
\label{qq-3}
&&-\int_{\mathbb{T}^N}\int_{\mathbb{R}}(f^{\pm}_1(x,t,\xi)\bar{f}^{\pm}_2(x,t,\xi)+\bar{f}^{\pm}_1(x,t,\xi)f^{\pm}_2(x,t,\xi))d\xi dx.
\end{eqnarray}
By utilizing $\int_{\mathbb{R}}\psi_{\delta}(\xi-\zeta)d\zeta=1$ and $\int^{\delta}_0\psi_{\delta}(\zeta')d\zeta'=\int^{0}_{-\delta}\psi_{\delta}(\zeta')d\zeta'=\frac{1}{2}$, we get
\begin{eqnarray}\notag
&&\Big|\int_{(\mathbb{T}^N)^2}\int_{\mathbb{R}}\rho_{\gamma}(x-y)f^{\pm}_1(x,t,\xi)\bar{f}^{\pm}_2(y,t,\xi)d\xi dxdy\\ \notag
&&\ -\int_{(\mathbb{T}^N)^2}\int_{\mathbb{R}^2}f^{\pm}_1(x,t,\xi)\bar{f}^{\pm}_2(y,t,\zeta)\rho_{\gamma}(x-y)\psi_{\delta}(\xi-\zeta)dxdyd\xi d\zeta\Big|\\ \notag
&=&\Big|\int_{(\mathbb{T}^N)^2}\rho_{\gamma}(x-y)\int_{\mathbb{R}}f^{\pm}_1(x,t,\xi)\int_{\mathbb{R}}\psi_{\delta}(\xi-\zeta)(\bar{f}^{\pm}_2(y,t,\xi)-\bar{f}^{\pm}_2(y,t,\zeta))d\zeta d\xi dxdy\Big|\\ \notag
&\leq&\int_{(\mathbb{T}^N)^2}\rho_{\gamma}(x-y)\int_{\mathbb{R}}f^{\pm}_1(x,t,\xi)\int^{\xi}_{\xi-\delta}\psi_{\delta}(\xi-\zeta)(\bar{f}^{\pm}_2(y,t,\xi)-\bar{f}^{\pm}_2(y,t,\zeta)) d\zeta d\xi dxdy\\ \notag
&&\ +\int_{(\mathbb{T}^N)^2}\int_{\mathbb{R}}\rho_{\gamma}(x-y)f^{\pm}_1(x,t,\xi)\int^{\xi+\delta}_{\xi}\psi_{\delta}(\xi-\zeta)(\bar{f}^{\pm}_2(y,t,\zeta)-\bar{f}^{\pm}_2(y,t,\xi)) d\zeta d\xi dxdy\\ \notag
&\leq&\int_{(\mathbb{T}^N)^2}\rho_{\gamma}(x-y)\int^{\delta}_0\psi_{\delta}(\zeta')\int_{\mathbb{R}}f^{\pm}_1(x,t,\xi)(\bar{f}^{\pm}_2(y,t,\xi)-\bar{f}^{\pm}_2(y,t,\xi-\zeta')) d\xi d\zeta' dxdy\\ \notag
&&\ +\int_{(\mathbb{T}^N)^2}\rho_{\gamma}(x-y)\int^{0}_{-\delta}\psi_{\delta}(\zeta')\int_{\mathbb{R}}f^{\pm}_1(x,t,\xi)(\bar{f}^{\pm}_2(y,t,\xi-\zeta')-\bar{f}^{\pm}_2(y,t,\xi)) d\xi d\zeta' dxdy\\ \notag
&\leq&\delta\int_{(\mathbb{T}^N)^2}\rho_{\gamma}(x-y)\Big(\int^{\delta}_0\psi_{\delta}(\zeta')d\zeta'\Big)\Big(\int_{\mathbb{R}}\partial_{\xi'}\bar{f}^{\pm}_2(y,t,\xi') d\xi'\Big) dxdy\\ \notag
&&\ +\delta\int_{(\mathbb{T}^N)^2}\rho_{\gamma}(x-y)\Big(\int^{0}_{-\delta}\psi_{\delta}(\zeta')d\zeta'\Big)\Big(\int_{\mathbb{R}}\partial_{\xi'}\bar{f}^{\pm}_2(y,t,\xi') d\xi'\Big) dxdy\\
\label{qeq-14}
&\leq & \frac{1}{2}\delta+\frac{1}{2}\delta=\delta,
\end{eqnarray}
where we have taken into account the facts that  $\bar{f}^{\pm}_2(y,t,\xi)$ is increasing in $\xi$, $f^{\pm}_1(x,t,\xi)\leq 1$ and $\int_{\mathbb{R}}\partial_{\xi}\bar{f}^{\pm}_2(y,t,\xi)d\xi=\int_{\mathbb{R}}\nu^2_{y,t}(d\xi)=1$.

Similarly,
\begin{eqnarray}\notag
&&\Big|\int_{(\mathbb{T}^N)^2}\int_{\mathbb{R}}\rho_{\gamma}(x-y)\bar{f}^{\pm}_1(x,t,\xi){f}^{\pm}_2(y,t,\xi)d\xi dxdy\\ \notag
&&\ -\int_{(\mathbb{T}^N)^2}\int_{\mathbb{R}^2}\bar{f}^{\pm}_1(x,t,\xi){f}^{\pm}_2(y,t,\zeta)\rho_{\gamma}(x-y)\psi_{\delta}(\xi-\zeta)dxdyd\xi d\zeta\Big|\\ \notag
&=&\Big|\int_{(\mathbb{T}^N)^2}\rho_{\gamma}(x-y)\int_{\mathbb{R}}\bar{f}^{\pm}_1(x,t,\xi)\int_{\mathbb{R}}\psi_{\delta}(\xi-\zeta)({f}^{\pm}_2(y,t,\xi)-{f}^{\pm}_2(y,t,\zeta))d\zeta d\xi dxdy\Big|\\ \notag
&\leq&\int_{(\mathbb{T}^N)^2}\rho_{\gamma}(x-y)\int_{\mathbb{R}}\bar{f}^{\pm}_1(x,s,\xi)\int^{\xi}_{\xi-\delta}\psi_{\delta}(\xi-\zeta)({f}^{\pm}_2(y,t,\zeta)-{f}^{\pm}_2(y,t,\xi)) d\zeta d\xi dxdy\\ \notag
&&\ +\int_{(\mathbb{T}^N)^2}\int_{\mathbb{R}}\rho_{\gamma}(x-y)\bar{f}^{\pm}_1(x,s,\xi)\int^{\xi+\delta}_{\xi}\psi_{\delta}(\xi-\zeta)({f}^{\pm}_2(y,t,\xi)-{f}^{\pm}_2(y,t,\zeta)) d\zeta d\xi dxdy\\ \notag
&\leq&\int_{(\mathbb{T}^N)^2}\rho_{\gamma}(x-y)\int^{\delta}_0\psi_{\delta}(\zeta')\int_{\mathbb{R}}\bar{f}^{\pm}_1(x,t,\xi)({f}^{\pm}_2(y,t,\xi-\zeta')-{f}^{\pm}_2(y,t,\xi)) d\xi d\zeta' dxdy\\ \notag
&&\ +\int_{(\mathbb{T}^N)^2}\rho_{\gamma}(x-y)\int^{0}_{-\delta}\psi_{\delta}(\zeta')\int_{\mathbb{R}}\bar{f}^{\pm}_1(x,t,\xi)({f}^{\pm}_2(y,t,\xi)-{f}^{\pm}_2(y,t,\xi-\zeta')) d\xi d\zeta' dxdy\\
\label{qeq-14-1}
&\leq & \frac{1}{2}\delta+\frac{1}{2}\delta=\delta.
\end{eqnarray}
Moreover, it follows that
\begin{eqnarray*}
&&\Big|\int_{(\mathbb{T}^N)^2}\int_{\mathbb{R}}\rho_{\gamma}(x-y)f^{\pm}_1(x,t,\xi)\bar{f}^{\pm}_2(y,t,\xi)d\xi dxdy-\int_{\mathbb{T}^N}\int_{\mathbb{R}}f^{\pm}_1(x,t,\xi)\bar{f}^{\pm}_2(x,t,\xi)d\xi dx\Big|\\ \notag
&=&\Big|\int_{\mathbb{T}^N}\int_{|z|<\gamma}\int_{\mathbb{R}}\rho_{\gamma}(z)f^{\pm}_1(x,t,\xi)\bar{f}^{\pm}_2(x-z,t,\xi)d\xi dxdz-\int_{\mathbb{T}^N}\int_{|z|<\gamma}\int_{\mathbb{R}}\rho_{\gamma}(z)f^{\pm}_1(x,t,\xi)\bar{f}^{\pm}_2(x,t,\xi)d\xi dxdz\Big|\\ \notag
&\leq&\sup_{|z|<\gamma}\int_{\mathbb{T}^N}\int_{\mathbb{R}}f^{\pm}_1(x,t,\xi)|\bar{f}^{\pm}_2(x-z,t,\xi)-\bar{f}^{\pm}_2(x,t,\xi)|d\xi dx\\ \notag
&\leq& \sup_{|z|<\gamma}\int_{\mathbb{T}^N}\int_{\mathbb{R}}|-f^{\pm}_2(x-z,t,\xi)+I_{0>\xi}-I_{0>\xi}+f^{\pm}_2(x,t,\xi)|d\xi dx\\
&=& \sup_{|z|<\gamma}\int_{\mathbb{T}^N}\int_{\mathbb{R}}|\Lambda_{f^{\pm}_2}(x-z,t,\xi)-\Lambda_{f^{\pm}_2}(x,t,\xi)|d\xi dx.
\end{eqnarray*}
In view of (\ref{qq-6}),  we have
\begin{eqnarray}\label{qq-1}
\lim_{\gamma\rightarrow 0}\Big|\int_{(\mathbb{T}^N)^2}\int_{\mathbb{R}}\rho_{\gamma}(x-y)f^{\pm}_1(x,t,\xi)\bar{f}^{\pm}_2(y,t,\xi)d\xi dxdy-\int_{\mathbb{T}^N}\int_{\mathbb{R}}f^{\pm}_1(x,t,\xi)\bar{f}^{\pm}_2(x,t,\xi)d\xi dx\Big|= 0.
\end{eqnarray}
Similarly, it holds that
\begin{eqnarray}\label{qq-2}
\lim_{\gamma\rightarrow 0}\Big|\int_{(\mathbb{T}^N)^2}\int_{\mathbb{R}}\rho_{\gamma}(x-y)\bar{f}^{\pm}_1(x,t,\xi)f^{\pm}_2(y,t,\xi)d\xi dxdy-\int_{\mathbb{T}^N}\int_{\mathbb{R}}\bar{f}^{\pm}_1(x,t,\xi)f^{\pm}_2(x,t,\xi)d\xi dx\Big|=0.
\end{eqnarray}
Based on (\ref{e-23})-(\ref{qq-2}), we have
\begin{eqnarray}\label{qq-4}
\lim_{\gamma, \delta\rightarrow 0}\mathcal{E}_t(\gamma,\delta)=0.
\end{eqnarray}
In particular, when $t=0$, it holds that
\begin{eqnarray}\label{qq-5}
\lim_{\gamma, \delta\rightarrow 0}\mathcal{E}_0(\gamma,\delta)=0.
\end{eqnarray}

In the following, we devote to making estimates of $R_1$, $R_2$ and $R_3$.
We begin with the estimates of $R_1$. By the definition of $l(\xi,\zeta)$, we deduce that
\begin{eqnarray}\notag
&&-\int^t_0\int_{(\mathbb{T}^N)^2}\partial^2_{x_iy_i}\rho_{\gamma}(x-y)\int_{\mathbb{R}^2}\int^{\xi}_{\zeta}\int^{\xi'}_{\zeta}\psi_{\delta}(\xi'-r)a^2(\xi')drd\xi'd\nu^1_{x,s}\otimes d \nu^2_{y,s}(\xi,\zeta) dxdyds\\ \notag
&=&-\int^t_0\int_{(\mathbb{T}^N)^2}\partial^2_{x_iy_i}\rho_{\gamma}(x-y)\int_{\zeta\leq \xi}\int^{\xi}_{\zeta}\int^{\xi}_{\zeta}I_{r\leq \xi'}\psi_{\delta}(\xi'-r)a^2(\xi')drd\xi'd\nu^1_{x,s}\otimes d \nu^2_{y,s}(\xi,\zeta) dxdyds\\
\label{e-9}
&&-\int^t_0\int_{(\mathbb{T}^N)^2}\partial^2_{x_iy_i}\rho_{\gamma}(x-y)\int_{\zeta\geq \xi}\int^{\zeta}_{\xi}\int^{\zeta}_{\xi}I_{r\geq \xi'}\psi_{\delta}(\xi'-r)a^2(\xi')drd\xi'd\nu^1_{x,s}\otimes d \nu^2_{y,s}(\xi,\zeta) dxdyds.
\end{eqnarray}
Symmetrically, one has
\begin{eqnarray}\notag
&&-\int^t_0\int_{(\mathbb{T}^N)^2}\int_{\mathbb{R}^2}\partial^2_{x_iy_i}\rho_{\gamma}(x-y)\tilde{l}(\xi,\zeta)d\nu^1_{x,s}\otimes d \nu^2_{y,s}(\xi,\zeta) dxdyds\\ \notag
&=& -\int^t_0\int_{(\mathbb{T}^N)^2}\partial^2_{x_iy_i}\rho_{\gamma}(x-y)\int_{\zeta\leq \xi}\int^{\xi}_{\zeta}\int^{\xi}_{\zeta}I_{r\leq \xi'}\psi_{\delta}(\xi'-r)\tilde{a}^2(r)drd\xi'd\nu^1_{x,s}\otimes d \nu^2_{y,s}(\xi,\zeta) dxdyds\\
\label{e-7}
&&-\int^t_0\int_{(\mathbb{T}^N)^2}\partial^2_{x_iy_i}\rho_{\gamma}(x-y)\int_{\zeta\geq \xi}\int^{\zeta}_{\xi}\int^{\zeta}_{\xi}I_{r\geq \xi'}\psi_{\delta}(\xi'-r)\tilde{a}^2(r)drd\xi'd\nu^1_{x,s}\otimes d \nu^2_{y,s}(\xi,\zeta) dxdyds.
\end{eqnarray}
Adding (\ref{e-9}) and (\ref{e-7}) together, we get
\begin{eqnarray}\notag
&&R_1\\ \notag
&=&-\int^t_0\int_{(\mathbb{T}^N)^2}\partial^2_{x_iy_i}\rho_{\gamma}(x-y)\int_{\zeta\leq \xi}\int^{\xi}_{\zeta}\int^{\xi}_{\zeta}I_{\xi'\geq r}\psi_{\delta}(\xi'-r)(a^2(\xi')+\tilde{a}^2(r))drd\xi'd\nu^1_{x,s}\otimes d \nu^2_{y,s}(\xi,\zeta) dxdyds\\
\label{e-10}
&&-\int^t_0\int_{(\mathbb{T}^N)^2}\partial^2_{x_iy_i}\rho_{\gamma}(x-y)\int_{\zeta\geq \xi}\int^{\zeta}_{\xi}\int^{\zeta}_{\xi}I_{\xi'\leq r}\psi_{\delta}(\xi'-r)(a^2(\xi')+\tilde{a}^2(r))drd\xi'd\nu^1_{x,s}\otimes d \nu^2_{y,s}(\xi,\zeta) dxdyds.
\end{eqnarray}
By the definition of $\Psi$ and $\tilde{\Psi}$, we have
\begin{eqnarray}\notag
R_2&=&-2\int^t_0\int_{(\mathbb{T}^N)^2}\int_{\mathbb{R}^2}\nabla_x \Psi(\xi)\cdot\nabla_y \tilde{\Psi}(\zeta) \rho_{\gamma}(x-y)\psi_{\delta}(\xi-\zeta) d\nu^1_{x,s}\otimes d \nu^2_{y,s}(\xi,\zeta)dxdyds\\ \notag
&=& -2\int^t_0\int_{(\mathbb{T}^N)^2}\int_{\mathbb{R}^2}\partial_{x_i} \Psi(\xi)\partial_{y_i} \tilde{\Psi}(\zeta) \rho_{\gamma}(x-y)\psi_{\delta}(\xi-\zeta) d\nu^1_{x,s}\otimes d \nu^2_{y,s}(\xi,\zeta)dxdyds\\ \notag
&=&-2\int^t_0\int_{(\mathbb{T}^N)^2}\int_{\mathbb{R}^2}\partial_{x_i}\Psi(\xi)\rho_{\gamma}(x-y)\partial_{y_i}\int^{\zeta}_{\xi} \psi_{\delta}(\xi-r)\tilde{a}(r)dr d\nu^1_{x,s}\otimes d \nu^2_{y,s}(\xi,\zeta) dxdyds\\ \notag
&=&2\int^t_0\int_{(\mathbb{T}^N)^2}\int_{\mathbb{R}^2}\partial_{x_i}\Psi(\xi)\partial_{y_i}\rho_{\gamma}(x-y)\int^{\zeta}_{\xi} \psi_{\delta}(\xi-r)\tilde{a}(r)dr d\nu^1_{x,s}\otimes d \nu^2_{y,s}(\xi,\zeta) dxdyds\\ \notag
&=&-2\int^t_0\int_{(\mathbb{T}^N)^2}\int_{\mathbb{R}^2}\partial^2_{x_iy_i}\rho_{\gamma}(x-y)\int^{\xi}_{\zeta}\int^{\zeta}_{\xi'} \psi_{\delta}(\xi'-r)\tilde{a}(r)dr a(\xi')d\xi' d\nu^1_{x,s}\otimes d \nu^2_{y,s}(\xi,\zeta) dxdyds\\ \notag
&=&2\int^t_0\int_{(\mathbb{T}^N)^2}\partial^2_{x_iy_i}\rho_{\gamma}(x-y)\int_{\zeta\leq\xi}\int^{\xi}_{\zeta}\int^{\xi}_{\zeta} I_{r\leq\xi'}\psi_{\delta}(\xi'-r)\tilde{a}(r)a(\xi')drd\xi' d\nu^1_{x,s}\otimes d \nu^2_{y,s}(\xi,\zeta) dxdyds\\ \label{e-11}
&&+2\int^t_0\int_{(\mathbb{T}^N)^2}\partial^2_{x_iy_i}\rho_{\gamma}(x-y)\int_{\zeta\geq\xi}\int^{\zeta}_{\xi}\int^{\zeta}_{\xi} I_{r\geq\xi'}\psi_{\delta}(\xi'-r)\tilde{a}(r)a(\xi')drd\xi' d\nu^1_{x,s}\otimes d \nu^2_{y,s}(\xi,\zeta) dxdyds.
\end{eqnarray}
Combining (\ref{e-10}) and (\ref{e-11}), we get
\begin{eqnarray}\notag
&&R_1+R_2\\ \notag
&=&
 \int^t_0\int_{(\mathbb{T}^N)^2}\partial^2_{x_iy_i}\rho_{\gamma}(x-y)\int_{\zeta\leq\xi}\int^{\xi}_{\zeta}\int^{\xi}_{\zeta} I_{r\leq\xi'}\psi_{\delta}(\xi'-r)(2\tilde{a}(r)a(\xi')-a^2(\xi')-\tilde{a}^2(r))drd\xi' d\nu^1_{x,s}\otimes d \nu^2_{y,s}(\xi,\zeta) dxdyds\\ \notag
&&+\int^t_0\int_{(\mathbb{T}^N)^2}\partial^2_{x_iy_i}\rho_{\gamma}(x-y)\int_{\zeta\geq\xi}\int^{\zeta}_{\xi}\int^{\zeta}_{\xi} I_{r\geq\xi'}\psi_{\delta}(\xi'-r)(2\tilde{a}(r)a(\xi')-a^2(\xi')-\tilde{a}^2(r))drd\xi' d\nu^1_{x,s}\otimes d \nu^2_{y,s}(\xi,\zeta) dxdyds\\
\label{e-12}
&\leq& \int^t_0\int_{(\mathbb{T}^N)^2}\partial^2_{x_iy_i}\rho_{\gamma}(x-y)\int_{\zeta\geq\xi}\int^{\zeta}_{\xi}\int^{\zeta}_{\xi} I_{r\geq\xi'}\psi_{\delta}(\xi'-r)|a(r)-a(\xi')|^2drd\xi' d\nu^1_{x,s}\otimes d \nu^2_{y,s}(\xi,\zeta) dxdyds.
\end{eqnarray}
By the similar argument as in the proof of (4.18) in \cite{DGG}, we get for all $\alpha\in (0,1\wedge \frac{m}{2})$,  there exists a constant $N_0$ independent of $\gamma, \delta, \lambda$ such that
\begin{eqnarray}\notag
2(R_1+R_2)&\leq& N_0\gamma^{-2}(\lambda^2+\delta^{2\alpha})\Big[1+\int^T_0\int_{(\mathbb{T}^N)^2}\int_{\mathbb{R}^2}(|\xi|^m+|\zeta|^m)d\nu^1_{x,s}\otimes d \nu^2_{y,s}(\xi,\zeta) dxdyds\Big]\\
\notag
&& +N_0\gamma^{-2}\int^T_0\int_{(\mathbb{T}^N)^2}\int_{\mathbb{R}^2}I_{|\xi|\geq R_{\lambda}}(1+|\xi|)^m d\nu^1_{x,s}\otimes d \nu^2_{y,s}(\xi,\zeta) dxdyds\\
\label{e-13}
&&+N_0\gamma^{-2}\int^T_0\int_{(\mathbb{T}^N)^2}\int_{\mathbb{R}^2}I_{|\zeta|\geq R_{\lambda}}(1+|\zeta|)^m d\nu^1_{x,s}\otimes d \nu^2_{y,s}(\xi,\zeta) dxdyds.
\end{eqnarray}

For $R_3$, it can be estimated as follows:
\begin{eqnarray}\notag
R_3
&\leq& \int^t_0\int_{(\mathbb{T}^N)^2}\rho_{\gamma}(x-y)\int_{\mathbb{R}^2}{\gamma}_1(\xi,\zeta)\sum_{k\geq 1}| g^k(x,\xi)-\tilde{g}^{k}(y,\zeta)||h_k(s)|d \nu^1_{x,s}\otimes d\nu^2_{y,s}(\xi,\zeta) dxdyds\\ \notag
&\leq& \int^t_0\int_{(\mathbb{T}^N)^2}\rho_{\gamma}(x-y)\int_{\mathbb{R}^2}{\gamma}_1(\xi,\zeta)\sum_{k\geq 1}| g^k(x,\xi)- g^k(y,\zeta)||h_k(s)|d \nu^1_{x,s}\otimes d\nu^2_{y,s}(\xi,\zeta) dxdyds\\ \notag
&&+\int^t_0\int_{(\mathbb{T}^N)^2}\rho_{\gamma}(x-y)\int_{\mathbb{R}^2}{\gamma}_1(\xi,\zeta)\sum_{k\geq 1}| g^k(y,\zeta)-\tilde{g}^{k}(y,\zeta)||h_k(s)|d \nu^1_{x,s}\otimes d\nu^2_{y,s}(\xi,\zeta) dxdyds\\
\label{e-16}
&:=& L_1+L_2.
\end{eqnarray}
By H\"{o}lder inequality and  (\ref{equ-29}), we deduce that
\begin{eqnarray*}
L_1&\leq& \int^t_0\int_{(\mathbb{T}^N)^2}\rho_{\gamma}(x-y)
\int_{\mathbb{R}^2}{\gamma}_1(\xi,\zeta)\Big(\sum_{k\geq 1}| g^k(x,\xi)- g^k(y,\zeta)|^2\Big)^{\frac{1}{2}}
\Big(\sum_{k\geq 1}|h_k(s)|^2\Big)^{\frac{1}{2}}
d \nu^1_{x,s}\otimes d\nu^2_{y,s}(\xi,\zeta) dxdyds\\
&\leq& K\int^t_0|h(s)|_U\int_{(\mathbb{T}^N)^2}\rho_{\gamma}(x-y)|x-y|\int_{\mathbb{R}^2}{\gamma}_1(\xi,\zeta)d \nu^1_{x,s}\otimes d\nu^2_{y,s}(\xi,\zeta) dxdyds\\
&&+ K\int^t_0|h(s)|_U\int_{(\mathbb{T}^N)^2}\rho_{\gamma}(x-y)\int_{\mathbb{R}^2}{\gamma}_1(\xi,\zeta)|\xi-\zeta|d \nu^1_{x,s}\otimes d\nu^2_{y,s}(\xi,\zeta) dxdyds\\
&=:&L_{1,1}+L_{1,2},
\end{eqnarray*}

Note that
\begin{eqnarray*}
\int_{\mathbb{R}^2}{\gamma}_1(\xi,\zeta)d \nu^1_{x,s}\otimes d\nu^2_{y,s}(\xi,\zeta)&\leq& 1,
\\
\int_{(\mathbb{T}^N)^2}\rho_{\gamma}(x-y)|x-y|dxdy&\leq&\gamma,
\end{eqnarray*}
it follows that
\begin{eqnarray}\label{P-14}
L_{1,1}\leq K \gamma\int^t_0|h(s)|_Uds
\leq K(T+M)\gamma.
\end{eqnarray}
To dealing with the term $L_{1,2}$, we adopt the similar method as \cite{DWZZ}.
Taking into account $\nu^1_{x,s}(\xi)=-\partial_{\xi}f^{\pm}_1(s,x,\xi)=\partial_{\xi}\bar{f}^{\pm}_1(s,x,\xi)$ and $\nu^2_{y,s}(\zeta)=\partial_{\zeta}\bar{f}^{\pm}_2(s,y,\zeta)=-\partial_{\zeta}f^{\pm}_2(s,y,\zeta)$, it follows that
\begin{eqnarray*} \notag
L_{1,2}&\leq&  K\int^t_0|h(s)|_U\int_{(\mathbb{T}^N)^2}\int_{\mathbb{R}^2}\rho_{\gamma}(x-y)|\xi-\zeta|d \nu^1_{x,s}\otimes d\nu^2_{y,s}(\xi,\zeta) dxdyds\\ \notag
&=&  K\int^t_0|h(s)|_U\int_{(\mathbb{T}^N)^2}\rho_{\gamma}(x-y)(\xi-\zeta)^{+}d \nu^1_{x,s}\otimes d\nu^2_{y,s}(\xi,\zeta) dxdyds\\ \notag
&&\ + K\int^t_0|h(s)|_U\int_{(\mathbb{T}^N)^2}\rho_{\gamma}(x-y)(\xi-\zeta)^{-}d \nu^1_{x,s}\otimes d\nu^2_{y,s}(\xi,\zeta) dxdyds\\
&=& K\int^t_0|h(s)|_U\int_{(\mathbb{T}^N)^2}\int_{\mathbb{R}}\rho_{\gamma}(x-y)(f^{\pm}_1(x,s,\xi)\bar{f}^{\pm}_2(y,s,\xi)+\bar{f}^{\pm}_1(x,s,\xi){f}^{\pm}_2(y,s,\xi))d\xi dxdyds,
\end{eqnarray*}
where we have used $\delta_{\xi=\zeta}=-\partial_{\xi}\partial_{\zeta}(\xi-\zeta)^{+}=-\partial_{\xi}\partial_{\zeta}(\xi-\zeta)^{-}$.

Then, we deduce from (\ref{qeq-14}) and (\ref{qeq-14-1}) that
\begin{eqnarray}\notag
L_{1,2}&\leq& 2\delta K\int^t_0|h(s)|_Uds\\ \notag
&& + K\int^t_0|h(s)|_U\int_{(\mathbb{T}^N)^2}\int_{\mathbb{R}^2}(f^{\pm}_1\bar{f}^{\pm}_2+\bar{f}^{\pm}_1{f}^{\pm}_2)\rho_{\gamma}(x-y)\psi_{\delta}(\xi-\zeta)dxdyd\xi d\zeta ds\\
\notag
&\leq& 2\delta K(T+M)\\
\label{qeq-16}
&&+ K\int^t_0|h(s)|_U\int_{(\mathbb{T}^N)^2}\int_{\mathbb{R}^2}(f^{\pm}_1\bar{f}^{\pm}_2+\bar{f}^{\pm}_1{f}^{\pm}_2)\rho_{\gamma}(x-y)\psi_{\delta}(\xi-\zeta)dxdyd\xi d\zeta ds.
\end{eqnarray}
Hence, combining (\ref{P-14}) and (\ref{qeq-16}), we get
\begin{eqnarray}\notag
L_1&\leq&  K(\gamma+2\delta)(T+M)\\
\label{qeq-17}
&&\ + K\int^t_0|h(s)|_U\int_{(\mathbb{T}^N)^2}\int_{\mathbb{R}^2}(f^{\pm}_1\bar{f}^{\pm}_2+\bar{f}^{\pm}_1{f}^{\pm}_2)\rho_{\gamma}(x-y)\psi_{\delta}(\xi-\zeta)dxdyd\xi d\zeta ds.
\end{eqnarray}

Utilizing H\"{o}lder inequality, it yields
\begin{eqnarray}\notag
L_2&\leq&\int^t_0\int_{(\mathbb{T}^N)^2}\rho_{\gamma}(x-y)\int_{\mathbb{R}^2}{\gamma}_1(\xi,\zeta)\Big(\sum_{k\geq 1}| g^k(y,\zeta)-\tilde{g}^{k}(y,\zeta)|^2\Big)^{\frac{1}{2}}\Big(\sum_{k\geq 1}|h_k(s)|^2\Big)^{\frac{1}{2}}d \nu^1_{x,s}\otimes d\nu^2_{y,s}(\xi,\zeta) dxdyds\\ \notag
&\leq& \int^t_0|h(s)|_U\int_{(\mathbb{T}^N)^2}\rho_{\gamma}(x-y)\int_{\mathbb{R}^2}{\gamma}_1(\xi,\zeta)d^{\frac{1}{2}}(g,\tilde{g})(1+|\zeta|)^{\frac{m+1}{2}}d \nu^1_{x,s}\otimes d\nu^2_{y,s}(\xi,\zeta) dxdyds\\ \notag
&\leq& d^{\frac{1}{2}}(g,\tilde{g})\int^t_0|h(s)|_U\int_{(\mathbb{T}^N)^2}\rho_{\gamma}(x-y)\int_{\mathbb{R}^2}(1+|\zeta|)^{\frac{m+1}{2}}d \nu^1_{x,s}\otimes d\nu^2_{y,s}(\xi,\zeta) dxdyds\\ \notag
&\leq& d^{\frac{1}{2}}(g,\tilde{g})C(m)\int^t_0|h(s)|_U\int_{(\mathbb{T}^N)^2}\rho_{\gamma}(x-y)dxdyds\\
\notag
&\leq& d^{\frac{1}{2}}(g,\tilde{g})C(m)\int^t_0|h(s)|_Uds\\
\label{e-15}
&\leq&  d^{\frac{1}{2}}(g,\tilde{g})C(m)(T+M),
\end{eqnarray}
where we have used the property that the measures $\nu^{1}_{x,s}$ and $\nu^2_{y,s}$ vanish at the infinity.

Combining (\ref{e-16}), (\ref{qeq-17}) and (\ref{e-15}), it yields
\begin{eqnarray}\notag
R_3&\leq &  K(\gamma+2\delta)(T+M)+d^{\frac{1}{2}}(g,\tilde{g})C(m)(T+M)\\
\label{e-17}
&& + K\int^t_0|h(s)|_U\int_{(\mathbb{T}^N)^2}\int_{\mathbb{R}^2}(f^{\pm}_1\bar{f}^{\pm}_2+\bar{f}^{\pm}_1{f}^{\pm}_2)\rho_{\gamma}(x-y)\psi_{\delta}(\xi-\zeta)dxdyd\xi d\zeta ds.
\end{eqnarray}
Collecting (\ref{equ-39}), (\ref{e-13}) and (\ref{e-17}), we conclude that
\begin{eqnarray*}\notag
 &&\int_{(\mathbb{T}^N)^2}\int_{\mathbb{R}^2}\rho_\gamma (x-y)\psi_{\delta}(\xi-\zeta)(f^{\pm}_1(x,t,\xi)\bar{f}^{\pm}_2(y,t,\zeta)+\bar{f}^{\pm}_1(x,t,\xi){f}^{\pm}_2)(y,t,\zeta)d\xi d\zeta dxdy\\ \notag
&\leq& \int_{(\mathbb{T}^N)^2}\int_{\mathbb{R}^2}\rho_\gamma (x-y)\psi_{\delta}(\xi-\zeta)(f_{1,0}(x,\xi)\bar{f}_{2,0}(y,\zeta)+\bar{f}_{1,0}(x,\xi){f}_{2,0}(y,\zeta))d\xi d\zeta dxdy\\ \notag
&&+N_0\gamma^{-2}(\lambda^2+\delta^{2\alpha})\Big[1+\int^T_0\int_{(\mathbb{T}^N)^2}\int_{\mathbb{R}^2}(|\xi|^m+|\zeta|^m)d\nu^1_{x,s}\otimes d \nu^2_{y,s}(\xi,\zeta) dxdyds\Big]\\
\notag
&& +N_0\gamma^{-2}\int^T_0\int_{(\mathbb{T}^N)^2}\int_{\mathbb{R}^2}I_{|\xi|\geq R_{\lambda}}(1+|\xi|)^m d\nu^1_{x,s}\otimes d \nu^2_{y,s}(\xi,\zeta) dxdyds\\
\notag
&&+N_0\gamma^{-2}\int^T_0\int_{(\mathbb{T}^N)^2}\int_{\mathbb{R}^2}I_{|\zeta|\geq R_{\lambda}}(1+|\zeta|)^m d\nu^1_{x,s}\otimes d \nu^2_{y,s}(\xi,\zeta) dxdyds \\
&& +2 K(\gamma+2\delta)(T+M)+2d^{\frac{1}{2}}(g,\tilde{g})C(m)(T+M)\\
&& +2 K\int^t_0|h(s)|_U\int_{(\mathbb{T}^N)^2}\int_{\mathbb{R}^2}(f^{\pm}_1\bar{f}^{\pm}_2+\bar{f}^{\pm}_1{f}^{\pm}_2)\rho_{\gamma}(x-y)\psi_{\delta}(\xi-\zeta)dxdyd\xi d\zeta ds\\
&\leq& \int_{\mathbb{T}^N}\int_{\mathbb{R}}(f_{1,0}\bar{f}_{2,0}+\bar{f}_{1,0}{f}_{2,0})d\xi dx+\mathcal{E}_0(\gamma,\delta)\\
&& +N_0\gamma^{-2}(\lambda^2+\delta^{2\alpha})\Big[1+\int^T_0\int_{(\mathbb{T}^N)^2}\int_{\mathbb{R}^2}(|\xi|^m+|\zeta|^m)d\nu^1_{x,s}\otimes d \nu^2_{y,s}(\xi,\zeta) dxdyds\Big]\\
\notag
&& +N_0\gamma^{-2}\int^T_0\int_{(\mathbb{T}^N)^2}\int_{\mathbb{R}^2}I_{|\xi|\geq R_{\lambda}}(1+|\xi|)^m d\nu^1_{x,s}\otimes d \nu^2_{y,s}(\xi,\zeta) dxdyds\\
&&+N_0\gamma^{-2}\int^T_0\int_{(\mathbb{T}^N)^2}\int_{\mathbb{R}^2}I_{|\zeta|\geq R_{\lambda}}(1+|\zeta|)^m d\nu^1_{x,s}\otimes d \nu^2_{y,s}(\xi,\zeta) dxdyds \\
&& +2 K(\gamma+2\delta)(T+M)+2d^{\frac{1}{2}}(g,\tilde{g})C(m)(T+M)\\
&& +2 K\int^t_0|h(s)|_U\int_{(\mathbb{T}^N)^2}\int_{\mathbb{R}^2}(f^{\pm}_1\bar{f}^{\pm}_2+\bar{f}^{\pm}_1{f}^{\pm}_2)\rho_{\gamma}(x-y)\psi_{\delta}(\xi-\zeta)dxdyd\xi d\zeta ds,
\end{eqnarray*}
where $\mathcal{E}_0(\gamma,\delta)$ is defined by (\ref{qq-3}).
%
%

Utilizing Gronwall inequality and by $\int^t_0|h(s)|_Uds\leq (T+M)/2$, we obtain
\begin{eqnarray}\notag
 &&\int_{(\mathbb{T}^N)^2}\int_{\mathbb{R}^2}\rho_\gamma (x-y)\psi_{\delta}(\xi-\zeta)(f^{\pm}_1(x,t,\xi)\bar{f}^{\pm}_2(y,t,\zeta)+\bar{f}^{\pm}_1(x,t,\xi){f}^{\pm}_2(y,t,\zeta))d\xi d\zeta dxdy\\ \notag
&\leq& e^{ K(T+M)}\Big[\int_{\mathbb{T}^N}\int_{\mathbb{R}}(f_{1,0}\bar{f}_{2,0}+\bar{f}_{1,0}{f}_{2,0})d\xi dx+\mathcal{E}_0(\gamma,\delta)\\ \notag
&&+N_0\gamma^{-2}(\lambda^2+\delta^{2\alpha})\Big[1+\int^T_0\int_{(\mathbb{T}^N)^2}\int_{\mathbb{R}^2}(|\xi|^m+|\zeta|^m)d\nu^1_{x,s}\otimes d \nu^2_{y,s}(\xi,\zeta) dxdyds\Big]\\
\notag
&& +N_0\gamma^{-2}\int^T_0\int_{(\mathbb{T}^N)^2}\int_{\mathbb{R}^2}I_{|\xi|\geq R_{\lambda}}(1+|\xi|)^m d\nu^1_{x,s}\otimes d \nu^2_{y,s}(\xi,\zeta) dxdyds\\ \notag
&&+N_0\gamma^{-2}\int^T_0\int_{(\mathbb{T}^N)^2}\int_{\mathbb{R}^2}I_{|\zeta|\geq R_{\lambda}}(1+|\zeta|)^m d\nu^1_{x,s}\otimes d \nu^2_{y,s}(\xi,\zeta) dxdyds \\
\label{qeq-18}
&& +2 K(\gamma+2\delta)(T+M)+2d^{\frac{1}{2}}(g,\tilde{g})C(m)(T+M)\Big].
\end{eqnarray}

%

Then, it follows that
\begin{eqnarray}\notag
&&\int_{\mathbb{T}^N}\int_{\mathbb{R}}(f^{\pm}_1(x,t,\xi)\bar{f}^{\pm}_2(x,t,\xi)+\bar{f}^{\pm}_1(x,t,\xi){f}^{\pm}_2(x,t,\xi))dxd\xi\\ \notag
&=& \int_{(\mathbb{T}^N)^2}\int_{\mathbb{R}^2}\rho_\gamma (x-y)\psi_{\delta}(\xi-\zeta)(f^{\pm}_1(x,t,\xi)\bar{f}^{\pm}_2(y,t,\zeta)+\bar{f}^{\pm}_1(x,t,\xi){f}^{\pm}_2(y,t,\zeta))d\xi d\zeta dxdy+\mathcal{E}_t(\gamma,\delta)\\ \notag
&\leq& \mathcal{E}_t(\gamma,\delta)+e^{ K(T+M)}\Big[\int_{\mathbb{T}^N}\int_{\mathbb{R}}(f_{1,0}\bar{f}_{2,0}+\bar{f}_{1,0}{f}_{2,0})d\xi dx+\mathcal{E}_0(\gamma,\delta)\\ \notag
&&+N_0\gamma^{-2}(\lambda^2+\delta^{2\alpha})\Big[1+\int^T_0\int_{(\mathbb{T}^N)^2}\int_{\mathbb{R}^2}(|\xi|^m+|\zeta|^m)d\nu^1_{x,s}\otimes d \nu^2_{y,s}(\xi,\zeta) dxdyds\Big]\\
\notag
&&+N_0\gamma^{-2}\int^T_0\int_{(\mathbb{T}^N)^2}\int_{\mathbb{R}^2}I_{|\xi|\geq R_{\lambda}}(1+|\xi|)^m d\nu^1_{x,s}\otimes d \nu^2_{y,s}(\xi,\zeta) dxdyds\\
\notag
&&+N_0\gamma^{-2}\int^T_0\int_{(\mathbb{T}^N)^2}\int_{\mathbb{R}^2}I_{|\zeta|\geq R_{\lambda}}(1+|\zeta|)^m d\nu^1_{x,s}\otimes d \nu^2_{y,s}(\xi,\zeta) dxdyds \\
\label{equ-39-1}
&& +2 K(\gamma+2\delta)(T+M)+2d^{\frac{1}{2}}(g,\tilde{g})C(m)(T+M)\Big],
\end{eqnarray}
where $\mathcal{E}_0(\gamma,\delta), \mathcal{E}_t(\gamma,\delta)\rightarrow 0$, as $\gamma,\delta\rightarrow 0$.

When $A=\tilde{A}, g=\tilde{g}$, we can take $\lambda=0$ and $R_{\lambda}=\infty$. Then, we deduce from (\ref{equ-39-1}) that
\begin{eqnarray}\notag
&&\int_{\mathbb{T}^N}\int_{\mathbb{R}}(f^{\pm}_1(x,t,\xi)\bar{f}^{\pm}_2(x,t,\xi)+\bar{f}^{\pm}_1(x,t,\xi){f}^{\pm}_2(x,t,\xi))dxd\xi\\ \notag
&\leq& \mathcal{E}_t(\gamma,\delta)+e^{ K(T+M)}\Big[\int_{\mathbb{T}^N}\int_{\mathbb{R}}(f_{1,0}\bar{f}_{2,0}+\bar{f}_{1,0}{f}_{2,0})d\xi dx+\mathcal{E}_0(\gamma,\delta)\\
\notag
&& +N_0C(T)\gamma^{-2}\delta^{2\alpha}+2 K(\gamma+2\delta)(T+M)\Big],
\end{eqnarray}
where we have used the property that the measures $\nu^1_{x,s}$ and $\nu^2_{y,s}$ vanish at the infinity.

Taking $\delta=\gamma^{\frac{3}{2\alpha}}$, we get
\begin{eqnarray}\notag
&&\int_{\mathbb{T}^N}\int_{\mathbb{R}}(f^{\pm}_1(x,t,\xi)\bar{f}^{\pm}_2(x,t,\xi)+\bar{f}^{\pm}_1(x,t,\xi){f}^{\pm}_2(x,t,\xi))dxd\xi\\ \notag
&\leq& \mathcal{E}_t(\gamma,\delta)+e^{ K(T+M)}\Big[\int_{\mathbb{T}^N}\int_{\mathbb{R}}(f_{1,0}\bar{f}_{2,0}+\bar{f}_{1,0}{f}_{2,0})d\xi dx+\mathcal{E}_0(\gamma,\delta)\\
\label{eeee}
&& +N_0C(T)\gamma+2 K(\gamma+2\gamma^{\frac{3}{2\alpha}})(T+M)\Big].
\end{eqnarray}
Letting $\gamma\rightarrow 0$, by (\ref{eeee}), it yields
\begin{eqnarray}\notag
&&\int_{\mathbb{T}^N}\int_{\mathbb{R}}(f^{\pm}_1(x,t,\xi)\bar{f}^{\pm}_2(x,t,\xi)+\bar{f}^{\pm}_1(x,t,\xi){f}^{\pm}_2(x,t,\xi))dxd\xi\\ \label{eqqq}
&\leq& e^{ K(T+M)}\int_{\mathbb{T}^N}\int_{\mathbb{R}}(f_{1,0}\bar{f}_{2,0}+\bar{f}_{1,0}{f}_{2,0})d\xi dx.
\end{eqnarray}
The reduction of generalized solutions to kinetic solutions is the same as the proof of Theorem 15 in \cite{D-V-1}, we therefore omit it here.
Suppose $u:=u^h$ and $\tilde{u}:=\tilde{u}^h$ are kinetic solutions to (\ref{P-2}) and (\ref{e-6}), respectively. By using the following identities
\begin{eqnarray}\label{equ-1}
\int_{\mathbb{R}}I_{u>\xi}\overline{I_{\tilde{u}>\xi}}d\xi=(u-\tilde{u})^+,
\quad
\int_{\mathbb{R}}\overline{I_{u>\xi}}I_{\tilde{u}>\xi}d\xi=(u-\tilde{u})^-,
\end{eqnarray}
we deduce from (\ref{eqqq}) with $f_1=I_{u>\xi}, f_2=I_{\tilde{u}>\xi}, f_{1,0}=I_{u_0>\xi}, f_{2,0}=I_{\tilde{u}_0>\xi}$ that
\begin{eqnarray*}
\|u(t)-\tilde{u}(t)\|_{L^1(\mathbb{T}^N)}
\leq e^{K(T+M)}\|u_0-\tilde{u}_0\|_{L^1(\mathbb{T}^N)}.
\end{eqnarray*}
We complete the proof of (1).

Now, it remains to prove (2). We deduce from (\ref{equ-39-1}) with $f_1=I_{u>\xi}, f_2=I_{\tilde{u}>\xi}, f_{1,0}=I_{u_0>\xi}, f_{2,0}=I_{\tilde{u}_0>\xi}$ that
\begin{eqnarray}\notag
\|u(t)-\tilde{u}(t)\|_{L^1(\mathbb{T}^N)}
&\leq& \mathcal{E}_t(\gamma,\delta)+e^{K(T+M)}\Big[\|u_0-\tilde{u}_0\|_{L^1(\mathbb{T}^N)} +\mathcal{E}_0(\gamma,\delta)\\ \notag
&& +N_0\gamma^{-2}(\lambda^2+\delta^{2\alpha})\Big(1+\|u\|^m_{L^m([0,T]\times \mathbb{T}^N)}+\|\tilde{u}\|^m_{L^m([0,T]\times \mathbb{T}^N)}\Big)\\
\notag
&& +N_0\gamma^{-2}\|I_{|u|\geq R_{\lambda}}(1+|u|)\|^m_{L^m([0,T]\times \mathbb{T}^N)}+N_0\gamma^{-2}\|I_{|\tilde{u}|\geq R_{\lambda}}(1+|\tilde{u}|)\|^m_{L^m([0,T]\times \mathbb{T}^N)} \\
\label{e-18}
&& +2 K(\gamma+2\delta)(T+M)+2d^{\frac{1}{2}}(g,\tilde{g})C(m)(T+M)\Big],
\end{eqnarray}
which implies (2) holds.
\end{proof}

Now, we are in a position to establish the uniqueness.

\begin{thm}\label{thm-8}
Assume $(A, g, u_0)$ satisfy Hypotheses H, then the skeleton equation (\ref{P-2}) has at most one kinetic solution.
\end{thm}
\begin{proof}
Taking $u_0=\tilde{u}_0$ in (\ref{e-20}), we get
\begin{eqnarray}\notag
\|u(t)-\tilde{u}(t)\|_{L^1(\mathbb{T}^N)}
=0 , \quad a.e. \ t\in [0,T],
\end{eqnarray}
which implies the uniqueness of solutions to (\ref{P-2}).
\end{proof}

Now, we devote to proving the existence of kinetic solution to (\ref{P-2}).
\begin{thm}\label{thm-1}
(Existence) Assume $(A, g, u_0)$ satisfy Hypothesis H, then for any $T>0$, (\ref{P-2}) has a kinetic solution $u^h$ on the time interval $[0,T]$.
\end{thm}

\begin{proof}
By the similar method as Proposition \ref{prp-6}, we get the equivalence between entropy solution and kinetic solution of (\ref{P-2}), hence, it suffices to prove the existence of entropy solutions. For technical reasons,
we follow the spirit of Dareiotis et al.  in  \cite{DGG} to introduce the approximations of the coefficients of (\ref{P-2}).
Define
\begin{eqnarray}\label{ee-4}
g_n:=\rho^{\otimes (N+1)}_{1/n}\ast g(\cdot, -n\vee (\cdot\wedge n)), \quad u_{0,n}:=\rho^{\otimes N}_{1/n}\ast (-n\vee (u_0(\cdot)\wedge n)).
\end{eqnarray}
It's easy to verify that if $g$ and $u_0$ satisfy Hypotheses H with a constant $K\geq 1$, then the same holds for $g_n$ and $u_{0,n}$ with constant $2K$. It is also clear that $g_n\in C^{\infty}(\mathbb{T}^N\times \mathbb{R})$ with one-order derivatives bounded by $C(n,K)$.
Also, $u_{0,n}$ is a bounded $C^k(\mathbb{T}^N)-$valued random variable for any $k\in \mathbb{N}$ and
\begin{eqnarray}\label{e-21}
d(g, g_n)\rightarrow 0,\quad \|u_0-u_{0,n}\|^{m+1}_{L^{m+1}(\mathbb{T}^N)}\rightarrow 0, \ {\rm{as}} \ n\rightarrow \infty.
\end{eqnarray}
Now, we focus on the approximation of $A$. Taking a symmetric mollifier $\bar{\rho}_{\theta}$ supported on $[-\theta, \theta]$, for instance, $\bar{\rho}_{\theta}(r):=\int_{\mathbb{R}} \rho_{\theta}(r+s)\rho_{\theta}(s)ds$. Set $\theta_n:=\sup\{\theta\in (0,1]:|a(r)-a(\zeta)|\leq 1/n, \forall |r|\leq 3n, |\zeta-r|\leq 3\theta\}>0$. Then, define
\begin{eqnarray}\label{e-22}
A_n(r)=\int^r_0a^2_n(\zeta)d\zeta, \quad \Psi_n(r)=\int^r_0a_n(s)ds, \quad a_n(r)=\bar{\rho}_{\theta_n}\ast(2/n+a(3\theta_n\vee |r|\wedge3n)).
\end{eqnarray}
Referring to Proposition 5.1 in \cite{DGG}, we know that for all $n$, $a_n\in C^{\infty}(\mathbb{R})$ satisfying $a_n(r)\geq 2/n$,
\begin{eqnarray}\label{e-29}
\sup_{|r|\leq n}|a(r)-a_n(r)|&\leq& 4/n,\\
\label{e-23}
|a_n(r)|\leq C(n,m,K), \quad |a'_n(r)|&\leq& 2K|r|^{\frac{m-3}{2}}, \quad \forall r\in \mathbb{R}.
\end{eqnarray}
Then, by (\ref{e-22}), it yields $A_n\in C^{\infty}(\mathbb{R})$ with one-order derivatives bounded by $C(n,m, K)$.
Moreover, if $A$ satisfies Hypotheses H with a constant $K\geq 1$, then $A_n$ satisfy Hypotheses H with constant $3K$.

Based on $(A_n, g_n, u_{0,n})$ defined above, for $h(t)=\sum_{k\geq 1}h_k(t)e_k$, let us consider the following approximation
\begin{eqnarray}\label{eqq-1-1}
\left\{
  \begin{array}{ll}
    du^{h}_n(t,x)=\Delta A_n(u^{h}_n(t,x))dt+\sum_{k\geq 1}g^k_n(x,u^{h}_n(t,x)) h_k(t)dt,\\
    u^{h}_n=u_{0,n}.
  \end{array}
\right.
\end{eqnarray}
Note that $\Delta A_n(u^{h}_n(t,x))=div(a^2_n(u^{h}_n(t,x))\nabla u^{h}_n(t,x))$ and $\frac{4}{n^2}\leq a^2_n(u^{h}_n(t,x)\leq C(n,m,K)$. Referring to Section 4 in \cite{DZZ} with $B=0, A=a^2_n$, we know that
 for each $n$, (\ref{eqq-1-1}) has a unique solution $u^{h}_n\in C([0,T];H)\cap L^2([0,T];H^1)$. Moreover, using integration by parts for $\|u^{h}_n(t)\|^2_H$ and $\|u^{h}_n(t)\|^{m+1}_{L^{m+1}(\mathbb{T}^N)}$, we get the following energy estimates uniformly in $n$:
\begin{eqnarray*}
\sup_{t\in[0,T]}\|u^{h}_n(t)\|^p_H +\int^T_0\|\nabla \Psi_n( u^{h}_n(s))\|^p_Hds&\leq& C(M,K,T, p, \|u_{0,n}\|_H),\\
\sup_{t\in[0,T]}\|u^{h}_n(t)\|^{m+1}_{L^{m+1}} &\leq& C(M,K,T,m, \|u_{0,n}\|_{L^{m+1}}),
\end{eqnarray*}
for all $p\geq 2$. Applying  integration by parts for the function $u\rightarrow\int_{\mathbb{T}^N}\int^u_0A_n(s)ds$ and using (\ref{e-24-1})-(\ref{e-25-1}), it yields
\[
\int^T_0\|\nabla A_n( u^{h}_n(s))\|^2_Hds\leq C(M,K,T,m, \|u_{0,n}\|_{L^{m+1}}).
\]
Note that $u_{0,n}$ are bounded by $n$, which implies that the right hand side of the above inequalities are finite. Moreover, since $u_{0,n}\leq u_0$, we get for all $p\geq 2$,
\begin{eqnarray}\label{e-24-1}
\sup_{t\in[0,T]}\|u^{h}_n(t)\|^p_H +\int^T_0\|\nabla \Psi_n( u^{h}_n(s))\|^p_Hds&\leq& C(M,K,T, p, \|u_{0}\|_H),\\
\label{e-25-1}
\sup_{t\in[0,T]}\|u^{h}_n(t)\|^{m+1}_{L^{m+1}}+ \int^T_0\|\nabla A_n( u^{h}_n(s))\|^2_Hds&\leq& C(M,K,T,m, \|u_{0}\|_{L^{m+1}}).
\end{eqnarray}
Since $a_n\geq \frac{2}{n}>0$, we have $|\nabla u^{h}_n|\leq N(n)|\nabla \Psi_n( u^{h}_n)|$, then by (\ref{e-24-1}), it yields
\begin{eqnarray}\label{e-26-1}
\int^T_0\|\nabla u^{h}_n(t)\|^p_Hdt\leq N(n)C(M,K,T, p, \|u_{0}\|_H).
\end{eqnarray}
In the following, for the sake of convenience, denote $u_n=u^{h}_n$ and $u=u^{h}$.

We will show that $(u_n)_{n\in \mathbb{N}}$ is a Cauchy sequence in $L_1([0,T];L^1(\mathbb{T}^N))$. For any $\alpha\in (0,1\wedge\frac{m}{2})$, and $n\leq n'$, we apply Theorem \ref{thm-2} to $u_n$ and $u_{n'}$.  Setting $\delta=\gamma^{\frac{3}{2\alpha}}$ and $\lambda=\frac{8}{n}$, then we deduce from (\ref{e-29}) that $R_{\lambda}\geq n$. Further, using (\ref{e-25-1}), it yields for a.e. $t\in [0,T]$, it holds that
\begin{eqnarray*}
\|u_n(t)-{u}_{n'}(t)\|_{L^1(\mathbb{T}^N)}
&\leq& \mathcal{E}_t(\gamma,\delta)+e^{K(T+M)}\Big[\|u_{0,n}-{u}_{0,n'}\|_{L^1(\mathbb{T}^N)} +\mathcal{E}_0(\gamma,\delta)\\ \notag
&& +N_0(\gamma^{-2}n^{-2}+\gamma)\Big(1+\|u_n\|^m_{L^m([0,T]\times \mathbb{T}^N)}+\|u_{n'}\|^m_{L^m([0,T]\times \mathbb{T}^N)}\Big)\\
\notag
&& +N_0\gamma^{-2}\|I_{|u_n|\geq n}(1+|u_n|)\|^m_{L^m([0,T]\times \mathbb{T}^N)}+N_0\gamma^{-2}\|I_{|u_{n'}|\geq n}(1+|u_{n'}|)\|^m_{L^m([0,T]\times \mathbb{T}^N)} \\
&& +2 K(\gamma+2\gamma^{\frac{3}{2\alpha}})(T+M)+2d^{\frac{1}{2}}(g_n,g_{n'})C(m)(T+M)\Big]\\
&\leq& \mathcal{E}_t(\gamma,\gamma^{\frac{3}{2\alpha}})+e^{K(T+M)}\Big[\|u_{0,n}-{u}_{0}\|_{L^1(\mathbb{T}^N)} +\|u_{0}-{u}_{0,n'}\|_{L^1(\mathbb{T}^N)} +\mathcal{E}_0(\gamma,\gamma^{\frac{3}{2\alpha}}) +N_0(\gamma^{-2}n^{-2}+\gamma)\\
\notag
&& +N_0\gamma^{-2}\|I_{|u_n|\geq n}(1+|u_n|)\|^m_{L^m([0,T]\times \mathbb{T}^N)}+N_0\gamma^{-2}\|I_{|u_{n'}|\geq n}(1+|u_{n'}|)\|^m_{L^m([0,T]\times \mathbb{T}^N)} \\
&& +2 K(\gamma+2\gamma^{\frac{3}{2\alpha}})(T+M)+2d^{\frac{1}{2}}(g_n,g)C(m)(T+M)+2d^{\frac{1}{2}}(g,g_{n'})C(m)(T+M)\Big]\\
&=:&M(\gamma)+e^{K(T+M)}\Big[\|u_{0,n}-{u}_{0}\|_{L^1(\mathbb{T}^N)} +\|u_{0}-{u}_{0,n'}\|_{L^1(\mathbb{T}^N)}+ N_0\gamma^{-2}n^{-2}\\
&& +N_0\gamma^{-2}\|I_{|u_n|\geq n}(1+|u_n|)\|^m_{L^m([0,T]\times \mathbb{T}^N)}+N_0\gamma^{-2}\|I_{|u_{n'}|\geq n}(1+|u_{n'}|)\|^m_{L^m([0,T]\times \mathbb{T}^N)} \\
&& +2d^{\frac{1}{2}}(g_n,g)C(m)(T+M)+2d^{\frac{1}{2}}(g,g_{n'})C(m)(T+M)\Big],
\end{eqnarray*}
where $M(\gamma)=\mathcal{E}_t(\gamma,\gamma^{\frac{3}{2\alpha}})+e^{K(T+M)}\Big[\mathcal{E}_0(\gamma,\gamma^{\frac{3}{2\alpha}}) +N_0\gamma+2 K(\gamma+2\gamma^{\frac{3}{2\alpha}})(T+M)\Big]\rightarrow 0$ as $\gamma, \delta\rightarrow 0$ and the constant $N_0$ is independent of $\gamma,\delta,\lambda$.

For any $\iota>0$, let $\gamma>0$ be such that $M(\gamma)<\iota$. By (\ref{e-21}), we can choose $n_0$ big enough such that for $n_0\leq n\leq n'$,
\begin{eqnarray*}
&&\|u_{0,n}-{u}_{0}\|_{L^1(\mathbb{T}^N)} +\|u_{0}-{u}_{0,n'}\|_{L^1(\mathbb{T}^N)}+ N_0\gamma^{-2}n^{-2}\\
&&+2d^{\frac{1}{2}}(g_n,g)C(m)(T+M)+2d^{\frac{1}{2}}(g,g_{n'})C(m)(T+M)\leq 5\iota
\end{eqnarray*}
Due to the uniform integrability of $(1+|u_n|)^m$, for such $n_0$, we also have
\[
N_0\gamma^{-2}\|I_{|u_n|\geq n}(1+|u_n|)\|^m_{L^m([0,T]\times \mathbb{T}^N)}+N_0\gamma^{-2}\|I_{|u_{n'}|\geq n}(1+|u_{n'}|)\|^m_{L^m([0,T]\times \mathbb{T}^N)}\leq \iota.
\]
Hence, for $n_0\leq n\leq n'$, we get for a.e. $t\in [0,T]$,
\[
\|u_n(t)-{u}_{n'}(t)\|_{L^1(\mathbb{T}^N)}\leq 7\iota,
\]
which implies that $(u_n)_{n\in \mathbb{N}}$ converges in $L_1([0,T];L^1(\mathbb{T}^N))$. Moreover, by passing to a subsequence, we may assume that
\begin{eqnarray}\label{ee-1}
\lim_{n\rightarrow \infty} u_n=u, \quad a.e.\  (t,x)\in [0,T]\times \mathbb{T}^N.
\end{eqnarray}
In addition, it follows from (\ref{e-25-1}) that for any $q<m+1$,
\begin{eqnarray}\label{ee-2}
(|u_n(t,x)|^q)_{n\geq 1}\ {\rm{is\ uniformly\ integrable\ on\ }}  [0,T]\times \mathbb{T}^N.
\end{eqnarray}
Taking $q=2$ in (\ref{ee-2}), we get $u_n\rightarrow u$ strongly in $L^2([0,T];L^2(\mathbb{T}^N))$.

Next, we prove that $u$ is an entropy solution of (\ref{P-2}). Utilizing (\ref{e-25-1}) and Fatou lemma, we deduce that (i) of Definition \ref{dfn-4-1} holds.

Let $f\in C_b(\mathbb{R})$ and $\eta$ as in the Definition \ref{dfn-4-1}.
Recall the entropy fucntion
\begin{eqnarray*}
\Psi_f(r):=\int^r_0f(s)a(s)ds,\quad \forall \ f\in C(\mathbb{R}),\quad q_{\eta}\ {\rm{is\ any\ function\ satisfying\ }} q'_{\eta}(\xi)=\eta'(\xi)a^2(\xi).
\end{eqnarray*}
Analogously, we define $\Psi_{n,f}, q_{n,\eta}$ similar to the above with $a$ replaced by $a_n$. For each $n$, we have $\Psi_{n,f}(u_n)\in L_2([0,T];H^1(\mathbb{T}^N))$ and $\partial_i \Psi_{n,f}(u_n)=f(u_n)\partial_i \Psi_{n}(u_n)$. Also, we have $\Psi_{n,f}(r)\leq \|f\|_{L^{\infty}}3K |r|^{\frac{m+1}{2}}$ for all $r\in \mathbb{R}$, which combined with (\ref{e-24-1}) and (\ref{e-25-1}) gives that
\[
\sup_{n}\int^T_0\|\Psi_{n,f}(u_n)\|^2_{H^1(\mathbb{T}^N)}<\infty.
\]
Hence, for a subsequence, we have
$\Psi_{n,f}(u_n)\rightharpoonup v_f$, $\Psi_{n}(u_n)\rightharpoonup v$ for some $v_f, v\in L^2([0,T];H^1(\mathbb{T}^N))$. With the aid of (\ref{e-29}), (\ref{ee-1}) and (\ref{ee-2}), we deduce that $v_f=\Psi_{f}(u), v=\Psi(u)$. Moreover, by $\partial_i\Psi_n(u_n)\rightharpoonup \partial_i\Psi(u)$ weakly in $L^2([0,T];H)$, for any $\phi\in C^{\infty}(\mathbb{T}^N)$,  it holds that
\begin{eqnarray*}
\int^T_0\int_{\mathbb{T}^N}\partial_i\Psi_f(u)\phi&=&\lim_{n\rightarrow \infty}\int^T_0\int_{\mathbb{T}^N}\partial_i\Psi_{f,n}(u_n)\phi\\
&=&\lim_{n\rightarrow \infty}\int^T_0\int_{\mathbb{T}^N}f(u_n)\partial_i\Psi_{n}(u_n)\phi\\
&=&\int^T_0\int_{\mathbb{T}^N}f(u)\partial_i\Psi(u)\phi,
\end{eqnarray*}
where we have used $f(u_n)\rightarrow f(u)$ strongly in $L^2([0,T];H)$. Hence, (ii) of Definition \ref{dfn-4-1} is obtained.

In the following, we will show (iii) in Definition \ref{dfn-4-1}  holds.
Let $\eta$ and $\phi$ be as in (iii). By integration by parts, we get
\begin{eqnarray}\notag
-\int^T_0\int_{\mathbb{T}^N}\eta(u_n)\partial_t \phi dxdt&\leq & \int_{\mathbb{T}^N}\eta(u_{0,n})\phi(0)dx+\int^T_0\int_{\mathbb{T}^N}q_{n,\eta}(u_n)\Delta \phi dxdt\\
\label{ee-3}
&& -\int^T_0\int_{\mathbb{T}^N}\phi\eta''(u_n)|\nabla \Psi_n(u_n)|^2dxdt+\int^T_0\int_{\mathbb{T}^N}\phi\eta'(u_n)g^k_n(u_n)h_k(t)dxdt.
\end{eqnarray}
On the basis of (\ref{e-21})-(\ref{e-23}), (\ref{ee-1}) and (\ref{ee-2}), we have
\begin{eqnarray*}
\lim_{n\rightarrow \infty}\int^T_0\int_{\mathbb{T}^N}\eta(u_n)\partial_t \phi dxdt&=&\int^T_0\int_{\mathbb{T}^N}\eta(u)\partial_t \phi dxdt,\\
\lim_{n\rightarrow \infty}\int_{\mathbb{T}^N}\eta(u_{0,n})\phi(0)dx&=&\int_{\mathbb{T}^N}\eta(u_{0})\phi(0)dx,\\
\lim_{n\rightarrow \infty}\int^T_0\int_{\mathbb{T}^N}q_{n,\eta}(u_n)\Delta \phi dxdt&=&\int^T_0\int_{\mathbb{T}^N}q_{\eta}(u)\Delta \phi dxdt,\\
\lim_{n\rightarrow \infty}\int^T_0\int_{\mathbb{T}^N}\phi\eta'(u_n)g^k_n(u_n)h_k(t)dxdt&=&\int^T_0\int_{\mathbb{T}^N}\phi\eta'(u)g^k(u)h_k(t)dxdt.
\end{eqnarray*}
Set $\tilde{f}(r)=\sqrt{\eta''(r)}$. Notice that $\partial_i \Psi_{n,\tilde{f}}(u_n)=\sqrt{\eta''(u_n)}\partial_i \Psi_{n}(u_n)$. As before, we have (after passing to a subsequence if necessary) $\partial_i \Psi_{n,\tilde{f}}(u_n)\rightharpoonup \partial_i \Psi_{\tilde{f}}(u)$ in $L^2([0,T];H)$. In particular, we have $\partial_i \Psi_{n,\tilde{f}}(u_n)\rightharpoonup \partial_i \Psi_{\tilde{f}}(u)$ in $L^2([0,T]\times \mathbb{T}^N, \bar{\mu})$, where $d\bar{\mu}:=dx\otimes dt$. This implies that
\[
\int^T_0\int_{\mathbb{T}^N}\phi\eta''(u)|\nabla \Psi(u)|^2dxdt\leq \liminf_{n\rightarrow \infty}\int^T_0\int_{\mathbb{T}^N}\phi\eta''(u_n)|\nabla \Psi_n(u_n)|^2dxdt.
\]
Hence, taking $\liminf$ in both sides of (\ref{ee-3}), we by choosing an appropriate subsequence, we obtain that $u$ satisfies (iii) of Definition \ref{dfn-4-1}. We complete the proof.

\end{proof}

In view of Theorem \ref{thm-8} and Theorem \ref{thm-1}, we can define $\mathcal{G}^0: C([0,T];\mathcal{U})\rightarrow L^1([0,T];L^1(\mathbb{T}^N))$ by
\begin{eqnarray*}
\mathcal{G}^0(\check{h}):=\left\{
                   \begin{array}{ll}
                      u^{h}, & {\rm{if}}\ \check{h}= \int^{\cdot}_0 h(s)ds, \ {\rm{for\ some}}\ h\in L^2([0,T];U),\\
                    0, & {\rm{otherwise}},
                   \end{array}
                  \right.
\end{eqnarray*}
where $u^h$ is the solution of equation (\ref{P-2}).

\subsection{The continuity of the skeleton equation}
In this part, we aim to prove the continuity of the mapping $\mathcal{G}^0$. Namely, let $u^{h^{\varepsilon}}$ denote the kinetic solution  of (\ref{P-2}) with $h$ replaced by $h^\varepsilon$ and we will show that  $u^{h^{\varepsilon}}$ converges to the kinetic solution $u^h$ of the skeleton equation (\ref{P-2}) in $L^1([0,T];L^1(\mathbb{T}^N))$, if $h^\varepsilon\rightarrow h$ weakly in $L^2([0,T];U)$. To achieve it, we need the auxiliary approximating process $(A_n, g_n, u_{0,n})$ defined by (\ref{ee-4})-(\ref{e-23}).

For any family $\{h^\varepsilon,\varepsilon>0\}\subset S_M$ with $h^\varepsilon(t)=\sum_{k\geq 1}h^\varepsilon_k(t)e_k$, let us consider the following approximation
\begin{eqnarray}\label{eqq-1}
\left\{
  \begin{array}{ll}
    du^{h^\varepsilon}_n(t,x)=\Delta A_n(u^{h^\varepsilon}_n(t,x))dt+\sum_{k\geq 1}g^k_n(x,u^{h^\varepsilon}_n(t,x)) h^\varepsilon_k(t)dt,\\
    u^{h^\varepsilon}_n=u_{0,n}.
  \end{array}
\right.
\end{eqnarray}
Referring to Section 4 in \cite{DZZ}, for each $n$, the Cauchy problem $\mathcal{E}(A_n, g_n, u_{0,n})$ has a unique solution $u^{h^\varepsilon}_n\in C([0,T];H)\cap L^2([0,T];H^1)$ satisfying the following energy estimates uniformly in $n$:
\begin{eqnarray}\label{e-24}
\sup_{\varepsilon}\left\{\sup_{t\in[0,T]}\|u^{h^\varepsilon}_n(t)\|^p_H +\int^T_0\|\nabla \Psi_n( u^{h^\varepsilon}_n(s))\|^p_Hds\right\}&\leq& C(M,K,T, p, \|u_{0}\|_H),\\
\label{e-25}
\sup_{\varepsilon}\left\{\sup_{t\in[0,T]}\|u^{h^\varepsilon}_n(t)\|^{m+1}_{L^{m+1}(\mathbb{T}^N)} +\int^T_0\|\nabla A_n( u^{h^\varepsilon}_n(s))\|^2_Hds\right\}&\leq& C(M,K,T,m, \|u_{0}\|_{L^{m+1}(\mathbb{T}^N)}),
\end{eqnarray}
for all $p\geq 2$, where we have used $u_{0,n}\leq u_0$. Since $a_n\geq \frac{2}{n}>0$, we have $|\nabla u^{h^\varepsilon}_n|\leq N(n)|\nabla \Psi_n( u^{h^\varepsilon}_n)|$, then by (\ref{e-24}), it yields
\begin{eqnarray}\label{e-26}
\sup_{\varepsilon>0}\int^T_0\|\nabla u^{h^\varepsilon}_n(t)\|^p_Hdt\leq N(n)C(M,K,T, p, \|u_{0}\|_H).
\end{eqnarray}

With the above approximation process (\ref{eqq-1}), for any $\varepsilon>0$ and $n\geq 1$, we have
\begin{eqnarray*}
&&\|u^{h^{\varepsilon}}-u^h\|_{L^1([0,T];L^1(\mathbb{T}^N))}\\ \notag
&\leq& \|u^{h^{\varepsilon}}_n-u^{h^{\varepsilon}}\|_{L^1([0,T];L^1(\mathbb{T}^N))}
+\|u^{h^{\varepsilon}}_{n}-u^{h}_{n}\|_{L^1([0,T];L^1(\mathbb{T}^N))}+
+\|u^h_n-u^h\|_{L^1([0,T];L^1(\mathbb{T}^N))}.
\end{eqnarray*}
In order to establish the continuity of the skeleton equations, several steps are involved.

Firstly, we show the uniform convergence of the sequence $\{u^{h}_n, n\geq 1\}$ to $u^h$ over $h\in S_M$.
\begin{prp}\label{prp-4}
Assume $(A,g,u_0)$ satisfies Hypotheses H, then
\begin{eqnarray*}
\lim_{n\rightarrow \infty}\sup_{h\in S_M}\|u^h_n-u^h\|_{L^1([0,T];L^1(\mathbb{T}^N))}=0.
\end{eqnarray*}
\end{prp}
\begin{proof}
As discussed above, under assumption, we know that $(A_n, g_n, u_{0,n})$ satisfies Hypotheses H. According to (1) of Theorem \ref{thm-2}, we have that for all $\gamma, \delta\in (0,1)$, $\lambda\in [0,1]$, $\alpha\in (0,1\wedge \frac{m}{2})$ and a.e. $t\in [0,T]$,
\begin{eqnarray*}
\|u^h_n(t)-u^h(t)\|_{L^1(\mathbb{T}^N)}
&\leq& \mathcal{E}_t(\gamma,\delta)+e^{K(T+M)}\Big[\|u_0-u_{0,n}\|_{L^1(\mathbb{T}^N)} +\mathcal{E}_0(\gamma,\delta)\\ \notag
&& +N_0\gamma^{-2}(\lambda^2+\delta^{2\alpha})\Big(1+\|u^h_n\|^m_{L^m([0,T]\times \mathbb{T}^N)}+\|u^h\|^m_{L^m([0,T]\times \mathbb{T}^N)}\Big)\\
\notag
&& +N_0\gamma^{-2}\|I_{|u^h_n|\geq R_{\lambda}}(1+|u^h_n|)\|^m_{L^m([0,T]\times \mathbb{T}^N)}+N_0\gamma^{-2}\|I_{|u^h|\geq R_{\lambda}}(1+|u^h|)\|^m_{L^m([0,T]\times \mathbb{T}^N)} \\
&& +2 K(\gamma+2\delta)(T+M)+2d^{\frac{1}{2}}(g,g_n)C(m)(T+M)\Big],
\end{eqnarray*}
where $\mathcal{E}_0(\gamma,\delta), \mathcal{E}_t(\gamma,\delta)\rightarrow 0$ as $\gamma, \delta\rightarrow 0$, the constant $N_0$ is independent of $\gamma,\delta,\lambda$ and $R_{\lambda}$ is defined by
\[
R_{\lambda}=\sup\{R\in [0,\infty]: |a(r)-a_n(r)|\leq \lambda, \forall |r|< R\}.
\]
 For $n\geq 4$, we can choose $\lambda=\frac{4}{n}$, by (\ref{e-29}), we deduce that $R_{\lambda}\geq n$.
Furthermore, let $\gamma=n^{-\frac{1}{2}}$ and $\delta=\gamma^{\frac{3}{2\alpha}}$, it follows that
\begin{eqnarray*}
&&\underset{0\leq t\leq T}{{\rm{ess\sup}}}\ \|u^h_n(t)-u^h(t)\|_{L^1(\mathbb{T}^N)}\\
&\leq& \mathcal{E}_t(\gamma,\delta)+e^{K(T+M)}\Big[\|u_0-u_{0,n}\|_{L^1(\mathbb{T}^N)} +\mathcal{E}_0(\gamma,\delta)\\ \notag
&& +N_0\frac{16}{n}\Big(1+\|u^h_n\|^m_{L^m([0,T]\times \mathbb{T}^N)}+\|u^h\|^m_{L^m([0,T]\times \mathbb{T}^N)}\Big)\\ \notag
&& +N_0\frac{1}{\sqrt{n}}\Big(1+\|u^h_n\|^m_{L^m([0,T]\times \mathbb{T}^N)}+\|u^h\|^m_{L^m([0,T]\times \mathbb{T}^N)}\Big)\\
\notag
&& +N_0 n\|I_{|u^h_n|\geq n}(1+|u^h_n|)\|^m_{L^m([0,T]\times \mathbb{T}^N)}+N_0n\|I_{|u^h|\geq n}(1+|u^h|)\|^m_{L^m([0,T]\times \mathbb{T}^N)} \\
&& +2 K(n^{-\frac{1}{2}}+2n^{-\frac{3}{4\alpha}})(T+M)+2d^{\frac{1}{2}}(g,g_n)C(m)(T+M)\Big].
\end{eqnarray*}
Note that
\begin{eqnarray*}
n\|I_{|u^h_n|\geq n}(1+|u^h_n|)\|^m_{L^m([0,T]\times \mathbb{T}^N)}&\leq& \int^T_0\int_{\mathbb{T}^N}I_{|u^h_n|\geq n}|u^h_n(x,t)|(1+|u^h_n(x,t)|)^mdxdt\\
&\leq& C(m)\int^T_0\int_{\mathbb{T}^N}I_{|u^h_n|\geq n}(1+|u^h_n(x,t)|^{m+1})dxdt,
\end{eqnarray*}
hence, by (\ref{e-25}), we get $\sup_{h\in S_M}n\|I_{|u^h_n|\geq n}(1+|u^h_n|)\|^m_{L^m([0,T]\times \mathbb{T}^N)}\rightarrow 0$, as $n\rightarrow \infty$. Similarly, we obtain $\sup_{h\in S_M}n\|I_{|u^h|\geq n}(1+|u^h|)\|^m_{L^m([0,T]\times \mathbb{T}^N)}\rightarrow 0$, as $n\rightarrow \infty$.

Taking into account (\ref{e-19}), (\ref{e-21}) and (\ref{e-25}), we get
\begin{eqnarray*}
\sup_{h\in S_M}\|u^h_n-u^h\|_{L^1([0,T];L^1(\mathbb{T}^N))}
\leq T\cdot \sup_{h\in S_M}\underset{0\leq t\leq T}{{\rm{ess\sup}}}\ \|u^h_n(t)-u^h(t)\|_{L^1(\mathbb{T}^N)}\rightarrow 0,\quad {\rm{as}}\ n\rightarrow \infty.
\end{eqnarray*}

\end{proof}

In the following,
we prove the compactness of $\{u^{h^\varepsilon}_n, \varepsilon> 0\}$. For simplicity,  we set $u^{\varepsilon}_{n}:=u^{h^\varepsilon}_{n}$.

As in \cite{FG95}, we introduce the following space. Let $Y$ be a separable Banach space with the norm $\|\cdot\|_Y$.
Given $p>1, \beta\in (0,1)$, let $W^{\beta,p}([0,T]; Y)$ be the Sobolev space of all functions $u\in L^p([0,T];Y)$ such that
\[
\int^T_0\int^T_0\frac{\|u(t)-u(s)\|^p_Y}{|t-s|^{1+\beta p}}dtds< \infty,
\]
which is then endowed with the norm
\[
\|u\|^p_{W^{\beta,p}([0,T]; Y)}=\int^T_0\|u(t)\|_Y^pdt+\int^T_0\int^T_0\frac{\|u(t)-u(s)\|_Y^{ p}}{|t-s|^{1+\beta p}}dtds.
\]
The following result can be found in \cite{FG95}.
\begin{lemma}\label{lem-1-1}
Let $B_0\subset B\subset B_1$ be three Banach spaces. Assume that both $B_0$ and $B_1$ are reflexive, and $B_0$ is compactly embedded in $B$. Let $p\in (1, \infty)$ and $\beta \in (0, 1)$ be given. Let $\Lambda$ be the space
\[
\Lambda:= L^p([0, T]; B_0)\cap W^{\beta, p}([0,T]; B_1)
\]
endowed with the natural norm. Then the embedding of $\Lambda$ in $L^p([0,T];B)$ is compact.

\end{lemma}
{\color{rr}
To obtain the compactness, we also need the following lemma.
\begin{lemma}\label{lem-2-1}
Assume $B_2\subset B_3$ are two Banach spaces with compact embedding, and the real numbers $\beta\in (0,1)$, $p>1$ satisfy $\beta p>1$, then the space $C([0,T];B_2)\cap W^{\beta, p}([0,T];B_3)$ is compactly embedded into $C([0,T];B_3)$.
\end{lemma}
\begin{proof}
Clearly, the space $W^{\beta, p}([0,T];B_3)$ is continuously embedded into $C^{\iota}([0,T];B_3)$ for all $\iota\in (0, \beta p-1)$. Thus, if a set $\mathcal{R}$ is bounded in  $W^{\beta, p}([0,T];B_3)\cap C([0,T];B_2)$, it is bounded in $C^{\iota}([0,T];B_3)$. It follows that the functions in $\mathcal{R}$ are uniformly equi-continuous in $C([0,T];B_3)$. Since $B_2\subset B_3$ is compactly embedded, for each $s\in [0,T]$ the set $\{f(s): f\in \mathcal{R}\}$ is bounded in $B_2$ and thus relatively compact in $B_3$. We can apply Ascoli-Arzel\`{a} theorem to conclude that $\mathcal{R}$ is relatively compact in $C([0,T];B_3)$. We complete the proof.

\end{proof}
}

With the help of Lemma \ref{lem-1-1}, we obtain the following result.
\begin{prp}\label{prpp-1}
For any $n\geq 1$, $\{u^{\varepsilon}_{n}, \varepsilon> 0\}$ is compact in $L^2([0,T];H)$.
\end{prp}
\begin{proof}
From (\ref{eqq-1}), $ u^{\varepsilon}_{n}$ can be written as
\begin{eqnarray*}
 u^{\varepsilon}_{n}(t)&=&u_{0,n}+\int^t_0 \Delta A_n(u^{\varepsilon}_{n}) ds +\int^t_0\sum_{k\geq 1}g^k_n(x,u^{\varepsilon}_{n}(t,x)) h^{\varepsilon}_k(s)ds\\
 &=:& I^\varepsilon_1+I^\varepsilon_2(t)+I^\varepsilon_3(t).
\end{eqnarray*}
Clearly, $\|I^\varepsilon_1\|_{H}\leq C_1$. Next,
\begin{eqnarray*}
\| \Delta A_n(u^{\varepsilon}_n) \|_{H^{-1}(\mathbb{T}^N)}
&=& \sup_{\|v\|_{H^1(\mathbb{T}^N)}\leq 1}|\langle v,\Delta A_n(u^{\varepsilon}_n)\rangle|\\
&=& \sup_{\|v\|_{H^1(\mathbb{T}^N)}\leq 1}|\langle \nabla v, \nabla A_n(u^{\varepsilon}_n)\rangle|\\
&\leq& C\|\nabla A_n(u^{\varepsilon}_n)\|_H
\end{eqnarray*}
which then yields the following
\begin{eqnarray*}
\|I^\varepsilon_2(t)-I^\varepsilon_2(s)\|^2_{H^{-1}(\mathbb{T}^N)}
&=&\|\int^t_s\Delta A_n(u^{\varepsilon}_n)(l)dl \|^2_{H^{-1}(\mathbb{T}^N)}\\
&\leq& C(t-s)\int^t_s\|\Delta A_n(u^{\varepsilon}_n)(l)\|^2_{H^{-1}(\mathbb{T}^N)}dl\\
&\leq& C(t-s)\int^t_s\|\nabla A_n(u^{\varepsilon}_n)(l)\|^2_Hdl.
\end{eqnarray*}
Hence, by (\ref{e-25}), we have for $\beta\in (0,\frac{1}{2})$,
\begin{eqnarray*}\notag
&&\sup_\varepsilon\|I^\varepsilon_2\|^2_{W^{\beta,2}([0,T];H^{-1}(\mathbb{T}^N))}\\ \notag
&\leq&\int^T_0 \|I^\varepsilon_2(t)\|^2_{H^{-1}(\mathbb{T}^N)}dt+\int^T_0\int^T_0\frac{\|I^\varepsilon_2(t)-I^\varepsilon_2(s)\|^2_{H^{-1}(\mathbb{T}^N)}}{|t-s|^{1+2\beta}}dsdt\\
&\leq& C_2(\beta).
\end{eqnarray*}
Moreover, by H\"{o}lder inequality and (\ref{equ-28}), it follows that
\begin{eqnarray*}
\|\sum_{k\geq 1}g^k_n(x,u^{\varepsilon}_{n}(l,x)) h^{\varepsilon}_k(l)\|^2_{H}&\leq& \int_{\mathbb{T}^N}\sum_{k\geq 1}|g^k_n(x,u^{\varepsilon}_{n}(l,x))|^2\sum_{k\geq 1}|h^{\varepsilon}_k(l)|^2\\
&\leq& K^2(1+\|u^{\varepsilon}_n(l)\|^2_H)|h^{\varepsilon}(l)|^2_U,
\end{eqnarray*}
then, by H\"{o}lder inequality, we get
\begin{eqnarray*}
\|I^\varepsilon_3(t)-I^\varepsilon_3(s)\|^2_{H}
&=&\|\int^t_s\sum_{k\geq 1}g^k_n(x,u^{\varepsilon}_{n}(l,x)) h^{\varepsilon}_k(l)dl \|^2_{H}\\
&\leq& (t-s)\int^t_s\|\sum_{k\geq 1}g^k_n(x,u^{\varepsilon}_{n}(l,x)) h^{\varepsilon}_k(l)\|^2_{H}dl\\
&\leq& K^2(t-s)(1+\sup_{t\in[0,T]}\|u^{\varepsilon}_n(t)\|^2_H)\int^t_s|h^{\varepsilon}(l)|^2_Udl\\
&\leq& K^2M(t-s)(1+\sup_{t\in[0,T]}\|u^{\varepsilon}_n(t)\|^2_H).
\end{eqnarray*}
Thus, we deduce from (\ref{e-24}) that for $\beta\in (0,\frac{1}{2})$,
\begin{eqnarray*}\notag
&&\sup_\varepsilon\|I^\varepsilon_3\|^2_{W^{\beta,2}([0,T];H)}\\ \notag
&\leq&\int^T_0 \|I^\varepsilon_3(t)\|^2_{H}dt+\int^T_0\int^T_0\frac{\|I^\varepsilon_3(t)-I^\varepsilon_3(s)\|^2_{H}}{|t-s|^{1+2\beta}}dsdt\\
&\leq& C_3(\beta).
\end{eqnarray*}
Collecting the above estimates, we conclude that for $\beta\in (0,\frac{1}{2})$,
\begin{eqnarray*}
\sup_\varepsilon\|u^{\varepsilon}_{n}\|^2_{W^{\beta,2}([0,T];H^{-1}(\mathbb{T}^N))}\leq C(\beta).
\end{eqnarray*}
{\color{rr}Taking into account that $u^{\varepsilon}_{n}\in C([0,T];H)\cap L^2([0,T];H^1(\mathbb{T}^N))$ and applying Lemma \ref{lem-1-1} with $B_0=H^1(\mathbb{T}^N)$, $B_1=H^{-1}(\mathbb{T}^N)$ and $B=H$, we conclude the desired result.}

\end{proof}

{\color{rr}Arguing similarly to the above, on the basis of the estimates (\ref{e-24}) and (\ref{e-25}), we have $\{u^{\varepsilon}_{n}; \varepsilon >0\}$ are bounded in $W^{\beta,p}([0,T];H^{-1}(\mathbb{T}^N))\cap C([0,T];H)$ for any $p\geq 2$. Choosing $\beta p>1$ and applying Lemma \ref{lem-2-1} with $B_2=H, B_3=H^{-1}(\mathbb{T}^N)$, we get}
\begin{prp}\label{prp-5}
For any $n\geq 1$, $\{u^{h_{\varepsilon}}_{n}; \varepsilon >0\}$ is compact in $L^2([0,T];H)\cap C([0,T];H^{-1}(\mathbb{T}^N))$.
\end{prp}
Now, we are in a position to prove the continuity of $\mathcal{G}^0$.
\begin{thm}\label{thmm-1}
Assume $h^\varepsilon\rightarrow h$ weakly in $L^2([0,T];U)$. Then $u^{h^{\varepsilon}}$ converges to $u^h$ in $L^1([0,T];L^1(\mathbb{T}^N))$, where $u^{h^{\varepsilon}}$ is the kinetic solution of (\ref{P-2}) with $h$ replaced by $h^{\varepsilon}$.
\end{thm}
\begin{proof}
Fix any $n\geq 1$.
For the solution $u^{h^{\varepsilon}}_n$ of (\ref{eqq-1}), we shall firstly prove that when $h^{\varepsilon}\rightarrow h$ weakly in $L^2([0,T];U)$, we have
\begin{eqnarray}\label{e-37}
\lim_{\varepsilon\rightarrow 0}\|u^{h^{\varepsilon}}_{n}-u^h_{n}\|_{L^1([0,T];L^1(\mathbb{T}^N))}=0,
\end{eqnarray}
where $u^{h}_{n}$ is the solution of (\ref{eqq-1}) with $h^{\varepsilon}$ replaced by $h$. This will be achieved if we show that for any sequence ${\varepsilon}_m\rightarrow 0$, one can find a subsequence ${\varepsilon}_{m_{k}}\rightarrow  0$ such that
\begin{eqnarray}\label{e-34}
\lim_{k\rightarrow \infty}\|u^{h^{{\varepsilon}_{m_{k}}}}_{n}-u^h_{n}\|_{L^1([0,T];L^1(\mathbb{T}^N))}=0.
\end{eqnarray}

From (\ref{e-24}), (\ref{e-26}) and Proposition \ref{prp-5}, we know that for sequence ${\varepsilon}_m\rightarrow 0$, there exists a subsequence \{$m_k$, $k\geq 1$\} and an element $u^{\ast}\in L^{\infty}([0,T];H)\cap L^2([0,T];H)\cap L^2([0,T];H^1)\cap C([0,T];H^{-1})$ such that
\begin{eqnarray}\label{e-31}
u^{h^{{\varepsilon}_{m_k}}}_{n}&\rightarrow & u^{\ast}\ {\rm{strongly\  in\ }} L^2([0,T];H)\cap C([0,T];H^{-1}),\\
\label{e-32}
u^{h^{{\varepsilon}_{m_k}}}_{n} &\rightarrow& u^{\ast}\ {\rm{weakly\ in}} \ L^2([0,T];H^1).
\end{eqnarray}
Clearly, we also have $h^{{\varepsilon}_{m_k}}\rightarrow h$ weakly in $L^2([0,T];U)$.
Thus, we only need to prove $u^{\ast}=u^h_{n}$.

From (\ref{eqq-1}), we know that for a test function $\phi\in C^{\infty}(\mathbb{T}^N)$, it holds that
\begin{eqnarray}\label{e-33}
\langle u^{h^{{\varepsilon}_{m_k}}}_{n}(t), \phi\rangle-\langle u_{0,n}, \phi\rangle=\int^t_0\langle A_n(u^{h^{{\varepsilon}_{m_k}}}_{n}), \Delta \phi\rangle ds+\int^t_0\langle \sum_{k\geq 1}g^k_n(x,u^{h^{{\varepsilon}_{m_k}}}_{n}(s,x))h^{{\varepsilon}_{m_k}}_k, \phi\rangle ds.
\end{eqnarray}
Due to (\ref{e-31}), we get
\begin{eqnarray*}
|\langle u^{h^{{\varepsilon}_{m_k}}}_{n}(t)-u^{\ast}(t), \phi\rangle|\rightarrow 0, \quad {\rm{as}}\ k\rightarrow \infty.
\end{eqnarray*}
Using (\ref{e-31}) and the Lipschitz property of $A_n$, we deduce that
\begin{eqnarray*}
&&\int^t_0\langle A_n(u^{h^{{\varepsilon}_{m_k}}}_{n})- A_n(u^{\ast}),  \Delta\phi\rangle ds\\
&\leq& C(n,m,K)\|\Delta\phi\|_{L^{\infty}}T^{\frac{1}{2}}\Big(\int^t_0\|u^{h^{{\varepsilon}_{m_k}}}_{n}-u^{\ast}\|^2_Hds\Big)^{\frac{1}{2}}\rightarrow 0.
\end{eqnarray*}
Note that
\begin{eqnarray*}
\int^t_0\langle \sum_{k\geq 1}\Big(g^k_n(x,u^{h^{{\varepsilon}_{m_k}}}_n)h^{{\varepsilon}_{m_k}}_k-g^k_n(x,u^{\ast})h_k\Big), \phi\rangle ds&=&\int^t_0 \langle \sum_{k\geq 1}(g^k_n(x,u^{h^{{\varepsilon}_{m_k}}}_n)-g^k_n(x,u^{\ast}))h^{{\varepsilon}_{m_k}}_k, \phi\rangle ds\\
&& +\int^t_0\langle \sum_{k\geq 1}g^k_n(x,u^{\ast})(h^{{\varepsilon}_{m_k}}_k-h_k), \phi\rangle ds.
\end{eqnarray*}
With the aid of (\ref{e-31}), we get
\begin{eqnarray*}
&&\int^t_0 \langle \sum_{k\geq 1}(g^k_n(x,u^{h^{{\varepsilon}_{m_k}}}_n)-g^k_n(x,u^{\ast}))h^{{\varepsilon}_{m_k}}_k, \phi\rangle ds\\
&\leq& \|\phi\|_{L^{\infty}}\int^t_0\int_{\mathbb{T}^N}\left(\sum_{k\geq 1}|g^k_n(x,u^{h^{{\varepsilon}_{m_k}}}_n)-g^k_n(x,u^{\ast})|^2\right)^{\frac{1}{2}}\left(\sum_{k\geq 1}|h^{{\varepsilon}_{m_k}}_k|^2\right)^{\frac{1}{2}} ds\\
&\leq& CK\left(\int^t_0\|u^{h^{{\varepsilon}_{m_k}}}_n- u^{\ast}\|^2_Hds\right)^{\frac{1}{2}}\left(\int^t_0|h^{{\varepsilon}_{m_k}}(s)|_Uds\right)^{\frac{1}{2}}\\
&\leq & CKM^{\frac{1}{2}}\left(\int^t_0\|u^{h^{{\varepsilon}_{m_k}}}_n- u^{\ast}\|^2_Hds\right)^{\frac{1}{2}}\rightarrow 0.
\end{eqnarray*}
Since $h^{{\varepsilon}_{m_k}}\rightarrow h$ weakly in $L^2([0,T];U)$, we get
\begin{eqnarray*}
\int^t_0\langle \sum_{k\geq 1}g^k_n(x,u^{\ast})(h^{{\varepsilon}_{m_k}}_k-h_k), \phi\rangle ds\rightarrow 0.
\end{eqnarray*}
Thus, let $k\rightarrow \infty$ in (\ref{e-33}) to obtain
\begin{eqnarray*}
\langle u^{\ast}(t), \phi\rangle-\langle u_{0,n}, \phi\rangle=\int^t_0\langle A_n(u^{\ast}), \Delta\phi\rangle dt+\int^t_0\langle g_n(u^{\ast})h, \phi\rangle dt,
\end{eqnarray*}
which means that $u^{\ast}$ is the solution to (\ref{eqq-1}) with $h^{\varepsilon}$ replaced by $h$. By the uniqueness of (\ref{eqq-1}), we conclude $u^{\ast}=u^h_n$. Thus, (\ref{e-34}) is proved, which implies (\ref{e-37}) holds.

Note that for any $\varepsilon>0$ and $n\geq 1$, we have
\begin{eqnarray}\notag
&&\|u^{h^{\varepsilon}}-u^h\|_{L^1([0,T];L^1(\mathbb{T}^N))}\\
\label{eqq-6}
&\leq& \|u^{h^{\varepsilon}}_n-u^{h^{\varepsilon}}\|_{L^1([0,T];L^1(\mathbb{T}^N))}
+\|u^{h^{\varepsilon}}_{n}-u^{h}_{n}\|_{L^1([0,T];L^1(\mathbb{T}^N))}+\|u^h_n-u^h\|_{L^1([0,T];L^1(\mathbb{T}^N))}.
\end{eqnarray}
For any $\iota>0$, by Proposition \ref{prp-4}, there exists $N_{0}$ such that for all $\varepsilon>0$,
$$\|u^{h^{\varepsilon}}_{N_0}-u^{h^{\varepsilon}}\|_{L^1([0,T];L^1(\mathbb{T}^N))}\leq \frac{\iota}{3} \,\,\mbox{and}\,\, \|u^h_{N_0}-u^h\|_{L^1([0,T];L^1(\mathbb{T}^N))}\leq \frac{\iota}{3}.$$
Letting $n=N_0$, we deduce from (\ref{eqq-6}) that
\begin{eqnarray*}\label{eqq-16}
\|u^{h^{\varepsilon}}-u^h\|_{L^1([0,T];L^1(\mathbb{T}^N))}
\leq \frac{2\iota}{3}
+\|u^{h^{\varepsilon}}_{N_0}-u^{h}_{N_0}\|_{L^1([0,T];L^1(\mathbb{T}^N))}.
\end{eqnarray*}
Using (\ref{e-37}), we know that there are $\varepsilon_0>0$ such that for any $\varepsilon\leq \varepsilon_0$, it holds that
\begin{eqnarray*}
\|u^{h^{\varepsilon}}_{N_0}-u^{h}_{N_0}\|_{L^1([0,T];L^1(\mathbb{T}^N))}\leq \frac{\iota}{3}.
\end{eqnarray*}
Thus, we conclude that
\begin{eqnarray*}
\lim_{\varepsilon\rightarrow 0}\|u^{h^{\varepsilon}}-u^h\|_{L^1([0,T];L^1(\mathbb{T}^N))}\leq \iota.
\end{eqnarray*}
Since the constant $\iota$ is arbitrary, we obtain the desired result.
\end{proof}

\section{Large deviations}

For any family $\{h^{\varepsilon}; 0<\varepsilon\leq 1\}\subset \mathcal{A}_M$ with $h^{\varepsilon}(t)=\sum_{k\geq 1}h^{\varepsilon}_k(t)e_k$, we consider the following equation
\begin{eqnarray}\label{P-20}
\left\{
  \begin{array}{ll}
  d\bar{u}^{\varepsilon}=\Delta A(\bar{u}^{\varepsilon})dt+\sum_{k\geq1}g^k(x,\bar{u}^{\varepsilon}) h^{\varepsilon}_k(t)dt+\sqrt{\varepsilon}\sum_{k\geq1}g^k(x,\bar{u}^{\varepsilon}) d\beta_k(t),\\
\bar{u}^{\varepsilon}(0)=u_0.
  \end{array}
\right.
\end{eqnarray}
Combining Theorem \ref{thm-4}, Theorem \ref{thm-8} and Theorem \ref{thm-1}, we know that there exists (\ref{P-20}) admits a unique kinetic solution $\bar{u}^{\varepsilon}$ with initial data $u_0\in L^{m+1}(\mathbb{T}^N)$. Then, based on the definition of kinetic solution, it holds that
\begin{eqnarray}\label{e-36}
\sup_{0\leq \varepsilon\leq 1}\mathbb{E}\left(\int^T_0\int_{\mathbb{T}^N}\int_{\mathbb{R}}|\xi|^{m+1}d\nu^{\varepsilon}_{x,t}(\xi)dxdt\right)\leq C_m,
\end{eqnarray}
and there exists a kinetic measure $\bar{m}^\varepsilon\in\mathcal{M}^+_0(\mathbb{T}^N\times [0,T]\times \mathbb{R})$ such that $\bar{f}^\varepsilon:=I_{\bar{u}^{\varepsilon}>\xi}$ fulfills that for all $\varphi\in C_c^1(\mathbb{T}^N\times[0,T)\times \mathbb{R})$,
\begin{eqnarray*}\notag
&&\int^T_0\langle \bar{f}^\varepsilon(t), \partial_t \varphi(t)\rangle dt+\langle f_0, \varphi(0)\rangle +\int^T_0\langle \bar{f}^\varepsilon(t), a^2(\xi)\Delta \varphi (t)\rangle dt\\ \notag
&=& \int^T_0\int_{\mathbb{T}^N}\int_{\mathbb{R}}\partial_{\xi}\varphi (x,t,\xi)|\nabla \Psi(\xi)|^2d\nu^{\varepsilon}_{x,t}(\xi)dxdt\\ \notag
&& -\sqrt{\varepsilon}\sum_{k\geq 1}\int^T_0\int_{\mathbb{T}^N}\int_{\mathbb{R}} g^k(x,\xi)\varphi(x,t,\xi)d\nu^{\varepsilon}_{x,t}(\xi)dxd\beta_k(t)\\ \notag
&&\ -\frac{\varepsilon}{2}\int^T_0\int_{\mathbb{T}^N}\int_{\mathbb{R}}\partial_{\xi}\varphi(x,t,\xi)G^2(x,\xi) d\nu^{\varepsilon}_{x,t}(\xi)dxdt\\
\label{equ-3}
&&\ -\sum_{k\geq 1}\int^T_0\int_{\mathbb{T}^N}\int_{\mathbb{R}}\int_{\mathbb{R}}\varphi(x,t,\xi) g^k(x,\xi)h^{\varepsilon}_k(t)d\nu^{\varepsilon}_{x,t}(\xi)dxdt + \bar{m}^\varepsilon(\partial_{\xi} \varphi), \quad a.s.,
\end{eqnarray*}
 where $G^2:=\sum_{k\geq1}|g^k|^2$ and $\nu^{\varepsilon}_{x,t}(\xi)=-\partial_{\xi}f=\delta_{\bar{u}^{\varepsilon}(x,t)=\xi}$. According to the definition of $\mathcal{G}^{\varepsilon}$, it is clear  that $\mathcal{G}^{\varepsilon}\Big(W(\cdot)+\frac{1}{\sqrt{\varepsilon}}\int^{\cdot}_0h^{\varepsilon}(s)ds\Big)=\bar{u}^{\varepsilon}(\cdot)$.

Due to Theorem \ref{thm-7} (the sufficient condition B) and Theorem \ref{thmm-1}, we only need to prove the following result to establish the main result.
\begin{thm}\label{thm-5}
For every $M<\infty$, let $\{h^{\varepsilon}: 0<\varepsilon\leq 1\}$ $\subset \mathcal{A}_M$. Then
\[
\Big\|\mathcal{G}^{\varepsilon}\Big(W(\cdot)+\frac{1}{\sqrt{\varepsilon}}\int^{\cdot}_0h^\varepsilon(s)ds\Big)-\mathcal{G}^{0}\Big(\int^{\cdot}_0h^\varepsilon(s)ds\Big)\Big\|_{L^1([0,T];L^1(\mathbb{T}^N))}\rightarrow 0\quad {\rm{in \ probability}}, \quad {\rm{as}} \ \ \varepsilon\rightarrow 0.
\]

\end{thm}
\begin{proof}
Recall that $\bar{u}^\varepsilon=\mathcal{G}^{\varepsilon}\Big(W(\cdot)+\frac{1}{\sqrt{\varepsilon}}\int^{\cdot}_0h^\varepsilon(s)ds\Big)$ is the kinetic solution to (\ref{P-20}) with the corresponding kinetic measure ${m}^{\varepsilon}_1$. Moreover, $v^{\varepsilon}:=\mathcal{G}^{0}\Big(\int^{\cdot}_0h^\varepsilon(s)ds\Big)$ is the kinetic solution to the skeleton equation (\ref{P-2}) with $h$ replaced by $h^\varepsilon$ and the corresponding kinetic measure is denoted by $m^{\varepsilon}_2$. Also, we have
\begin{eqnarray}\label{e-36-1}
\sup_{0\leq \varepsilon\leq 1}\mathbb{E}\int^T_0\|v^{\varepsilon}(t)\|^{m+1}_{L^{m+1}(\mathbb{T}^N)}dt<+\infty.
\end{eqnarray}

Denote $f_1(x,t,\xi):=I_{\bar{u}^\varepsilon(x,t)>\xi}$ and $f_2(y,t,\zeta):=I_{v^{\varepsilon}(y,t)>\zeta}$. Using the same procedure as Lemma \ref{lem-1}, we have for all $\varphi_1(x,\xi)\in C^{\infty}_c(\mathbb{T}^N_x\times \mathbb{R}_{\xi})$,
\begin{eqnarray*}
\langle {f}^{\pm}_1(t), \varphi_1\rangle&=&\langle {f}_{1,0}, \varphi_1\rangle \delta_0([0,t])+\int^t_0\langle {f}_1(s), a^2(\xi)\Delta_x \varphi_1 (x,\xi)\rangle ds\\
&& -\int^t_0\int_{\mathbb{T}^N}\int_{\mathbb{R}}\partial_{\xi}\varphi_1 (x,\xi)|\nabla_x \Psi(\xi)|^2d{\nu}^{1,\varepsilon}_{x,s}(\xi)dxds\\ \notag
&& +\sqrt{\varepsilon}\sum_{k\geq 1}\int^t_0\int_{\mathbb{T}^N}\int_{\mathbb{R}} g^k(x,\xi)\varphi_1(x,\xi)d{\nu}^{1,\varepsilon}_{x,s}(\xi)dxd\beta_k(s)\\ \notag
&& +\frac{\varepsilon}{2}\int^t_0\int_{\mathbb{T}^N}\int_{\mathbb{R}}\partial_{\xi}\varphi_1(x,\xi)G^2(x,\xi)d{\nu}^{1,\varepsilon}_{x,s}(\xi)dxds\\ \notag
&&+\sum_{k\geq 1}\int^t_0\int_{\mathbb{T}^N}\int_{\mathbb{R}}\varphi_1(x,\xi) g^k(x,\xi)h^{\varepsilon}_k(s)d{\nu}^{1,\varepsilon}_{x,s}(\xi)dxds-\langle {m}^\varepsilon_1,\partial_{\xi} \varphi_1\rangle([0,t])\\
&=:& \langle m^*_1, \partial_{\xi}\varphi_1\rangle([0,t])+F_1(t), \quad  a.s.,
\end{eqnarray*}
with
 \begin{eqnarray*}
 \langle m^*_1, \partial_{\xi}\varphi_1\rangle([0,t])&=&\langle f_{1,0}, \varphi_1\rangle\delta_0([0,t])+\int^t_0\langle {f}_1(s), a^2(\xi)\Delta_x \varphi_1 (x,\xi)\rangle ds\\
&& -\int^t_0\int_{\mathbb{T}^N}\int_{\mathbb{R}}\partial_{\xi}\varphi_1 (x,\xi)|\nabla_x \Psi(\xi)|^2d{\nu}^{1,\varepsilon}_{x,s}(\xi)dxds\\
 && +\frac{\varepsilon}{2}\int^t_0\int_{\mathbb{T}^N}\int_{\mathbb{R}}\partial_{\xi}\varphi_1(x,\xi)G^2(x,\xi)d{\nu}^{1,\varepsilon}_{x,s}(\xi)dxds\\ \notag
&&+\sum_{k\geq 1}\int^t_0\int_{\mathbb{T}^N}\int_{\mathbb{R}}\varphi_1(x,\xi) g^k(x,\xi)h^{\varepsilon}_k(s)d{\nu}^{1,\varepsilon}_{x,s}(\xi)dxds- \langle m^{\varepsilon}_1,\partial_{\xi} \varphi_1\rangle([0,t]),
 \end{eqnarray*}
 and
  \begin{eqnarray*}
 F_1(t)=\sqrt{\varepsilon}\sum_{k\geq 1}\int^t_0\int_{\mathbb{T}^N}\int_{\mathbb{R}}g^k(x,\xi)\varphi_1(x,\xi)d\nu^{1,\varepsilon}_{x,s}(\xi)dxd\beta_k(s).
 \end{eqnarray*}
where ${f}_{1,0}=I_{u_0>\xi}$ and ${\nu}^{1,\varepsilon}_{x,s}(\xi)=-\partial_{\xi}{f}^{\pm}_1(s,x,\xi)=\partial_{\xi}\bar{f}^{\pm}_1(s,x,\xi)=\delta_{\bar{u}^{\varepsilon,\pm}(x,t)=\xi}$.
 Moreover, referring to Remark 12 in \cite{D-V-1}, it gives that $\langle m^*_1, \partial_{\xi}\varphi_1\rangle(\{0\})=\langle f_{1,0}, \varphi_1\rangle$.

 Similarly, in view of (\ref{equ-3}), for all $\varphi_2(y,\zeta)\in C^{\infty}_c(\mathbb{T}^N_y\times \mathbb{R}_\zeta)$, we have
\begin{eqnarray*}\notag
\langle \bar{f}^{\pm}_2(t),\varphi_2\rangle&=&\langle \bar{f}_{2,0}, \varphi_2\rangle\delta_0([0,t])+\int^t_0\langle \bar{f}_2(s), a^2(\zeta)\Delta_y \varphi_2(y,\zeta)\rangle ds\\ \notag
&& +\int^t_0\int_{\mathbb{T}^N}\int_{\mathbb{R}}\partial_{\zeta}\varphi_2 (y,\zeta)|\nabla_y \Psi(\zeta)|^2d\nu^{2,\varepsilon}_{y,s}(\zeta)dyds\\ \notag
&&-\sum_{k\geq 1}\int^t_0\int_{\mathbb{T}^N}\int_{\mathbb{R}}  g^k(y,\zeta)\varphi_2 (y,\zeta)h^{\varepsilon}_k(s)d\nu^{2,\varepsilon}_{y,s}(\zeta)dyds +\langle m^\varepsilon_2,\partial_{\zeta} \varphi_2\rangle([0,t]),
\end{eqnarray*}
where $f_{2,0}=I_{u_0>\zeta}$ and $\nu^{2,\varepsilon}_{y,s}(\zeta)=\partial_{\zeta}\bar{f}^{\pm}_2(s,y,\zeta)=-\partial_{\zeta}{f}^{\pm}_2(s,y,\zeta)=\delta_{v^{\varepsilon, \pm}(y,t)=\zeta}$.
Denote the duality distribution over $\mathbb{T}^N_x\times \mathbb{R}_\xi\times \mathbb{T}^N_y\times \mathbb{R}_\zeta$ by $\langle\langle\cdot,\cdot\rangle\rangle$. Setting $\alpha(x,\xi,y,\zeta)=\varphi_1(x,\xi)\varphi_2(y,\zeta)$.
Using It\^{o} formula for continuous semimartingales, we obtain that
\begin{eqnarray*}
\langle f^+_1(t), \varphi_1\rangle\langle \bar{f}^+_2(t),\varphi_2 \rangle=\langle\langle f^+_1(t)\bar{f}^+_2(t), \alpha \rangle\rangle
\end{eqnarray*}
satisfies
\begin{eqnarray*}\notag
&& \langle\langle{f}^{\pm}_1(t)\bar{f}^{\pm}_2(t), \alpha \rangle\rangle\\ \notag
&=&\langle\langle f_{1,0} \bar{f}_{2,0}, \alpha \rangle\rangle
+\int^t_0\int_{(\mathbb{T}^N)^2}\int_{\mathbb{R}^2}f_1\bar{f}_2(a^2(\xi)\Delta_x+a^2(\zeta)\Delta_y)\alpha d\xi d\zeta dxdyds\\ \notag
&& +\int^t_0\int_{(\mathbb{T}^N)^2}\int_{\mathbb{R}^2}f^{\pm}_1\partial_{\zeta}\alpha |\nabla_y \Psi(\zeta)|^2d\nu^{2,\varepsilon}_{y,s}(\zeta)dyds -\int^t_0\int_{(\mathbb{T}^N)^2}\int_{\mathbb{R}^2}\bar{f}^{\pm}_2\partial_{\xi}\alpha|\nabla_x \Psi(\xi)|^2d{\nu}^{1,\varepsilon}_{x,s}(\xi)dxds\\
&& +\frac{\varepsilon}{2}\int^t_0\int_{(\mathbb{T}^N)^2}\int_{\mathbb{R}^2}\partial_{\xi} \alpha \bar{f}^{\pm}_2(s,y,\zeta) G^2(x,\xi)d{\nu}^{1,\varepsilon}_{x,s}(\xi)d\zeta dxdyds\\ \notag
&& +\sum_{k\geq 1}\int^t_0\int_{(\mathbb{T}^N)^2}\int_{\mathbb{R}^2}\bar{f}^{\pm}_2(s,y,\zeta)\alpha g^k(x,\xi)h^{\varepsilon}_k(s)d\zeta d{\nu}^{1,\varepsilon}_{x,s}(\xi)dxdyds\\ \notag
&& -\sum_{k\geq 1}\int^t_0\int_{(\mathbb{T}^N)^2}\int_{\mathbb{R}^2}{f}^{\pm}_1(s,x,\xi)\alpha g^k(y,\zeta)h^{\varepsilon}_k(s)d\xi d\nu^{2,\varepsilon}_{y,s}(\zeta)dxdyds\\ \notag
&& -\int^t_0\int_{(\mathbb{T}^N)^2}\int_{\mathbb{R}^2}\bar{f}^{\pm}_2(s,y,\zeta)\partial_{\xi} \alpha d{m}^\varepsilon_1(x,\xi,s)d\zeta dy+\int^t_0\int_{(\mathbb{T}^N)^2}\int_{\mathbb{R}^2}{f}^{\pm}_1(s,x,\xi)\partial_{\zeta} \alpha dm^\varepsilon_2(y,\zeta,s)d\xi dx\\ \notag
&& +\sqrt{\varepsilon}\sum_{k\geq 1}\int^t_0\int_{(\mathbb{T}^N)^2}\int_{\mathbb{R}^2}\bar{f}^{\pm}_2(s,y,\zeta) g^k(x,\xi)\alpha d\zeta d{\nu}^{1,\varepsilon}_{x,s}(\xi)dxdyd\beta_k(s)\\
&=:&\langle\langle {f}_{1,0}\bar{f}_{2,0}, \alpha \rangle\rangle+J_1+J_2+J_3+J_4+J_5+J_6+J_7+J_8+J_9,\quad a.s..
\end{eqnarray*}
Similarly, we get
\begin{eqnarray*}\notag
&& \langle\langle\bar{f}^{\pm}_1(t){f}^{\pm}_2(t), \alpha \rangle\rangle\\ \notag
&=&\langle\langle \bar{f}_{1,0} {f}_{2,0}, \alpha \rangle\rangle
+\int^t_0\int_{(\mathbb{T}^N)^2}\int_{\mathbb{R}^2}\bar{f}_1{f}_2(a^2(\xi)\Delta_x+a^2(\zeta)\Delta_y) \alpha d\xi d\zeta dxdyds\\ \notag
&& -\int^t_0\int_{(\mathbb{T}^N)^2}\int_{\mathbb{R}^2}\bar{f}^{\pm}_1\partial_{\zeta}\alpha |\nabla_y \Psi(\zeta)|^2d\nu^{2,\varepsilon}_{y,s}(\zeta)dyds+\int^t_0\int_{(\mathbb{T}^N)^2}\int_{\mathbb{R}^2}{f}^{\pm}_2\partial_{\xi}\alpha|\nabla_x \Psi(\xi)|^2d{\nu}^{1,\varepsilon}_{x,s}(\xi)dxds\\
&& -\frac{\varepsilon}{2}\int^t_0\int_{(\mathbb{T}^N)^2}\int_{\mathbb{R}^2}\partial_{\xi} \alpha {f}^{\pm}_2(s,y,\zeta) G^2(x,\xi)d{\nu}^{1,\varepsilon}_{x,s}(\xi)d\zeta dxdyds\\ \notag
&& -\sum_{k\geq 1}\int^t_0\int_{(\mathbb{T}^N)^2}\int_{\mathbb{R}^2}{f}^{\pm}_2(s,y,\zeta)\alpha g^k(x,\xi)h^{\varepsilon}_k(s)d\zeta d{\nu}^{1,\varepsilon}_{x,s}(\xi)dxdyds\\ \notag
&& +\sum_{k\geq 1}\int^t_0\int_{(\mathbb{T}^N)^2}\int_{\mathbb{R}^2}\bar{f}^{\pm}_1(s,x,\xi)\alpha g^k(y,\zeta)h^{\varepsilon}_k(s)d\xi d\nu^{2,\varepsilon}_{y,s}(\zeta)dxdyds\\ \notag
 && +\int^t_0\int_{(\mathbb{T}^N)^2}\int_{\mathbb{R}^2}{f}^{\pm}_2(s,y,\zeta)\partial_{\xi} \alpha d{m}^\varepsilon_1(x,\xi,s)d\zeta dy-\int^t_0\int_{(\mathbb{T}^N)^2}\int_{\mathbb{R}^2}\bar{f}^{\pm}_1(s,x,\xi)\partial_{\zeta} \alpha dm^\varepsilon_2(y,\zeta,s)d\xi dx\\ \notag
&& -\sqrt{\varepsilon}\sum_{k\geq 1}\int^t_0\int_{(\mathbb{T}^N)^2}\int_{\mathbb{R}^2}{f}^{\pm}_2(s,y,\zeta) g^k(x,\xi)\alpha d\zeta d{\nu}^{1,\varepsilon}_{x,s}(\xi)dxdyd\beta_k(s)\\
&=:&\langle\langle \bar{f}_{1,0}{f}_{2,0}, \alpha \rangle\rangle+\bar{J}_1+\bar{J}_2+\bar{J}_3+\bar{J}_4+\bar{J}_5+\bar{J}_6+\bar{J}_7+\bar{J}_8+\bar{J}_9,\quad a.s..
\end{eqnarray*}
By the same technology as Proposition 4.1 in \cite{Z-19}, we can deduce that the above two equations remain true if $\alpha \in C^{\infty}_b(\mathbb{T}^N_x\times \mathbb{R}_\xi\times \mathbb{T}^N_y\times \mathbb{R}_\zeta)$ is compactly supported in a neighbourhood of the diagonal
\[
\Big\{(x,\xi,x,\xi); x\in \mathbb{T}^N, \xi\in \mathbb{R}\Big\}.
\]
Taking $\alpha(x,y,\xi,\zeta)=\rho_{\gamma}(x-y)\psi_{\delta}(\xi-\zeta)$, where $\rho_{\gamma}$ and $\psi_{\delta}$ are approximations to the identity on $\mathbb{T}^N$ and $\mathbb{R}$, respectively. Then, we have
\begin{eqnarray}\notag
&&\int_{(\mathbb{T}^N)^2}\int_{\mathbb{R}^2}\rho_\gamma (x-y)\psi_{\delta}(\xi-\zeta)(f^{\pm}_1(x,t,\xi)\bar{f}^{\pm}_2(y,t,\zeta)+\bar{f}^{\pm}_1(x,t,\xi){f}^{\pm}_2(y,t,\zeta))d\xi d\zeta dxdy\\ \notag
&\leq& \int_{(\mathbb{T}^N)^2}\int_{\mathbb{R}^2}\rho_\gamma (x-y)\psi_{\delta}(\xi-\zeta)(f_{1,0}(x,\xi)\bar{f}_{2,0}(y,\zeta)+\bar{f}_{1,0}(x,\xi){f}_{2,0}(y,\zeta))d\xi d\zeta dxdy\\ \label{eee-2}
&& +\sum^{9}_{i=1}(\tilde{J}_i+\tilde{\bar{J}}_i), \quad a.s.,
\end{eqnarray}
where $\tilde{J}_i, \tilde{\bar{J}}_i$  in (\ref{eee-2}) are the corresponding $J_i, \bar{J}_i$ with $\alpha(x,y,\xi,\zeta)=\rho_{\gamma}(x-y)\psi_{\delta}(\xi-\zeta)$, for $i=1,\cdot\cdot\cdot,9$.

In view of (\ref{P-11}), it holds that
\begin{eqnarray*}
\tilde{J}_7&=&\int^t_0\int_{(\mathbb{T}^N)^2}\int_{\mathbb{R}^2}\bar{f}^{\pm}_2(s,y,\zeta)\partial_{\zeta} \alpha d{m}^\varepsilon_1(x,\xi,s)d\zeta dy\\
&=&-\int^t_0\int_{(\mathbb{T}^N)^2}\int_{\mathbb{R}^2}\alpha d{m}^\varepsilon_1(x,\xi,s)d\nu^{2,\varepsilon,\pm}_{y,s}(\zeta)dy\leq 0, \quad a.s.,
\end{eqnarray*}
and
\begin{eqnarray*}
\tilde{J}_8&=&-\int^t_0\int_{(\mathbb{T}^N)^2}\int_{\mathbb{R}^2}{f}^{\pm}_1(s,x,\xi)\partial_{\xi} \alpha dm^\varepsilon_2(y,\zeta,s)d\xi dx\\
&=&-\int^t_0\int_{(\mathbb{T}^N)^2}\int_{\mathbb{R}^2} \alpha dm^\varepsilon_2(y,\zeta,s)d\nu^{1,\varepsilon,\pm}_{x,s} dx\leq 0, \quad a.s..
\end{eqnarray*}
By the same method as above, we deduce that $\tilde{\bar{J}}_7+ \tilde{\bar{J}}_8\leq 0$, a.s..

Moreover,
applying (\ref{e-13}) with $\lambda=0$ and $R_{\lambda}=\infty$, we have for any $\alpha\in (0,1\wedge \frac{m}{2})$, there exists a constant $N_0$ independent of $\gamma,\delta,\lambda,\varepsilon$ such that
\begin{eqnarray*}\label{e-35}
\sum^{3}_{i=1}(\tilde{J}_i+\tilde{\bar{J}}_i)\leq N_0\gamma^{-2}\delta^{2\alpha}(1+\|\bar{u}^\varepsilon\|^m_{L^m([0,T]\times \mathbb{T}^N)}+\|v^{\varepsilon}\|^m_{L^m([0,T]\times \mathbb{T}^N)}), \quad a.s..
\end{eqnarray*}
With the aid of $\chi_1(\xi,\zeta)$, $\chi_2(\xi,\zeta)$ and by using (\ref{equ-28}), we have
\begin{eqnarray*}\notag
\tilde{\bar{J}}_4&=&\tilde{J}_4\\
&=&\frac{\varepsilon}{2}\int^t_0\int_{(\mathbb{T}^N)^2}\int_{\mathbb{R}^2}\alpha  G^2(x,\xi)d\nu^{1,\varepsilon}_{x,s}\otimes d\nu^{2,\varepsilon}_{y,s}(\xi,\zeta)dxdyds\\ \notag
&\leq& \frac{\varepsilon}{2}K^2\int^t_0\int_{(\mathbb{T}^N)^2}\int_{\mathbb{R}^2}\alpha(1+|\xi|^2)  d\nu^{1,\varepsilon}_{x,s}\otimes d\nu^{2,\varepsilon}_{y,s}(\xi,\zeta)dxdyds\\ \notag
&\leq&\frac{\varepsilon}{2}K^2\int^t_0\int_{(\mathbb{T}^N)^2}\int_{\mathbb{R}^2}\alpha  d\nu^{1,\varepsilon}_{x,s}\otimes d\nu^{2,\varepsilon}_{y,s}(\xi,\zeta)dxdyds\\ \notag
&& +\frac{\varepsilon}{2}K^2\int^t_0\int_{(\mathbb{T}^N)^2}\int_{\mathbb{R}^2}\alpha  |\xi|^2d\nu^{1,\varepsilon}_{x,s}\otimes d\nu^{2,\varepsilon}_{y,s}(\xi,\zeta)dxdyds.
\end{eqnarray*}
Clearly, it holds that
\begin{eqnarray}\notag
&&\int_{(\mathbb{T}^N)^2}\int_{\mathbb{R}^2}\alpha  d\nu^{1,\varepsilon}_{x,s}\otimes d\nu^{2,\varepsilon}_{y,s}(\xi,\zeta)dxdy\\ \notag
&\leq& \|\psi_{\delta}\|_{L^{\infty}}\int_{(\mathbb{T}^N)^2}\int_{\mathbb{R}^2}\rho_\gamma(x-y)d\nu^{1,\varepsilon}_{x,s}\otimes d\nu^{2,\varepsilon}_{y,s}(\xi,\zeta)dxdy\\ \notag
&\leq & \|\psi_{\delta}\|_{L^{\infty}}\int_{(\mathbb{T}^N)^2}\rho_\gamma(x-y)dxdy\\
\label{q-1}
&\leq&\delta^{-1}.
\end{eqnarray}
Moreover, it follows that
\begin{eqnarray}\notag
&&\int_{(\mathbb{T}^N)^2}\int_{\mathbb{R}^2}\alpha  |\xi|^2d\nu^{1,\varepsilon}_{x,s}\otimes d\nu^{2,\varepsilon}_{y,s}(\xi,\zeta)dxdy\\ \notag
&\leq&\int_{(\mathbb{T}^N)^2}\rho_\gamma(x-y)
\int_{\mathbb{R}^2}\psi_{\delta}(\xi-\zeta)|\xi|^2d\nu^{1,\varepsilon}_{x,s}\otimes d\nu^{2,\varepsilon}_{y,s}(\xi,\zeta)dxdy\\ \notag
&\leq&\|\psi_{\delta}\|_{L^{\infty}}\|\rho_\gamma\|_{L^{\infty}} \int_{(\mathbb{T}^N)^2}\int_{\mathbb{R}^2}|\xi|^2d\nu^{1,\varepsilon}_{x,s}\otimes d\nu^{2,\varepsilon}_{y,s}(\xi,\zeta)dxdy\\
\label{q-2}
&\leq& C\delta^{-1}\gamma^{-N}\|\bar{u}^{\varepsilon}(s)\|_{L^2(\mathbb{T}^N)}.
\end{eqnarray}
Hence, combining (\ref{q-1}) and (\ref{q-2}), we deduce that
\begin{eqnarray*}
\tilde{\bar{J}}_4=\tilde{J}_4
\leq \frac{\varepsilon}{2} K^2T\delta^{-1}+\frac{\varepsilon}{2}CK^2\delta^{-1}\gamma^{-N}T^{\frac{1}{2}}\|\bar{u}^{\varepsilon}\|_{L^2([0,T];L^2(\mathbb{T}^N))}, \quad a.s..
\end{eqnarray*}
Recall
\begin{eqnarray}\label{eee-3}
\chi_2(\zeta,\xi)=\int^{\infty}_{\zeta}\psi_{\delta}(\xi-\zeta')d\zeta'.
\end{eqnarray}
Using the similar arguments as in the proof of Proposition \ref{prp-1}, we have
\begin{eqnarray*}
\tilde{\bar{J}}_5+\tilde{\bar{J}}_6&=&\tilde{J}_5+\tilde{J}_6\\
&=& \sum_{k\geq 1}\int^t_0\int_{(\mathbb{T}^N)^2}\int_{\mathbb{R}^2}\chi_2(\zeta,\xi)\rho_\gamma(x-y)\Big( g^k(x,\xi) - g^k(y,\zeta)\Big)h^{\varepsilon}_k(s)d \nu^{1,\varepsilon}_{x,s}\otimes d\nu^{2,\varepsilon}_{y,s}(\xi,\zeta) dxdyds\\
&\leq& \sum_{k\geq 1}\int^t_0\int_{(\mathbb{T}^N)^2}\int_{\mathbb{R}^2}\chi_2(\zeta,\xi)\rho_\gamma(x-y)| g^k(x,\xi)- g^k(y,\zeta)||h^{\varepsilon}_k(s)| d \nu^{1,\varepsilon}_{x,s}\otimes d\nu^{2,\varepsilon}_{y,s}(\xi,\zeta) dxdyds\\
&\leq& \int^t_0\int_{(\mathbb{T}^N)^2}\int_{\mathbb{R}^2}\chi_2(\zeta,\xi)\rho_\gamma(x-y)\Big(\sum_{k\geq 1}| g^k(x,\xi)- g^k(y,\zeta)|^2\Big)^{\frac{1}{2}}\Big(\sum_{k\geq 1}|h^{\varepsilon}_k(s)|^2\Big)^{\frac{1}{2}} d \nu^{1,\varepsilon}_{x,s}\otimes d\nu^{2,\varepsilon}_{y,s}(\xi,\zeta) dxdyds\\
&\leq& K\int^t_0|h^\varepsilon(s)|_{U}\int_{(\mathbb{T}^N)^2}\int_{\mathbb{R}^2}\chi_2(\zeta,\xi)\rho_\gamma(x-y)|x-y| d \nu^{1,\varepsilon}_{x,s}\otimes d\nu^{2,\varepsilon}_{y,s}(\xi,\zeta) dxdyds\\
&&\ + K\int^t_0|h^\varepsilon(s)|_{U}
\int_{(\mathbb{T}^N)^2}\rho_\gamma(x-y)\int_{\mathbb{R}^2}\chi_2(\zeta,\xi)|\xi-\zeta|d \nu^{1,\varepsilon}_{x,s}\otimes d\nu^{2,\varepsilon}_{y,s}(\xi,\zeta) dxdyds\\
&=:& \tilde{J}_{5,1}+\tilde{J}_{6,1}.
\end{eqnarray*}
By
\begin{eqnarray*}
\int_{(\mathbb{T}^N)^2}\rho_\gamma(x-y)|x-y|dxdy&\leq& \gamma,\\
\int_{(\mathbb{T}^N)^2}\chi_2(\zeta,\xi) d \nu^{1,\varepsilon}_{x,s}\otimes d\nu^{2,\varepsilon}_{y,s}(\xi,\zeta)&\leq& 1,\quad a.s.,
\end{eqnarray*}
it follows that
\begin{eqnarray*}
\tilde{J}_{5,1}\leq  K\gamma(T+M), \quad a.s..
\end{eqnarray*}
Using the same method as the estimate of $\tilde{K}_{2,2}$ in Theorem \ref{thm-2}, we have
\begin{eqnarray*}
\tilde{J}_{6,1}&\leq& 2K\delta(T+M)\\
&& +K\int^t_0|h^\varepsilon(s)|_{U}
\int_{(\mathbb{T}^N)^2}\int_{\mathbb{R}^2}\rho_\gamma (x-y)\psi_{\delta}(\xi-\zeta)(f^{\pm}_1\bar{f}^{\pm}_2+\bar{f}^{\pm}_1{f}^{\pm}_2)d\xi d\zeta dxdyds,\quad a.s..
\end{eqnarray*}
Based on the above estimates, it yields
\begin{eqnarray*}
\tilde{J}_5+\tilde{J}_6
&=&\tilde{J}_5+\tilde{J}_6\\
&\leq&K(\gamma+2\delta)(T+M)\\
&& +K\int^t_0|h^\varepsilon(s)|_{U}
\int_{(\mathbb{T}^N)^2}\int_{\mathbb{R}^2}\rho_\gamma (x-y)\psi_{\delta}(\xi-\zeta)(f^{\pm}_1\bar{f}^{\pm}_2+\bar{f}^{\pm}_1{f}^{\pm}_2)d\xi d\zeta dxdyds, \quad a.s..
\end{eqnarray*}
Combining all the previous estimates, it follows that
\begin{eqnarray*}
&& \int_{(\mathbb{T}^N)^2}\int_{\mathbb{R}^2}\rho_\gamma (x-y)\psi_{\delta}(\xi-\zeta)(f^{\pm}_1(x,t,\xi)\bar{f}^{\pm}_2(y,t,\zeta)+\bar{f}^{\pm}_1(x,t,\xi){f}^{\pm}_2(y,t,\zeta))d\xi d\zeta dxdy\\ \notag
&\leq& \int_{\mathbb{T}^N}\int_{\mathbb{R}}(f_{1,0}(x,\xi)\bar{f}_{2,0}(x,\xi)+\bar{f}_{1,0}(x,\xi){f}_{2,0}(x,\xi))dxd\xi +\mathcal{E}_0(\gamma,\delta)\\
&& +N_0\gamma^{-2}\delta^{2\alpha}(1+\|\bar{u}^\varepsilon\|^m_{L^m([0,T]\times \mathbb{T}^N)}+\|v^{\varepsilon}\|^m_{L^m([0,T]\times \mathbb{T}^N)})+\varepsilon K^2T\delta^{-1}\\ \notag
&&+\varepsilon CK^2\delta^{-1}\gamma^{-N}T^{\frac{1}{2}}\|\bar{u}^{\varepsilon}\|_{L^2([0,T];L^2(\mathbb{T}^N))}+2K(2\delta+\gamma)(T+M)\\
&& +|\tilde{J}_9|(t)+|\tilde{\bar{J}}_9|(t)+2K\int^t_0|h^\varepsilon(s)|_{U}
\int_{(\mathbb{T}^N)^2}\int_{\mathbb{R}^2}\rho_\gamma (x-y)\psi_{\delta}(\xi-\zeta)(f^{\pm}_1\bar{f}^{\pm}_2+\bar{f}^{\pm}_1{f}^{\pm}_2)d\xi d\zeta dxdyds,\quad a.s..
\end{eqnarray*}
Applying Gronwall inequality, we get
\begin{eqnarray*}
&& \int_{(\mathbb{T}^N)^2}\int_{\mathbb{R}^2}\rho_\gamma (x-y)\psi_{\delta}(\xi-\zeta)(f^{\pm}_1(x,t,\xi)\bar{f}^{\pm}_2(y,t,\zeta)+\bar{f}^{\pm}_1(x,t,\xi){f}^{\pm}_2(y,t,\zeta))d\xi d\zeta dxdy\\ \notag
&\leq& e^{K(T+M)}\Big[\int_{\mathbb{T}^N}\int_{\mathbb{R}}(f_{1,0}\bar{f}_{2,0}+\bar{f}_{1,0}{f}_{2,0})dxd\xi+\mathcal{E}_0(\gamma,\delta)\Big]\\
&&\ +e^{K(T+M)}\Big[N_0\gamma^{-2}\delta^{2\alpha}(1+\|\bar{u}^\varepsilon\|^m_{L^m([0,T]\times \mathbb{T}^N)}+\|v^{\varepsilon}\|^m_{L^m([0,T]\times \mathbb{T}^N)})+\varepsilon K^2T\delta^{-1}\\
&& +\varepsilon CK^2\delta^{-1}\gamma^{-N}T^{\frac{1}{2}}\|\bar{u}^{\varepsilon}\|_{L^2([0,T];L^2(\mathbb{T}^N))}+2K(2\delta+\gamma)(T+M)+|\tilde{J}_9|(t)+|\tilde{\bar{J}}_9|(t)\Big], \quad a.s..
\end{eqnarray*}
Thus, collecting all the above estimates, we deduce that
\begin{eqnarray}\notag
&&\int_{\mathbb{T}^N}\int_{\mathbb{R}}(f^{\pm}_1(x,t,\xi)\bar{f}^{\pm}_2(x,t,\xi)+\bar{f}^{\pm}_1(x,t,\xi){f}^{\pm}_2(x,t,\xi))dxd\xi\\ \notag
&=&\int_{(\mathbb{T}^N)^2}\int_{\mathbb{R}^2}(f^{\pm}_1(x,t,\xi)\bar{f}^{\pm}_2(y,t,\zeta)+\bar{f}^{\pm}_1(x,t,\xi){f}^{\pm}_2(y,t,\zeta))\rho_{\gamma}(x-y)\psi_{\delta}(\xi-\zeta)dxdyd\xi d\zeta+\mathcal{E}_t(\gamma,\delta)\\
\notag
&\leq& e^{K(T+M)}\Big[\int_{\mathbb{T}^N}\int_{\mathbb{R}}(f_{1,0}\bar{f}_{2,0}+\bar{f}_{1,0}{f}_{2,0})dxd\xi+\mathcal{E}_0(\gamma,\delta)\Big]\\
\notag
&&\ +e^{K(T+M)}\Big[N_0\gamma^{-2}\delta^{2\alpha}(1+\|\bar{u}^\varepsilon\|^m_{L^m([0,T]\times \mathbb{T}^N)}+\|v^{\varepsilon}\|^m_{L^m([0,T]\times \mathbb{T}^N)})+\varepsilon K^2T\delta^{-1}\\ \notag
&& +\varepsilon CK^2\delta^{-1}\gamma^{-N}T^{\frac{1}{2}}\|\bar{u}^{\varepsilon}\|_{L^2([0,T];L^2(\mathbb{T}^N))}+2K(2\delta+\gamma)(T+M)+|\tilde{J}_9|(t)+|\tilde{\bar{J}}_9|(t)\Big]+\mathcal{E}_t(\gamma,\delta)\\
\label{qeq-21}
&=:& e^{K(T+M)}\int_{\mathbb{T}^N}\int_{\mathbb{R}}(f_{1,0}\bar{f}_{2,0}+\bar{f}_{1,0}{f}_{2,0})dxd\xi
 +e^{K(T+M)}(|\tilde{J}_9|(t)+|\tilde{\bar{J}}_9|(t))+r(\varepsilon,\gamma,\delta,t), \quad a.s.,
\end{eqnarray}
where the remainder is given by
 \begin{eqnarray*}
r(\varepsilon,\gamma,\delta,t)&=&e^{K(T+M)}\Big[N_0\gamma^{-2}\delta^{2\alpha}(1+\|\bar{u}^\varepsilon\|^m_{L^m([0,T]\times \mathbb{T}^N)}+\|v^{\varepsilon}\|^m_{L^m([0,T]\times \mathbb{T}^N)})+\varepsilon K^2T\delta^{-1}\\ \notag
&& +\varepsilon CK^2\delta^{-1}\gamma^{-N}T^{\frac{1}{2}}\|\bar{u}^{\varepsilon}\|_{L^2([0,T];L^2(\mathbb{T}^N))}+2K(2\delta+\gamma)(T+M)+\mathcal{E}_0(\gamma,\delta)\Big]+\mathcal{E}_t(\gamma,\delta).
\end{eqnarray*}
Applying the Burkholder-Davis-Gundy inequality, and utilizing (\ref{eee-3}), (\ref{equ-28}) that
\begin{eqnarray*}
\mathbb{E}\sup_{t\in [0,T]}|\tilde{J}_9|(t)&\leq&\sqrt{\varepsilon}\mathbb{E}\sup_{t\in [0,T]}\Big|\sum_{k\geq 1}\int^t_0\int_{(\mathbb{T}^N)^2}\int_{\mathbb{R}^2}\bar{f}^{\pm}_2(s,y,\zeta) g^k(x,\xi)\alpha d\zeta d{\nu}^{1,\varepsilon}_{x,s}(\xi)dxdyd\beta_k(s)\Big|\\
&=& \sqrt{\varepsilon}\mathbb{E}\sup_{t\in [0,T]}\Big|\sum_{k\geq 1}\int^t_0\int_{(\mathbb{T}^N)^2}\int_{\mathbb{R}^2}\bar{f}^{\pm}_2(s,y,\zeta)\partial_{\zeta} \chi_2(\xi,\zeta) \rho_{\gamma}(x-y) g^k(x,\xi) d\zeta d{\nu}^{1,\varepsilon}_{x,s}(\xi)dxdyd\beta_k(s)\Big|\\
&=& \sqrt{\varepsilon}\mathbb{E}\sup_{t\in [0,T]}\Big|\sum_{k\geq 1}\int^t_0\int_{(\mathbb{T}^N)^2}\int_{\mathbb{R}^2}\chi_2(\xi,\zeta) \rho_{\gamma}(x-y) g^k(x,\xi) d\nu^{1,\varepsilon}_{x,s}\otimes d\nu^{2,\varepsilon}_{y,s}(\xi,\zeta)dxdyd\beta_k(s)\Big|\\
&\leq& \sqrt{\varepsilon}\mathbb{E}\Big[\int^T_0\int_{(\mathbb{T}^N)^2}\int_{\mathbb{R}^2}\gamma^2_2(\xi,\zeta) \rho^2_{\gamma}(x-y)\Big(\sum_{k\geq 1}g^2_{k}(x,\xi)\Big)d\nu^{1,\varepsilon}_{x,s}\otimes d\nu^{2,\varepsilon}_{y,s}(\xi,\zeta)dxdyds\Big]^{\frac{1}{2}}\\
&\leq& \sqrt{\varepsilon}K\mathbb{E}\Big[\int^T_0\int_{(\mathbb{T}^N)^2}\int_{\mathbb{R}^2}\gamma^2_2(\xi,\zeta) \rho^2_{\gamma}(x-y)(1+|\xi|^2)d\nu^{1,\varepsilon}_{x,s}\otimes d\nu^{2,\varepsilon}_{y,s}(\xi,\zeta)dxdyds\Big]^{\frac{1}{2}}\\
&\leq& \sqrt{\varepsilon}K\gamma^{-N}\Big[\mathbb{E}\int^T_0\int_{\mathbb{T}^N}(1+|\bar{u}^{\varepsilon}(x,s)|^2)dxds\Big]^{\frac{1}{2}}\\
&=:& \sqrt{\varepsilon}K\gamma^{-N}\mathcal{R},
\end{eqnarray*}
with
\begin{eqnarray*}
\mathcal{R}=\sup_{0\leq \varepsilon\leq 1}\Big[\mathbb{E}\int^T_0\int_{\mathbb{T}^N}(1+|\bar{u}^{\varepsilon}(x,s)|^2)dxds\Big]^{\frac{1}{2}}<\infty,
\end{eqnarray*}
where we have used (\ref{e-36}).

By the same method as above, we deduce that
\begin{eqnarray*}
\mathbb{E}\sup_{t\in [0,T]}|\tilde{\bar{J}}_9|(t)
\leq \sqrt{\varepsilon}K\gamma^{-N}\mathcal{R}.
\end{eqnarray*}
For the remainder, we have
\begin{eqnarray}\notag
 &&\sup_{t\in [0,T]} r(\varepsilon,\gamma,\delta,t)\\ \notag
 &\leq& e^{K(T+M)}\Big[N_0\gamma^{-2}\delta^{2\alpha}(1+\|\bar{u}^\varepsilon\|^m_{L^m([0,T]\times \mathbb{T}^N)}+\|v^{\varepsilon}\|^m_{L^m([0,T]\times \mathbb{T}^N)})+\varepsilon K^2T\delta^{-1}\\ \notag
\label{qeq-23}
&&\ +\varepsilon CK^2\delta^{-1}\gamma^{-N}T^{\frac{1}{2}}\|\bar{u}^{\varepsilon}\|_{L^2([0,T];L^2(\mathbb{T}^N))}+2K(2\delta+\gamma)(T+M)+\mathcal{E}_0(\gamma,\delta)\Big]+\sup_{t\in [0,T]}\mathcal{E}_t(\gamma,\delta).
\end{eqnarray}
In the following, we aim to make estimate of error term $\sup_{t\in [0,T]}\mathcal{E}_t(\gamma, \delta)$.
For any $t\in [0,T]$, we have
\begin{eqnarray*}
\mathcal{E}_t(\gamma, \delta)&=&\int_{\mathbb{T}^N}\int_{\mathbb{R}}(f^{\pm}_1(x,t,\xi)\bar{f}^{\pm}_2(x,t,\xi)+\bar{f}^{\pm}_1(x,t,\xi){f}^{\pm}_2(x,t,\xi))d\xi dx
\\
&& -\int_{(\mathbb{T}^N)^2}\int_{\mathbb{R}^2}(f^{\pm}_1(x,t,\xi)\bar{f}^{\pm}_2(y,t,\zeta)+\bar{f}^{\pm}_1(x,t,\xi){f}^{\pm}_2(y,t,\zeta))\rho_{\gamma}(x-y)\psi_{\delta}(\xi-\zeta)dxdyd\xi d\zeta \\
&=& \Big[\int_{\mathbb{T}^N}\int_{\mathbb{R}}(f^{\pm}_1(x,t,\xi)\bar{f}^{\pm}_2(x,t,\xi)+\bar{f}^{\pm}_1(x,t,\xi){f}^{\pm}_2(x,t,\xi))d\xi dx\\
&& -\int_{(\mathbb{T}^N)^2}\int_{\mathbb{R}}\rho_{\gamma}(x-y)(f^{\pm}_1(x,t,\xi)\bar{f}^{\pm}_2(y,t,\xi)+\bar{f}^{\pm}_1(x,t,\xi){f}^{\pm}_2(y,t,\xi))d\xi dxdy\Big]\\
&& +\Big[\int_{(\mathbb{T}^N)^2}\int_{\mathbb{R}}\rho_{\gamma}(x-y)(f^{\pm}_1(x,t,\xi)\bar{f}^{\pm}_2(y,t,\xi)+\bar{f}^{\pm}_1(x,t,\xi){f}^{\pm}_2(y,t,\xi))d\xi dxdy\\
&& -\int_{(\mathbb{T}^N)^2}\int_{\mathbb{R}^2}(f^{\pm}_1(x,t,\xi)\bar{f}^{\pm}_2(y,t,\zeta)+\bar{f}^{\pm}_1(x,t,\xi){f}^{\pm}_2(y,t,\zeta))\rho_{\gamma}(x-y)\psi_{\delta}(\xi-\zeta)dxdyd\xi d\zeta\Big]\\
&=:&H_1+H_2,
\end{eqnarray*}
Applying the same method as (\ref{qeq-14}) and (\ref{qeq-14-1}), it follows that
\begin{eqnarray}\label{qeq-10}
|H_2|\leq 2\delta, \quad a.s..
\end{eqnarray}
Moreover, it is easy to deduce that
\begin{eqnarray*}
|H_1|&\leq& \Big|\int_{(\mathbb{T}^N)^2}\rho_{\gamma}(x-y)\int_{\mathbb{R}}I_{\bar{u}^{\varepsilon, \pm}(x,t)>\xi}(I_{v^{\varepsilon,\pm}(x,t)\leq \xi}-I_{v^{\varepsilon,\pm}(y,t)\leq \xi})d\xi dxdy\Big|\\
&& +\Big|\int_{(\mathbb{T}^N)^2}\rho_{\gamma}(x-y)\int_{\mathbb{R}}I_{\bar{u}^{\varepsilon, \pm}(x,t)\leq\xi}(I_{v^{\varepsilon,\pm}(x,t)> \xi}-I_{v^{\varepsilon,\pm}(y,t)> \xi})d\xi dxdy\Big|\\
&\leq& 2\int_{(\mathbb{T}^N)^2}\rho_{\gamma}(x-y)|v^{\varepsilon,\pm}(x,t)-v^{\varepsilon,\pm}(y,t)|dxdy.
\end{eqnarray*}
By (\ref{qeq-10}) and (\ref{qeq-18}), we have
\begin{eqnarray*}
&&\int_{(\mathbb{T}^N)^2}\rho_{\gamma}(x-y)|v^{\varepsilon,\pm}(x,t)-v^{\varepsilon,\pm}(y,t)|dxdy\\
&=& \int_{(\mathbb{T}^N)^2}\int_{\mathbb{R}}\rho_{\gamma}(x-y)(f^{\pm}_2(x,t,\xi)\bar{f}^{\pm}_2(y,t,\xi)+\bar{f}^{\pm}_2(x,t,\xi){f}^{\pm}_2(y,t,\xi))d\xi dxdy\\
&\leq& \int_{(\mathbb{T}^N)^2}\int_{\mathbb{R}^2}\rho_{\gamma}(x-y)\psi_{\delta}(\xi-\zeta)(f^{\pm}_2(x,t,\xi)\bar{f}^{\pm}_2(y,t,\zeta)+\bar{f}^{\pm}_2(x,t,\xi){f}^{\pm}_2(y,t,\zeta))d\xi d\zeta dxdy+2\delta\\
&\leq& e^{K(T+M)}\Big[\int_{\mathbb{T}^N}\int_{\mathbb{R}}(f_{2,0}\bar{f}_{2,0}+\bar{f}_{2,0}{f}_{2,0})d\xi dx+\mathcal{E}_0(\gamma,\delta)\Big]\\
&&\ +2e^{K(T+M)}[N_0\gamma^{-2}\delta^{2\alpha}\Big(1+2\|v^{\varepsilon}\|^m_{L^m([0,T]\times \mathbb{T}^N)}\Big)+2 K(\gamma+2\delta)(T+M)]+2\delta\\
&=&e^{K(T+M)}\mathcal{E}_0(\gamma,\delta)+2e^{K(T+M)}\Big[N_0\gamma^{-2}\delta^{2\alpha}\Big(1+2\|v^{\varepsilon}\|^m_{L^m([0,T]\times \mathbb{T}^N)}\Big)+2 K(\gamma+2\delta)(T+M)\Big]+2\delta,
\end{eqnarray*}
where $\mathcal{E}_0(\gamma, \delta)$ is defined by (\ref{qq-3}). Then,
\begin{eqnarray*}
|H_1|&\leq& 4\delta+2 e^{K(T+M)}\mathcal{E}_0(\gamma,\delta)\\
&& +4e^{K(T+M)}\Big[N_0\gamma^{-2}\delta^{2\alpha}\Big(1+2\|v^{\varepsilon}\|^m_{L^m([0,T]\times \mathbb{T}^N)}\Big)+2 K(\gamma+2\delta)(T+M)\Big],\quad a.s..
\end{eqnarray*}
Combining all the above estimates, we conclude that
\begin{eqnarray*}
\sup_{t\in [0,T]}\mathcal{E}_t(\gamma, \delta)&\leq& 6\delta+2e^{K(T+M)}\mathcal{E}_0(\gamma,\delta)\\
&&+4e^{K(T+M)}\Big[N_0\gamma^{-2}\delta^{2\alpha}\Big(1+2\|v^{\varepsilon}\|^m_{L^m([0,T]\times \mathbb{T}^N)}\Big)+2K(\gamma+2\delta)(T+M)\Big],\quad a.s..
\end{eqnarray*}
Hence, we deduce from (\ref{qeq-23}) that
\begin{eqnarray*}
&&\sup_{t\in [0,T]} r(\varepsilon,\gamma,\delta,t)\\
&\leq& e^{K(T+M)}\Big[N_0\gamma^{-2}\delta^{2\alpha}\Big(1+\|\bar{u}^\varepsilon\|^m_{L^m([0,T]\times \mathbb{T}^N)}+\|v^{\varepsilon}\|^m_{L^m([0,T]\times \mathbb{T}^N)}\Big)\\
&& +\varepsilon K^2T\delta^{-1}+\varepsilon CK^2\delta^{-1}\gamma^{-N}T^{\frac{1}{2}}\|\bar{u}^{\varepsilon}\|_{L^2([0,T];L^2(\mathbb{T}^N))}+2K(2\delta+\gamma)(T+M)\Big]\\
&&\ +6\delta+2e^{K(T+M)}\mathcal{E}_0(\gamma,\delta)+4e^{K(T+M)}\Big[N_0\gamma^{-2}\delta^{2\alpha}\Big(1+2\|v^{\varepsilon}\|^m_{L^m([0,T]\times \mathbb{T}^N)}\Big)+2K(\gamma+2\delta)(T+M)\Big],\quad a.s..
\end{eqnarray*}
Letting
\[
\delta=\varepsilon^{\frac{3}{4}},\quad \gamma=\varepsilon^{\frac{1}{4N+1}},
\]
then, we deduce that
\begin{eqnarray*}
\mathbb{E}\sup_{t\in [0,T]}|\tilde{J}_9|(t)
\leq K\mathcal{R}\varepsilon^{\frac{1}{2}-\frac{N}{4N+1}}\rightarrow 0, \quad {\rm{as}}\ \varepsilon \rightarrow 0,
\end{eqnarray*}
and
\begin{eqnarray*}
\mathbb{E}\sup_{t\in [0,T]}|\tilde{\bar{J}}_9|(t)
\rightarrow 0, \quad {\rm{as}}\ \varepsilon \rightarrow 0,
\end{eqnarray*}
which implies that $\sup_{t\in [0,T]}|\tilde{J}_9|(t)\rightarrow 0$ in probability and $\sup_{t\in [0,T]}|\tilde{\bar{J}}_9|(t)\rightarrow 0$ in probability, as $\varepsilon \rightarrow 0$ by Chebyshev inequality.

Taking $\alpha\in (\frac{1}{3N},1\wedge\frac{m}{2})$, it yields $b_0:=\frac{3 \alpha}{2}-\frac{2}{4N+1}>0$. By utilizing (\ref{e-36}) and (\ref{e-36-1}), it follows that
\begin{eqnarray*}\notag
&& \mathbb{E}\sup_{t\in [0,T]} r(\varepsilon,\gamma,\delta,t)\\ \notag
 &\leq& e^{K(T+M)}\Big[N_0\varepsilon^{b_0}\Big(1+\sup_{0\leq \varepsilon\leq 1}\mathbb{E}\|\bar{u}^\varepsilon\|^m_{L^m([0,T]\times \mathbb{T}^N)}+\sup_{0\leq \varepsilon\leq 1}\mathbb{E}\|v^{\varepsilon}\|^m_{L^m([0,T]\times \mathbb{T}^N)}\Big)+\varepsilon^{\frac{1}{4}} K^2T\\ \notag
&& +\varepsilon^{\frac{1}{4}-\frac{N}{4N+1}} CK^2T^{\frac{1}{2}}\sup_{0\leq \varepsilon\leq 1}\mathbb{E}\|\bar{u}^{\varepsilon}\|_{L^2([0,T]\times\mathbb{T}^N)}+2K(2\varepsilon^{\frac{3}{4}}+\varepsilon^{\frac{1}{4N+1}})(T+M)\Big]+6\varepsilon^{\frac{3}{4}}+2e^{K(T+M)}\mathcal{E}_0(\gamma,\delta)\\ \notag
&& +4e^{K(T+M)}\Big[N_0\varepsilon^{b_0}\Big(1+2\sup_{0\leq \varepsilon\leq 1}\mathbb{E}\|v^{\varepsilon}\|^m_{L^m([0,T]\times \mathbb{T}^N)}\Big)+2K(2\varepsilon^{\frac{3}{4}}+\varepsilon^{\frac{1}{4N+1}})(T+M)\Big]\\
\label{qeq-12}
 &\rightarrow& 0, \quad {\rm{as}} \quad \varepsilon\rightarrow0,
\end{eqnarray*}
which gives $\sup_{t\in [0,T]} r(\varepsilon,\gamma,\delta,t)\rightarrow 0$ in probability, as $\varepsilon\rightarrow0$.

Notice that we have $f_1=I_{\bar{u}^\varepsilon>\xi}$ and $f_2=I_{v^{\varepsilon}>\xi}$ with initial data $f_{1,0}=I_{u_0>\xi}$ and ${f}_{2,0}=I_{u_0>\xi}$, respectively. With the help of identity (\ref{equ-1}), we deduce from (\ref{qeq-21}) that
\begin{eqnarray*}\notag
\|\bar{u}^\varepsilon(t)-v^\varepsilon(t)\|_{L^1(\mathbb{T}^N)}
\leq e^{K(T+M)}(|\tilde{J}_9|(t)+|\tilde{\bar{J}}_9|(t))+ r(\varepsilon,\gamma,\delta,t), \quad a.s..
\end{eqnarray*}
Hence, we get
\begin{eqnarray*}
&&\|\bar{u}^\varepsilon-v^\varepsilon\|_{L^1([0,T];L^1(\mathbb{T}^N))}\\ \notag
&\leq& T\cdot \underset{0\leq t\leq T}{{\rm{ess\sup}}}\ \|\bar{u}^\varepsilon(t)-v^\varepsilon(t)\|_{L^1(\mathbb{T}^N)}\\
&\leq& Te^{K(T+M)}\Big(\sup_{t\in [0,T]}|\tilde{J}_9|(t)+\sup_{t\in [0,T]}|\tilde{\bar{J}}_9|(t)\Big)+T\cdot \sup_{t\in [0,T]}r(\varepsilon,\gamma,\delta,t)\\
&\rightarrow& 0
\end{eqnarray*}
in probability as $\varepsilon\rightarrow 0$. We complete the proof.
\end{proof}

\vskip 0.2cm
\noindent{\bf  Acknowledgements}\quad  This work is partly supported by National Natural Science Foundation of China (No. 11801032). Key Laboratory of Random Complex Structures and Data Science, Academy of Mathematics and Systems Science, Chinese Academy of Sciences (No. 2008DP173182). China Postdoctoral Science Foundation funded project (No. 2018M641204).

\def\refname{ References}


\begin{thebibliography}{2}


\bibitem{BDaR} V. Barbu, G. Da Prato, M. R\"{o}ckner: \emph{Existence of strong solutions for stochastic porous media equation under general monotonicity conditions.} Ann. Probab. 37, no. 2, 428-452 (2009).
\bibitem{BR} V. Barbu, M. R\"{o}ckner: \emph{
On a random scaled porous media equation.} J. Differential Equations 251, no. 9, 2494-2514 (2011).

\bibitem{BVW} C. Bauzet, G. Vallet, P. Wittbold: \emph{A degenerate parabolic-hyperbolic Cauchy problem with a stochastic force.} J. Hyperbolic Differ. Equ. 12, no. 3, 501-533 (2015).
\bibitem{BDa} V.I. Bogachev, G. Da Prato, M. R\"{o}kner: \emph{
Invariant measures of generalized stochastic equations of porous media.}
Dokl. Akad. Nauk 396, no. 1, 7-11 (2004).


 \bibitem{MP} M. Bou\'{e}, P. Dupuis: \emph{A variational representation for certain functionals of Brownian motion}. Ann. Probab. 26 (4) 1641-1659 (1998).
\bibitem{BD} A. Budhiraja, P. Dupuis:  \emph{A variational representation for positive functionals of infinite dimensional Brownian motion}. Probab. Math. Statist. 20, 39-61 (2000).
 \bibitem{BDM} A. Budhiraja, P. Dupuis, V. Maroulas: \emph{Large deviations for infinite dimensional stochastic dynamical systems.} Ann. Probab., 36 (4), pp. 1390-1420 (2008).

\bibitem{CP} G. Chen, B. Perthame: \emph{Well-posedness for non-isotropic degenerate parabolic-hyperbolic equations.} Ann. Inst. H. Poincar\'{e} Anal. Non Lin\'{e}aire 20, no. 4, 645-668 (2003).
\bibitem{CL} M. G. Crandall, T. M. Liggett: \emph{
Generation of semi-groups of nonlinear transformations on general Banach spaces.}
Amer. J. Math. 93 265-298 (1971).

\bibitem{DaR} G. Da Prato, M. R\"{o}ckner: \emph{
Invariant measures for a stochastic porous medium equation.}  Stochastic analysis and related topics in Kyoto, 13-29, Adv. Stud. Pure Math., 41, Math. Soc. Japan, Tokyo, 2004.
\bibitem{DGG} K. Dareiotis, M. Gerencs\'{e}r, B. Gess: \emph{Entropy solutions for stochastic porous media equations.} J. Differential Equations 266, no. 6, 3732-3763 (2019).

\bibitem{DHV} A. Debussche, M. Hofmanov\'{a}, and J. Vovelle: \emph{Degenerate parabolic stochastic partial differential equations: Quasilinear case}. Ann. Probab. 44, no. 3, 1916-1955 (2016).

\bibitem{D-V-1} A. Debussche, J. Vovelle:
\emph{Scalar conservation laws with stochastic forcing (revised version). http://math.univ-lyon1.fr/vovelle/DebusscheVovelleRevised}.
J. Funct. Anal. 259, no. 4, 1014-1042 (2010).


 \bibitem{DZ} A. Dembo, O. Zeitouni: \emph{Large Deviations Techniques and Applications}.
 Boston, Jones and Bartlett (1993).
\bibitem{DWZZ} Z. Dong, J.-L. Wu, R. Zhang,  T. Zhang: \emph{Large derivation principles for first-order scalar conservation laws with stochastic forcing.} Arxiv: 1806.02955.
\bibitem{DZZ}Z. Dong, R. Zhang, T. Zhang: \emph{Large deviations for quasilinear parabolic stochastic partial differential equations.} Potential Analysis (2018). DOI 10.1007/s11118-019-09763-1.
  \bibitem{DE} P. Dupuis, R.S. Ellis: \emph{A Weak Convergence Approach to the Theory of Large Deviations}. New York, Wiley (1997).
      \bibitem{F-N} J. Feng, D. Nualart: \emph{Stochastic scalar conservation laws.} J. Funct. Anal. 255, no. 2, 313-373 (2008).
\bibitem{FG95}  F. Flandoli, D. Gatarek: \emph{Martingale and stationary solutions for stochastic Navier-Stokes equations}. Probab. Theory Related Fields 102, no. 3, 367-391 (1995).
\bibitem{GH} B. Gess, M. Hofmanov\'{a}: \emph{Well-posedness and regularity for quasilinear degenerate parabolic-hyperbolic SPDE.} Ann. Probab. 46, no. 5, 2495-2544 (2018).
\bibitem{K}J.U. Kim: \emph{ On a stochastic scalar conservation law}. Indiana Univ. Math. J. 52 227-256 (2003).
\bibitem{K70} S. N. Kru\v{z}kov: \emph{First order quasilinear equations with several independent variables.}
Mat. Sb. (N.S.) 81 (123) 228-255 (1970).
\bibitem{L-P-T} P.L. Lions, B. Perthame, E. Tadmor: \emph{ A kinetic formulation of multidimensional scalar conservation laws and related equations}. J. of A.M.S., 7, 169-191 (1994).
\bibitem{L10} W. Liu: \emph{Large deviations for stochastic evolution equations with small multiplicative noise.} Appl. Math. Optim. 61, no. 1, 27-56 (2010).

\bibitem{MSZ} A. Matoussi, W. Sabbagh, T. Zhang: \emph{Large Deviation Principles of Obstacle Problems for Quasilinear Stochastic PDEs.} Appl Math Optim (2019). https://doi.org/10.1007/s00245-019-09570-5.
\bibitem{O96} F. Otto: \emph{$L^1$-contraction and uniqueness for quasilinear elliptic-parabolic equations.}
J. Differential Equations 131, no. 1, 20-38 (1996).
\bibitem{P75} E. Pardoux: \emph{\'{E}quations aux d\'{e}riv\'{e}es partielles stochastiques de type monotone.} S\'{e}minaire sur les \'{E}quations aux D\'{e}riv\'{e}es Partielles (1974-1975), III, Exp. No. 2, 10 pp. Coll\`{e}ge de France, Paris, 1975.
\bibitem{PR} C. Pr\'{e}v\^{o}t, M. R\"{o}ckner: \emph{A concise course on stochastic partial differential equations.}
Lecture Notes in Mathematics, 1905. Springer, Berlin, 2007.

\bibitem{RRW} J. Ren, M. R\"{o}ckner, F. Wang: \emph{
Stochastic generalized porous media and fast diffusion equations.}
J. Differential Equations 238, no. 1, 118-152 (2007).

\bibitem{RWW} M. R\"{o}ckner, F. Wang, L. Wu: \emph{
Large deviations for stochastic generalized porous media equations.}
Stochastic Process. Appl. 116, no. 12, 1677-1689 (2006).

\bibitem{V-W} G. Vallet, P. Wittbold: \emph{
On a stochastic first-order hyperbolic equation in a bounded domain}.
Infin. Dimens. Anal. Quantum Probab. Relat. Top. 12, no. 4, 613-651 (2009).

\bibitem{Z-19} R. Zhang: \emph{On the small time asymptotics of scalar stochastic conservation laws}.	Preprint. ArXiv:1907.03397.
\end{thebibliography}
\end{document}